\documentclass[11pt, oneside]{amsart}
\RequirePackage{amsmath}
\usepackage{amsthm}
\usepackage[utf8]{inputenc}
\usepackage{hyperref}

\usepackage{nomencl}
\makenomenclature

\usepackage{mathtools}
\usepackage{enumerate}
\usepackage{cite}
\usepackage[shortlabels]{enumitem}
\usepackage{tikz-cd}
\usepackage{stmaryrd}

\usepackage{latexsym}
\usepackage{amsmath,amssymb,amsthm,amscd}
\usepackage{color}
\usepackage[nolabel, final]{showlabels}
\title{Iwasawa cohomology of analytic $(\varphi_L,\Gamma_L)$-modules}
\author{Rustam Steingart}
\address{Ruprecht-Karls-Universität Heidelberg,
	Mathematisches Institut,Im Neuenheimer Feld 205, D-69120 Heidelberg}
\email{rsteingart@mathi.uni-heidelberg.de}



\theoremstyle{plain}
\newtheorem{thm}{Theorem}[section]
\newtheorem{lem}[thm]{Lemma}
\newtheorem{rem}[thm]{Remark}
\newtheorem{prop}[thm]{Proposition}
\newtheorem{cor}[thm]{Corollary}
\newtheorem{con}[thm]{Conjecture}
\newtheorem*{cor*}{Corollary}
\newtheorem*{con*}{Conjecture}
\newtheorem{introtheorem}{Theorem}
\theoremstyle{definition}
\newtheorem{defn}[thm]{Definition}

\newcommand{\NN}{\mathbb{N}}
\newcommand{\N}{\mathbb{N}}

\newcommand{\Z}{\mathbb{Z}}
\newcommand{\QQ}{\mathbb{Q}}
\newcommand{\Q}{\mathbb{Q}}

\newcommand{\CC}{\mathbb{C}}
\newcommand{\C}{\mathbb{C}}

\newcommand{\cR}{\mathcal{R}}

\newcommand{\Dfm}{\mathbf{Dfm}}

\newcommand{\id}{\operatorname{id}}
\newcommand{\res}{\operatorname{res}}

\DeclarePairedDelimiter\abs{\lvert}{\rvert}%
\DeclarePairedDelimiter\norm{\lVert}{\rVert}%
\begin{document}
	\maketitle
\begin{abstract}
We show that the coadmissibility of the Iwasawa cohomology of an $L$-analytic Lubin--Tate $(\varphi_L,\Gamma_L)$-module $M$ is necessary and sufficient for the existence of a comparison isomorphism between the former and the analytic cohomology of its Lubin--Tate deformation, which, roughly speaking, is given by the base change of $M$ to the algebra of $L$-analytic distributions. 
We verify that coadmissibility is satisfied in the trianguline case and show that it can be ``propagated'' to a reasonably large class of modules, provided it can be proven in the étale case.
\end{abstract}
	\section*{Introduction}
	Iwasawa theory is classicaly concerned with the behaviour of $G_{\Q_p}$-modules along $\Z_p^{\times}$-towers such as the cyclotomic extension $\Q_{p}(\zeta_{p^{\infty}})$ with Galois group $\Gamma= \operatorname{Gal}(\Q_{p}(\zeta_{p^{\infty}})/\QQ_p) \cong \Z_p^{\times}.$ More generally one can consider, for $L/\QQ_p$ finite, a Lubin--Tate extension $L_\infty/L$ with Galois group $\Gamma_L$ and Frobenius power series $\varphi_L$ (for some fixed uniformiser $\pi_L$ of $L$). Denote by $q$ the cardinality of $o_L/\pi_Lo_L.$ Let $V$ be an $o_L$-linear representation of $G_L,$ i.e., a finite generated $o_L$-module with a continuous linear action of the absolute Galois group $G_L$ of $L.$ It is shown in \cite{SV15}, that the generalised Iwasawa cohomology $$H^i_{\mathrm{Iw}}(L_{\infty}/L,V):= \varprojlim_{L\subset F \subset L_{\infty}} H^i(F,V),$$ where $F$ runs through the finite extensions of $L$ and the transition maps are given by the corestriction, can be computed in terms of the Lubin--Tate $(\varphi_L,\Gamma_L)$-module $\mathbb{D}(V)$\footnote{See Theorem \ref{thm:fontaineequiv} for a precise definition of $\mathbb{D}(V).$} attached to $V$ by the complex $$\mathbb{D}(V(\tau))\xrightarrow{\psi_{\mathrm{LT}}-1}\mathbb{D}(V(\tau))$$ with a certain character $\tau \colon \Gamma_L \to o_L^{\times}$ using the ``integral'' $\psi_{\mathrm{LT}}$-operator (satisfying $\psi_{\mathrm{LT}}\circ \varphi_L= \frac{q}{\pi_L}$). By a variant of Shapiro's Lemma for the Iwasawa algebra $\Lambda = o_L \llbracket \Gamma_L\rrbracket$ one has $H^i_{\mathrm{Iw}}(L_\infty/L,V) \cong H^i(G_L, \Lambda\otimes_{o_L} V)$ relating the Iwasawa cohomology to the Galois cohomology of a $\Lambda$-linear representation, which one should think of as a family of Galois representations parametrised by $\operatorname{Spec}(\Lambda).$ In the following let $K/L$ be a complete field extension containing a period $\Omega_L$ of the Lubin--Tate group in the sense of \cite{schneider2001p}. Over such a field $K,$ the algebra of $L$-analytic distributions $D(o_L,K)$ of $o_L$ with values in $K$ is isomorphic to the ring of power series converging on the open unit disc by Fourier theory and the Lubin--Tate isomorphism of Schneider and Teitelbaum. We consider $(\varphi_L,\Gamma_L)$-modules over the Robba ring $\cR_K$ with coefficients in $K,$ whose $\Gamma_L$-action is $L$-analytic in the sense that the action of $\operatorname{Lie}(\Gamma_L)$ is $L$-bilinear (cf. Definition \ref{def:phigammamodule}). Such an action extends by continuity to an action of $D(\Gamma_L,K)$ and if we fix an open subgroup $U \subset\Gamma_L$ isomorphic to $o_L$ we obtain an element $Z\in D(\Gamma_L,K)$ such that any element of $D(U,K)$ is a convergent power series in $Z.$ Using the left-inverse $\Psi$ of $\varphi_L$ one can define a natural analogue of the Iwasawa cohomology for an $L$-analytic $(\varphi_L,\Gamma_L)$-module $M$ over the Robba ring in the form of the complex $$C_{\Psi}(M) = M \xrightarrow{\Psi-1}M$$ 
	
	concentrated in degree $1,2.$ We define the Iwasawa cohomology $H^i_{\mathrm{Iw}}(M)$ of $M$ as the cohomology of the above complex.  
	At this point a word of caution is in order. The operator $\Psi$ differs from $\psi_{\mathrm{LT}}$ used in \cite{SV15} by the constant $c=\frac{\pi_L}{q}.$ For our purposes it is convenient to work with the left-inverse $\Psi.$ It turns out, that $C_{\Psi}(M)$ is quasi-isomorphic to $C_{\psi_{\mathrm{LT}}}(M(\rho))$ for a suitable twist of $M$ (see Lemma \ref{lem:constantirrelevantpsi}). 
	Even though we are working over a field, the study of Iwasawa cohomology leads us to studying a non-trivial family of $(\varphi_L,\Gamma_L)$-modules, which we call the Lubin--Tate deformation of $M$ (see Definition \ref{def:dfm} for more details). The algebra $D(\Gamma_L,K)$ is isomorphic to the algebra of global sections of a rigid analytic space $\mathfrak{X}_{\Gamma_L},$  whose $K'$ points for $K'/K$ finite  are in bijection with $L$-analytic $K'$-valued characters of $\Gamma_L$ (cf. \cite{schneider2001p}).
 Roughly speaking one would like to take the completed tensor product $D(\Gamma_L,K)\hat{\otimes}_K M$ and as such define a family of $(\varphi_L,\Gamma_L)$-modules parametrised by the rigid analytic space $\mathfrak{X}_{\Gamma_L}$ which is a natural analogue of  $ \Lambda \otimes_{o_L}V$ in our situation. From the explicit description of the points of $\mathfrak{X}_{\Gamma_L}$ we get that the Lubin--Tate deformation of $M$ parametrises all twists of $M$ by $L$-analytic characters. 
	 Some care is required because $D(\Gamma_L,K)$ is no longer affinoid. In order to work within the framework of families as above we have to work over an affinoid cover leading us to a derived limit of the corresponding Herr complexes. 
In more precise terms, we can write $D(U,K)$ as a projective limit of affinoid algebras $D_n$ and the deformation (on the level of $U$) is defined as a ``sheaf'' $(\mathbf{Dfm}_n(M))_{n \in \N},$ where each term is given as $\mathbf{Dfm}_n(M) = D_n \hat{\otimes}_KM$ and $\gamma \in U$ acts as $\delta_{\gamma^{-1}} \otimes\gamma.$  Analytic cohomology serves as a replacement for Galois cohomology in the case considered by us. This cohomology theory was studied in \cite{RustamFiniteness} by means of an explicit generalised Herr complex $C_{\Psi,Z}(\mathbf{Dfm}_n(M))$ given as the total complex of the double complex $$\left(\begin{tikzcd} \mathbf{Dfm}_n(M) \arrow{d}{Z} \arrow{r}{\Psi-1}& \mathbf{Dfm}_n(M)\arrow{d}{-Z} \\	 \mathbf{Dfm}_n(M) \arrow{r}{\Psi-1} &\mathbf{Dfm}_n(M)\end{tikzcd}\right).$$
	The analytic cohomology of the deformation can be related to the Iwasawa cohomology of $M$ via the following theorem.
	\begin{introtheorem}[cf.\ Theorem \ref{thm:iwasawadfm}] Suppose $C_{\Psi}(M)$ has coadmissible cohomology groups (in the sense of \cite{schneider2003algebras}). Then there is a canonical\footnote{See Theorem \ref{thm:iwasawadfm} for a precise description of the maps. } isomorphism in the derived category $\mathbf{D}(D(U,K))$ of $D(U,K)$-modules $$\mathbf{Rlim}C_{\Psi,Z}(\Dfm_n(M)) \simeq C_{\Psi}(M)$$
		and, 	in particular, we have isomorphisms $$\varprojlim_{n} H_{\Psi,Z}^i(\Dfm_n(M))\cong H^i_{\mathrm{Iw}}(M).$$
		
	\end{introtheorem}
	
	This makes it clear that the coadmissibility of the Iwasawa cohomology of $M$ is a desirable property and we show that a sufficient condition is that $M^{\Psi=1}$ is finitely generated as a $D(U,K)$-module. 
	We proceed to prove that this condition is satisfied by modules of rank one, which are of the form $\cR_K(\delta)$ for a locally $L$-analytic character $\delta\colon L^{\times}\to K^{\times}.$ By a dévissage argument this carries over to trianguline modules, i.e., successive extensions of such $\cR_K(\delta).$ This leads us to the following theorem.
	\begin{introtheorem}(Theorem \ref{thm:perfectTrianguline})
		Let $M$ be a trianguline $L$-analytic $(\varphi_L,\Gamma_L)$-module over $\cR_K.$ Then $C_{c\Psi}(M)$ is a perfect complex of $D(\Gamma_L,K)$-modules for any constant $c \in K^{\times}.$
	\end{introtheorem}
	
	From the work of Schneider and Venjakob one obtains that the Iwasawa cohomology $H^i_{\mathrm{Iw}}(\Gamma_L,V)$  of an $o_L$-linear representation $V$ of $G_L$ is computed by the complex $$\mathbb{D}(V(\tau)) \xrightarrow{\psi_{\mathrm{LT}}-1} \mathbb{D}(V(\tau)),$$ with some character $\tau.$ Applying this to $V(\tau^{-1})$  instead of $V$ we obtain, in particular, that  $\mathbb{D}(V)^{\psi_{\mathrm{LT}}=1}$ is a finite $\Lambda= o_L\llbracket \Gamma_L \rrbracket $-module. If one assumes further that $V$ is $L$-analytic, one can show that $\mathbb{D}(V)^{\psi_{\mathrm{LT}}=1}$ is contained in the overconvergent submodule $\mathbb{D}^{\dagger}(V)$ providing us with a natural map $$D(\Gamma_L,K) \otimes_{\Lambda}\mathbb{D}^{\dagger}(V)^{\psi_{\mathrm{LT}}=1}  \to M^{\psi_{\mathrm{LT}}=1},$$ where $M$ denotes the completed base change to $K$ of the $(\varphi_L,\Gamma_L)$-module over $\cR_L$ attached to $V.$
	We conjecture that this map is surjective.
	\begin{con*} (Conjecture \ref{con:comparison})
		The natural map $$D(\Gamma_L,K) \otimes_{\Lambda}\mathbb{D}^{\dagger}(V)^{\psi_{\mathrm{LT}}=1}  \to M^{\psi_{\mathrm{LT}}=1}$$ is surjective.
	\end{con*}
We explain in section \ref{sec:etalecomp} that this map is not bijective unless $L= \QQ_p$.
Finally we show how the étale results of Schneider and Venjakob can be implemented to show general perfectness statements for those modules that arise as a base change of some $M_0$ over $\cR_L$ under the assumption of this conjecture. 

Note that Conjecture \ref{con:comparison} asserts a priori that $M^{\psi_{\mathrm{LT}}=1}$ is finitely generated over $D(\Gamma_L,K)$ as a quotient of the base change of $\mathbb{D}(V)^{\psi_{\mathrm{LT}}=1},$ which is $\Lambda$-finitely generated (cf.\cite{SV15}).  As mentioned before, we can compare $M^{\psi_{\mathrm{LT}}=1}$ to $H^1_{\mathrm{Iw}}(M)$ by twisting with a suitable character. But this character is not étale unless $L=\QQ_p.$

	The restriction to modules coming from $\cR_L$ is rooted in the methodology of the proof, which is adapted from \cite{KPX}. The idea is that (assuming Conjecture \ref{con:comparison}) the finite generation of $M^{{\psi_{\mathrm{LT}}}=1}$ and hence the perfectness of $C_{\psi_{\mathrm{LT}}}(M)$ (cf. Proposition \ref{prop:psiperfchar}) holds for a certain class of $(\varphi_L,\Gamma_L)$-modules, namely the étale modules studied by Schneider and Venjakob in \cite{SV15}. An induction over Harder--Narasimhan slopes of $(\varphi_L,\Gamma_L)$-modules ``propagates'' the result to all $L$-analytic $(\varphi_L,\Gamma_L)$-modules over $\cR_L.$ By introducing the large extension $K/L$ we leave the setting of Kedlayas slope filtrations due to no longer working with Bézout rings and our proofs make heavy use of the structure of $D(U,K)$ as a power series ring. For that reason we are uncertain whether the results can be descended to $M_0$ (as a $D(\Gamma_L,L)$-module) and whether they hold for any $M$ (not necessarily arising as a base change from $\cR_L$ or say $\cR_F$ for some finite extension $F/L$). 
	
Note that once we establish perfectness of $C_{\psi_{\mathrm{LT}}}(M)$ for (not nescessarily étale) modules of the above form, we can perform a twist by a suitable character to conclude perfectness of $C_{c\Psi}(M)$ for every constant $c \in L^\times.$
	\begin{introtheorem}[Theorem \ref{thm:psiperfect}, Corollary \ref{cor:constanttheorem}]
		Assume Conjecture \ref{con:comparison}.
		Let $M_0$ be an $L$-analytic $(\varphi_L,\Gamma_L)$-module over $\cR_L$ and let $M:= K\hat{\otimes}_{L}M_0$. Then the complex $C_{c\Psi}(M)$ of $D(\Gamma_L,K)$-modules is perfect for every $c \in L^\times.$
	\end{introtheorem}
	
	We go on to discuss the Euler--Poincaré characteristic formula for $C_{\Psi,Z}(M)$ and prove that the expected Euler--Poincaré characteristic formula $$\sum_{i \in \N_0} (-1)^i\operatorname{dim}_KH^i_{\Psi,Z}(M) = -[\Gamma_L:U]\operatorname{rank}_{\cR_L}(M)$$ holds in the trianguline case (see Remark \ref{rem:epctriangulin}). We also show that the formula in the general case would follow from $\mathcal{C}(M)= (1-\varphi_L)M^{\Psi=1}$ being projective over $D(U,K)$ of rank $[\Gamma_L:U]\operatorname{rank}_{\cR_K}(M).$
	We formulated our Conjecture \ref{con:comparison} in order to get a streamlined treatment of the perfectness of $C_{\Psi}(M).$ It would be of independent interest to describe the kernel of the comparison map and study the comparison map for the co-kernels of $\psi_{\mathrm{LT}}-1$ in the future.

\section*{Acknowledgements}
This article is based on parts of the author's Ph.D.\ thesis (cf. \cite{rusti}) and we thank Otmar Venjakob for his guidance and many helpful discussions. We also thank Jan Kohlhaase, David Loeffler, and Muhammad Manji for their comments on \cite{rusti}.   This research is funded by the Deutsche Forschungsgemeinschaft (DFG, German Research Foundation) TRR 326 \textit{Geometry and Arithmetic of Uniformized Structures}, project number 444845124.

	\section{Preliminaries}
\subsection{Robba rings and their $(\varphi_L,\Gamma_L)$-actions.}
Let $\varphi_L(T)$ be a Frobenius power series for some uniformiser $\pi_L$ of $L.$ We denote by $L_n$ the extension of $L$ that arises by adjoining the $\pi_L^n$-torsion points of the Lubin--Tate formal group attached to $\varphi_L$ to $L.$ Let $\Gamma_L= \operatorname{Gal}(L_\infty/L)\cong o_L^\times,$ where $L_\infty = \bigcup_{n\geq 0} L_n.$ Since $\varphi_L$ and $\Gamma_L$ induce endomorphisms of the LT group we obtain corresponding actions on $o_L\llbracket T\rrbracket.$

\begin{defn}
	Let $K \subset \CC_p$ be a complete field.
	We denote by $\mathcal{R}_K^{[r,s]}$ the ring of Laurent series (resp. power series if $r=0$) with coefficients in $K$ that converge on the annulus $r\leq \lvert x\rvert \leq s$ for $r,s \in p^{\QQ}$ and $x \in \CC_p.$
	It is a Banach algebra with respect to the norm $\lvert \cdot  \vert_{[r,s]} := \max (\lvert \cdot \rvert_r, \lvert \cdot \rvert_s).$ We further define the Fréchet algebra $\mathcal{R}_K^r:=\mathcal{R}_K^{[r,1)} := \varprojlim_{r<s<1} \mathcal{R}_K^{[r,s]}$ and finally the \textbf{Robba ring} $\mathcal{R}_K:= \varinjlim_{0\leq r<1} \mathcal{R}_K^{[r,1)}$ endowed with the locally convex inductive limit topology, i.e., the finest topology making $\cR_K$ locally convex such that $\cR_K^{[r,1)} \to \cR_K$ is continuous for every $r \in [0,1).$ We set $\cR_K^+:= \cR_K^{[0,1)}.$
	For an affinoid Algebra $A$ over $K,$ i.e., a quotient of a Tate-algebra in finitely many variables over $K,$ and $? \in \{+,[r,s],[r,1)\}$ we define $\cR_A^? := A\hat{\otimes}_K\cR_K^{?},$ where $A\hat{\otimes}_K\cR_K^{?}$ denotes the completion with respect to the projective tensor product topology. Lastly, we define the \textbf{relative Robba ring} $\cR_A := \varinjlim_r  A\hat{\otimes}_K\cR_K^{[r,1)}.$
\end{defn}
We obtain similar continuous actions of $\Gamma_L$ on $\cR_L^I$ for any interval $I=[r,s]\subset [0,1].$ For details concerning the $\varphi_L$-action we refer to \cite[Section 2.2]{Berger}. To ensure that the action of $\varphi_L$ on $\cR_L^I$ is well-defined
one has to assume that the lower boundary $r$ of $I$ is either $0$ or $r > \abs{u}^q=:r_L$ for a non-trivial $\pi_L$-torsion point $u$ of the Lubin--Tate group \footnote{In the second case the assumptions guarantee $\varphi(T)\in (\cR_L^{I^{1/q}})^{\times}.$}.
When $\varphi_L$ acts on $\cR_L^I$ it changes the radius of convergence and we obtain a morphism 
$$\varphi_L: \cR_L^{I}\to \cR_L^{I^{1/q}} .$$ We implicitly assume that $r,s$ lie in $\abs{\overline{\QQ_p}},$ because in this case the algebra $\cR_L^I$ is affinoid (cf. \cite[Example 1.3.2]{lutkebohmert2016rigid}) and this assumption is no restriction when considering the Robba ring due to cofinality considerations.
We henceforth endow the rings 
$\cR_L^{[r,1)}$ (for $r=0$ or $r>r_L$) and $\cR_L$ with the $(\varphi_L,\Gamma_L)$-actions induced by the actions on $\cR_L^I.$ 
We extend the actions to $\cR_K$ or more generally $\cR_A$ by letting $\Gamma_L$ and $\varphi_L$ act trivially on the coefficients. 

\begin{defn} \label{def:phigammamodule} Let $r_L<r_0.$ A (projective) \textbf{$\varphi_L$-module} over $\mathcal{R}_A^{r_0}=\mathcal{R}_A^{[r_0,1)}$ is a finite (projective)  $\mathcal{R}^{r_0}_A$-module $M^{r_0}$ equipped with an isomorphism $\varphi_L^*M^{r_0} \cong M^{r_0^{1/q}}:= M^{r_0} \otimes_{\mathcal{R}_A^{r_0}}\mathcal{R}_A^{r_0^{1/q}}.$ A $\varphi_L$-module $M$ over $\mathcal{R}_A$ is defined to be the base change of a  $\varphi_L$-module $M^{r_0}$ over some $\mathcal{R}_A^{r_0}.$ We call $M^{r_0}$ a \textbf{model} of $M$ over $[r_0,1).$\\
	A \textbf{$(\varphi_L,\Gamma_L)$-module} over the above rings is a $\varphi_L$-module whose model is projective and endowed with a semilinear continuous action of $\Gamma_L$ that commutes with $\varphi_L.$ Here continuous means that for every $m \in M^{[r,s]} := M^{r_0} \otimes_{\mathcal{R}_A^{r_0}}\mathcal{R}_A^{[r,s]}$ the orbit map $\Gamma_L \to M^{[r,s]}$ is continuous for the profinite topology on the left side and the Banach topology on the right-hand side.\\
	A $(\varphi_L,\Gamma_L)$-module over $\cR_A$ is called \textbf{$L$-analytic} if the induced action of $\operatorname{Lie}(\Gamma_L)$ is $L$-bilinear. 
\end{defn}
\begin{rem} $\cR_A^{r_0}$ is a Fréchet-Stein Algebra in the sense of Schneider--Teitelbaum. 
A finite projective $\cR_A^{r_0}$-module $M^{r_0}$ is automatically coadmissible, i.e., the natural map $$M^{r_0} \to \varprojlim_s (\cR_A^{[r,s]} \otimes_{\cR_A^{r_0}} M^{r_0})$$ is an isomorphism.
\end{rem}
\begin{rem} The $\Gamma_L$ action on an $L$-analytic module $M$ over $\cR_K$ (resp. $\cR_A$) extends by continuity to an action of the algebra of $L$-analytic distributions $D(\Gamma_L,K)$ (resp. $A \hat{\otimes} D(\Gamma_L,K)$).
\end{rem}
A result of Schneider and Teitelbaum (cf. \cite{schneider2001p}) asserts that there exists an isomorphisms of topological algebras $D(o_L,K) \cong \cR_K^{[0,1)}$ provided that $K$ contains a certain period of the Lubin--Tate group. For the entirety of this article we assume that $K$ is a complete subfield of $\C_p$ containing such a period. For an open subgroup $U \subset \Gamma_L$ isomorphic to $o_L$ we can take a preimage $Z$ of the variable $T \in \cR_K^{[0,1)}$ under the isomorphism from \cite{schneider2001p}. As a result any distribution in $D(U,K)$ can be expressed as a convergent power series in $Z.$ This choice depends, of course, on the chart used to identify $U$ and $o_L.$ 
Let $\Gamma_0 =\Gamma_L$ and for $n \geq 1$ let  $\Gamma_n := \chi_{\mathrm{LT}}^{-1}(1+\pi_L^no_L).$ Let $n_0$ be minimal among $n \in \NN$ with the property that the series $\log$ and $\exp$ induce isomorphisms $1+\pi_L^n \cong \pi_L^n.$
For the remainder of the article we fix the following compatible choice of variables.
\begin{defn} \label{def:variable}
	Consider for $n\ge n_0$ the system of commutative diagrams 
	$$	\begin{tikzcd}
		\Gamma_n \arrow[r, "\chi_{\mathrm{LT}}"]                 & 1+\pi_L^n \arrow[r, "\log_p"]                 & \pi_L^no_L                 \\
		\Gamma_m \arrow[u, hook] \arrow[r, "\chi_{\mathrm{LT}}"] & 1+\pi_L^m \arrow[u, hook] \arrow[r, "\log_p"] & \pi_L^mo_L. \arrow[u, hook]
	\end{tikzcd}
	$$
	We fix the variable $Z_{m} \in D(\Gamma_m,K)$ for every $m\geq n_0$ to be the preimage of $T \in \cR_K^{[0,1)}$ under the isomorphism provided by \cite{schneider2001p}. Since $\pi_L^{n}o_L = \pi_L^{m-n} (\pi_L^no_L)$ we obtain the relationship
	\begin{equation}\label{eq:Znbeziehung} Z_m = \varphi_L^{m-n}(Z_n)
	\end{equation}
	 for every $m\geq n$ using \cite[Remark 1.2.7]{RustamFiniteness}.
\end{defn}

We recall some technical results concerning the action (cf. \cite[Lemma 1.5.3]{RustamFiniteness}).
\begin{lem}\label{lem:operatornormestimates}
	Let $M$ be a finitely generated module over $\cR_A^{r}$ with an $L$-analytic semi-linear $\Gamma_L$-action. Fix any closed interval $I=[r,s]\subset[r,1)$ and any Banach norm on $M^I.$   
	\begin{enumerate}[(i)]
		\item We have $\lvert\lvert \gamma-1\rvert \rvert_{M^I}\xrightarrow{\gamma \to 1} 0.$
		\item Furthermore $\lvert\lvert Z_n\rvert \rvert_{M^I}\xrightarrow{n \to \infty} 0.$
	\end{enumerate}
\end{lem}
Fix some $n \geq n_0$ and let $Z:= Z_n.$
In \cite{RustamFiniteness} we showed that the analytic cohomology of the monoid $\varphi_L\times U$ with coefficients in $M$ can be computed by the generalised Herr complex 

$$C_{\varphi_L,Z}(M):=\mathrm{Tot}\left(\begin{tikzcd} M \arrow{d}{Z} \arrow{r}{\varphi_L-1}& M\arrow{d}{-Z} \\	 M \arrow{r}{\varphi_L-1} &M\end{tikzcd}\right).$$

In this article we will only be working with $C_{\varphi,Z}(-)$ or the analogously defined $\Psi$-version $C_{\Psi,Z}(-).$ As in \cite{RustamFiniteness} we denote by $\Psi$ the canonical left-inverse of $\varphi_L$ and $\psi_{\mathrm{LT}}:= \frac{q}{\pi_L}\Psi$. Recall that $C_{\Psi,Z}(M) \simeq C_{\varphi,Z}(M)$ by \cite[Remark 3.2.2]{RustamFiniteness}.
\begin{thm}
	\label{thm:perfect}
	Let $A,B$ be $K$-affinoid and let $M$ be an $L$-analytic $(\varphi_L,\Gamma_L)$-module over $\cR_A.$ Let $f\colon A  \to B$ be a morphism of $K$-affinoid algebras. Then:
	\begin{enumerate}[(1)]
		\item $C_{\varphi_L,Z}(M) \in \mathbf{D}^{[0,2]}_{\text{perf}}(A).$
		\item The natural morphism $C_{\varphi_L,Z}(M) \otimes_{A}^\mathbb{L} B \to C_{\varphi_L,Z}(M \hat{\otimes}_A B)$ is a quasi-isomorphism.
	\end{enumerate}
\end{thm}
\begin{proof}
	See \cite[Theorem 3.3.12]{RustamFiniteness}.
\end{proof}
\subsection{Functional analysis}
\begin{defn}
	A homomorphism of topological groups
	$f: G \to H$ is called \textbf{strict}, if the induced map $G/\ker f \to \operatorname{Im} f$ is a topological isomorphism with respect to the quotient topology on the left and the subspace topology on the right. 
\end{defn} 
\begin{lem}\label{lem:strictcompletion} 
	Let $G,H$ be metrizable topological groups with completions $\hat{G}$ (resp. $\hat{H}$) and
	let $f: G \to H$ be strict. Then the unique extension $$\hat{f}:\hat{G} \to \hat{H} $$ of $f$ is strict with kernel $\ker(\hat{f})=\overline{\ker(f)}$ and image $\operatorname{Im}(\hat{f}) = \overline{(\operatorname{Im}f)}.$  
\end{lem}
\begin{proof}
	See \cite[Chapter IX §3.1 Corollary 2 p.164]{BTG}.
\end{proof}
\begin{lem}\label{lem:strictdensesubset}
	Let $G,H$ be Hausdorff abelian topological groups and $f: G \to H$ continuous such that $f$ is strict on a dense subgroup $D \subset G.$ Then $f$ is strict.
\end{lem}
\begin{proof} Since $H$ is assumed to be Hausdorff $V/\ker f$ and $\operatorname{Im}f$ are Hausdorff and we can without loss of generality assume that $f$ is a continuous bijection that is strict when restricted to a dense subgroup and are left to show that its inverse is continuous. The abelian assumption asserts that $G,H$ admit (Hausdorff) completions $\widehat{G},$ $\widehat{H}$ (cf. \cite[III §3.5 Theorem 2]{BTG}) and since $f: D \to f(D)$ is a topological isomorphism, the induced map $\widehat{f_{\mid D}} \colon\widehat{D} \to \widehat{f(D)}$ is a topological isomorphism. Since $D$ (resp. $f(D)$) is dense in $G$ (resp. $H$) we conclude that the corresponding completions agree with $\widehat{G}$ (resp. $\widehat{H}$). The restriction of the inverse $\widehat{f_{\mid D}}^{-1}$ to $H$ agrees with $f^{-1}$ by construction and is continuous.
\end{proof}
The following criterion will play a crucial role for checking strictness and the Hausdorff property of certain cohomology groups of $(\varphi_L,\Gamma_L)$-modules. The second part of Lemma \ref{lem:strictcheck} is essentially an adaptation of \cite[Lemma 22.2]{SchneiderNFA} (which only treats Fredholm operators on Fréchet spaces) (cf. also \cite[Proposition 4.1.39]{oli}).
\begin{lem}\label{lem:strictcheck}
	
	\begin{enumerate}
		\item Any continuous linear surjection $V \to W$ between LF-spaces over $K$ is open and in particular strict.
		\item Any continuous linear map $f\colon V \to W$ between Hausdorff LF-spaces with finite dimensional cokernel is strict. In addition the cokernel is Hausdorff.
		
	\end{enumerate}
\end{lem}
\begin{proof}
	1.) See \cite[Proposition 8.8]{SchneiderNFA}. 2.) 
	Let $X \subset W$ be a finite dimensional subspace such that (algebraically) $W = \operatorname{im}(f)\oplus X.$ Since $W$ is Hausdorff, the subspace $X$ carries its natural norm-topology. By assumption $V/\operatorname{ker}f$ is Hausdorff and thus by \cite[Theorem 1.1.17]{emerton2017locally} $V/\operatorname{ker}f$ is itself an LF-space and we have a continuous bijection
	$$h:V/\operatorname{ker}f \oplus X \to W$$ of LF-spaces, which by the open mapping theorem is a homeomorphism.
	By construction $f\colon V/\ker f \to W$ factors via $V/\ker f \to V/\operatorname{ker}f \oplus X  \to W$ and is thus a homeomorphism onto its image.
	Furthermore $\operatorname{im}(f) = \operatorname{ker}(p_2 \circ h^{-1}:W \to X)$ is the kernel of a continuous map into a Hausdorff space and thus closed, which implies that the cokernel is Hausdorff.

\end{proof}
The following Lemma is implicitly contained in the proof of \cite[Proposition 2.1.23]{emerton2017locally}. 
\begin{lem}
	\label{lem:onablebasechange}
	Let $V',V,V''$ be $K$-Fréchet spaces that fit into a strict exact sequence $$0 \to V' \to V \to V'' \to 0.$$
	Let $F$ be  a $K$-Banach space of countable type over $K,$ then the induced sequence 
	$$0 \to F \hat{\otimes}_KV' \to F \hat{\otimes}_KV \to F \hat{\otimes}_KV'' \to 0$$ is strict exact.
\end{lem}
\begin{proof} By \cite[Corollary 2.3.9]{PGS} any infinite dimensional Banach space of countable type is isomorphic to the space of zero sequences $c_0(K)$ and observe that $c_0(K)$ can be identified with the completion of $\oplus_{n \in \N} K$ with respect to the $\sup$-norm. Without loss of generality assume $F$ is infinite dimensional and take the isomorphism $F \cong c_o(K)$ as an identification. Similarly $c_0(K) \hat{\otimes} V$ is functorially isomorphic to the space of zero-sequences in $V,$ which we denote by $c_0(V).$ Indeed write $V = \varprojlim V_n$ with Banach spaces $V_n$ and dense transition maps. Then by \cite[Lemma 2.1.4]{Berger} we have $$V \hat{\otimes} c_o(K) \cong \varprojlim V_n \hat{\otimes}_Kc_o(K).$$ Because a zero-sequence in $V$ is precisely a compatible family of zero sequences in each $V_n,$ it suffices to prove the statement for the Banach spaces $V_n,$ which is clear. 
	It remains to see that the sequence $$0 \to c_0(V') \to c_0(V) \to c_0(V'') \to 0$$ is strict exact. 
	One checks that it is algebraically exact and each map is continuous. Because the spaces in question are Fréchet spaces, we can conclude from the open-mapping theorem, that the induced sequence is strict exact. 

\end{proof}

\subsection{Duality}

 Let $f= \sum_{i}a_iT^i \in \cR_A$ we define the $\operatorname{res}(f):= a_{-1}\in A.$ We obtain an $A$-linear map $\cR_A \xrightarrow{\operatorname{res}} A.$
The following result is well-covered in the literature when $A$ is a (discretely valued) field (see for instance \cite[Chapter 5]{crew1998finiteness}). In the case of an affinoid $A$ over a not necessarily discretely valued extension $K/L$ some topological subtleties arise, that we treat by viewing $A$ as a (no longer affinoid) Banach Algebra over $L.$
\begin{prop}
	\label{prop:pairingoverA}
	Consider the bilinear map $\cR_A \times \cR_A \to A,$ mapping $(f,g)$ to $\operatorname{res}(fg).$ This induces topological isomorphisms (for the strong topologies\footnote{By the strong topology we mean the subspace topology of the space of continuous $K$-linear operators $\mathcal{L}_{K,b}(\cR_A,A)$ with the strong topology.}) $$\operatorname{Hom}_{A,cts}(\cR_A,A)\cong \cR_A,$$ $$\operatorname{Hom}_{A,cts}(\cR_A^+,A)\cong \cR_A/\cR_A^+$$ and $$\operatorname{Hom}_{A,cts}(\cR_A/\cR_A^+,A)\cong \cR_A^+.$$
\end{prop}
\begin{proof} Let $\mu \in \operatorname{Hom}_{A,cts}(\cR_A,A).$ We define a Laurent series given by  $f_{\mu}:= \sum_{n\in \mathbb{Z}} \mu(T^{-1-n})T^n.$ On the other hand to $h \in \cR_A$ we assign the functional $\operatorname{can}(h): g \mapsto \operatorname{res}(hg).$ 
One can check that these assignments are well-defined.
	Futhermore $f_\mu$ belongs to $\cR^+_A$ if and only if the coefficients of the principal part (i.e. $\mu(T^{n})$ for $n \geq 0$) vanish. This is the case if and only if $\mu(\cR^+_A) = 0.$
	Regarding the topologies we sketch how to deduce the general case from the case $A=L$ (treated by \cite{crew1998finiteness}) by showing that for $E \in \{\cR_L,\cR_L/\cR_L^+,\cR_L^+\}$ we have a canonical isomorphism $\operatorname{Hom}_{A,cts}(A\hat{\otimes}_{L}E,A) \cong \operatorname{Hom}_{L,cts}(E,A)=\mathcal{L}_b(E,A),$ which allows us  to deduce the general statement from the known case using $E\hat{\otimes}_{L}A \cong \mathcal{L}_b(E',A)$ (cf.\cite[Corollary 18.8]{SchneiderNFA}).\footnote{Alternatively one can redo the proof with $A$ replacing $L.$}
	We define the maps as follows:
	Let $f: A\hat{\otimes}_{L}E \to A$ be a continuous $A$-linear homomorphism. Set $\tilde{f}: E  \to A\otimes_{L}E \to A\hat{\otimes}_{L}E\to  A.$ Given by mapping $e$ to $1 \otimes e$ and post-composing with the natural map. By \cite[Theorem 10.3.9]{PGS} since $\norm{1}_A=1$ the first map is a homeomorphic embedding (in particular continuous). On the other hand let $h: E \to A$ be continuous $L$-linear then $h$ extends uniquely to an $A$-linear map $A \otimes_L E \to A$ and it remains to check that it is continuous. By \cite[Remark 1.1.13]{RustamFiniteness} it suffices to check separate continuity, which is clear. Due to $A$ being complete $h$ extends uniquely to a continuous $A$-linear map $h_A: A\hat{\otimes}_{L}E \to A.$
	Clearly $f \to \tilde{f}$ and $h \to h_A$ are inverse to one another. 
	It remains to check continuity with respect to the corresponding strong topologies. For that purpose we denote by $A_0$ the unit ball inside $A.$ Let $B \subset A \hat{\otimes} E$ be a bounded set and suppose $f(B) \subset A_0$, then (again using \cite[Theorem 10.3.9]{PGS}) the preimage $\tilde{B}$ of $B \cap 1 \otimes E$ in $E$ is bounded and by construction $\tilde{f}(\tilde{B}) \subset A_0.$ On the other hand let $B' \subset E$ be bounded and suppose $h(B') \subset A_0.$ Then the closure $B_A$ of $\operatorname{span}_{o_L}\{x \otimes y \mid x \in A_0, y \in B\}$ is bounded in $A \hat{\otimes}_{L}E$ and by construction $h_A(B_A) \subset A_0.$
\end{proof}
\begin{prop}
	\label{prop:pairing}
	Let $M$ be a $(\varphi_L,\Gamma_L)$-module over $\cR_A$ and consider the dual $\check{M}:= \operatorname{Hom}_{\cR_L}(M,\cR_A(\chi_{\mathrm{LT}})).$ The canonical pairing
	$$\check{M} \times M \to \cR_A(\chi_{\mathrm{LT}})$$
	is perfect and by post-composing with the residue map gives a bilinear pairing
	$$\check{M} \times M \xrightarrow{\langle \cdot ,\cdot \rangle} A$$
	identifying $\check{M}$ (resp. $M$) with $\operatorname{Hom}_{A,\text{cts}}(\check{M},A)$ ($\operatorname{Hom}_{A,\text{cts}}(M,A)$) with respect to the strong topology,
	satisfying 
	\begin{enumerate}
		\item $\langle \varphi_L(\check{m}),\varphi_L(m)\rangle = \frac{q}{\pi_L}\langle \check{m},m\rangle$
		\item $\langle\sigma \check{m},\sigma m\rangle =\langle \check{m}, m\rangle $ 
		\item $\langle \psi_{\mathrm{LT}}(\check{m}),m\rangle = \langle \check{m},\varphi_L(m)\rangle$
	\end{enumerate} 
	For all $\sigma \in \Gamma_L,$ $\check{m} \in \check{M}$ and $m \in M.$
\end{prop}
\begin{proof} The perfectness of the first pairing is well-known since $M$ is finitely generated and projective. The properties of the pairing are proved in \cite[Section 3]{SV15} for the ring $\widehat{o_L ((T))}^{p-\text{adic}}$ and their proofs carry over to this case.
	By writing $M$ as a direct summand of a finitely generated  free module  the identification of duals follows from the free case by induction over the rank from Proposition \ref{prop:pairingoverA}. 
\end{proof}
\subsection{Sen theory} Let $t_{\mathrm{LT}}:=\log_{\mathrm{LT}}(T) \in \cR_L^+$ be the logarithm of the Lubin--Tate formal group attached to $\varphi_L$ and let $L_n$ be the field obtained by adjoining the $\pi_L^n$-torsion points of the formal group to $L.$  We do not give a conceptual treatment of Sen theory and instead present some ad-hoc results that allow us to view a $(\varphi_L,\Gamma_L)$-module $M^r$ over $\cR_K^r$ as a $\Gamma_L$-submodule of a finite projective $(L_n\otimes K)\llbracket t_{\mathrm{LT}} \rrbracket$-module $D_{\text{dif},n}^+(M)$ for a suitable $n \in \N,$ which is technically useful because the latter module is a projective limit of finite-$K$-dimensional $\Gamma_L$-representations namely $D_{\text{dif},n}^+(M)/(t_{\mathrm{LT}}^k).$ The name ``Sen theory'' stems from the fact that these modules are the analogues of their counterparts in classical Sen theory (more precisely its extension to $B_{\mathrm{dR}}$-representations) for $(\varphi_L,\Gamma_L)$-modules which arise from Fontaine's equivalence of categories. For $n \in \N$ let $K_n := K \otimes_L L_n$ and fix a non-trivial compatible system $u_n$ of $\pi_L^n$-torsion points of the Lubin--Tate group. We endow $K_n$ with its canonical $K$-Banach space topology and endow $K_n \llbracket t_{\mathrm{LT}} \rrbracket = \varprojlim_k K_n \llbracket t_{\mathrm{LT}} \rrbracket /(t_{\mathrm{LT}}^k)$ with the projective limit topology of the canonical topologies on each term. Since $t_{\mathrm{LT}}=\log_{\mathrm{LT}}(T)$ has no constant term and non-vanishing derivative in $0$   the induced maps $$K_n\llbracket t_{\mathrm{LT}} \rrbracket/(t_{\mathrm{LT}}^k) \to K_n\llbracket T \rrbracket/(T^k)$$ are isomorphisms of finite dimensional $K$-vector spaces and hence topological for the respective canonical topologies. We obtain that the natural map $$K_n\llbracket t_{\mathrm{LT}} \rrbracket \to K_n \llbracket T \rrbracket$$ is an isomorphism if we endow the right-hand side with its weak topology. The definition for $\iota_n$ below is taken from \cite[Section 1.4.2]{colmez2016representations}.
\begin{lem}\label{lem:iotainjective} 
	Let $n \in \N$ and let $r^{(n)}= \abs{u_n}$ then
	
	\begin{align*}\iota_n: \cR_K^{[r^{(n)},1)} &\to K_n\llbracket t_{\mathrm{LT}}\rrbracket\\
		T &\mapsto \iota_n(T), 
	\end{align*}
	
	given by $\iota_n(T):= u_n +_{\mathrm{LT}} \exp_{\mathrm{LT}}(\pi_L^{-n}\log_{\mathrm{LT}}(T)),$
	is well-defined, injective and $\Gamma_L$-equivariant, where $\Gamma_L$ acts on the right-hand side via the trivial action on $K,$ the Galois action on $L_n$ and the usual action (via $\chi_{\mathrm{LT}}$) on $t_{\mathrm{LT}}.$
\end{lem}
\begin{proof}
	The convergence, injectivity and $\Gamma_L$-equivariance of $\iota_n$ in the cyclotomic case with $K=L=\Q_p$ is known (see \cite[Proposition 2.25 and the remark before 2.35]{berger2002DGL}) and can be analogously proved for the Robba ring $\cR_L$ over $L.$ The map over $\cR_K^{[r^{(n)},1)}$ arises by applying $K\hat{\otimes}_{L,\pi}$ to the version over the spherically complete $L.$
	We apply \cite[1.1.26]{emerton2017locally} to conclude that the induced map remains injective using that $K$ as an $L$-Banach space is automatically bornological. The compatibility with the actions is preserved, since we let $\Gamma_L$ act trivially on $K.$
\end{proof}
\begin{defn}
	Let $M$ be a $(\varphi_L,\Gamma_L)$-module over $\cR_K$ with model $M^{r^{(n)}}$ over $[r^{(n)},1)$ with $r^{(n)}$ as in Lemma \ref{lem:iotainjective}. 
	We define $$\mathbb{D}^+_{\mathrm{dif},n}(M):= K_n\llbracket t_{\mathrm{LT}}\rrbracket \otimes_{\iota_n,\cR_K^{[r{(n)},1)}} M^{r^{(n)}}.$$
\end{defn}
\begin{rem}
	\label{rem:sentheoryembedding}
	Let $M$ be a $(\varphi_L,\Gamma_L)$-module over $\cR_K$ with model $M^{r^{(n)}}$ over $[r^{(n)},1)$ with $r^{(n)}$ as in Lemma \ref{lem:iotainjective}.  Then the natural map 
	$$M^{r^{(n)}} \to\mathbb{D}^+_{\mathrm{dif},n}(M)$$ is injective and $\Gamma_L$-equivariant.	
\end{rem}
\begin{proof}
	Using that $M^{r^{(n)}}$ is projective the statement follows by tensoring the injective and $\Gamma_L$-equivariant map $\iota_n$ with $M^{r^{(n)}}.$ 
\end{proof}
\subsection{Étale $(\varphi_L,\Gamma_L)$-modules}
In this section we give a brief overview on étale modules over the $p$-adic completion  $\mathbf{A}_L$ of $o_L((T)).$ We denote by $\operatorname{Rep}_{o_L}(G_L)$ the category of finitely generated $o_L$-modules with $o_L$-linear continuous (with respect to the $p$-adic topology) $G_L$-action. Similary let $\operatorname{Rep}_L(G_L)$ be the category of finite dimensional $L$-vector space with continuous $L$-linear $G_L$-action. We denote by $\C_p^{\flat}$ the tilt of $\C_p$ and for an $o_L$-algebra $R$ we write $W(R)_L$ for the ring of ramified Witt vectors.
\begin{defn}
	A $(\varphi_L,\Gamma_L)$-module $D$ over $\mathbf{A}_L$ is a finitely generated $\mathbf{A}_L$-module with a $\varphi_L$-semi-linear map $\varphi_D$ and a semi-linear action of $\Gamma_L$ commuting with $\varphi_D,$ such that $\Gamma_L$ acts continuously with respect to the weak topology. $D$ is called \textbf{étale} if the linearised map $\mathbf{A}_L \otimes_{\mathbf{A}_L,\varphi_L} D \to D,$ mapping $a \otimes d$ to $a \varphi_D(d)$ is an isomorphism.
	Analogously we define $(\varphi_L,\Gamma_L)$-modules over $\mathbf{B}_L:= \mathbf{A}_L[1/p].$ Such a module is called \textbf{étale} if it arises as a base change of an étale $(\varphi_L,\Gamma_L)$-module over $\mathbf{A}_L.$
\end{defn}
\begin{rem}
	\label{rem:omegaembedding}
	Fixing a choice of generator of $\varprojlim_n LT[\pi_L^n]$ leads to a canonical continuous $\Gamma_L$-equivariant embedding $$\iota\colon \mathbf{A}_L \to (W(\C_p^{\flat})_L)^{H_L}$$ such that $\iota\varphi_L(f) = \operatorname{Frob}_q(\iota(f)),$ whose image is independent of the choice. This embedding extends to an embedding of the $p$-adic completion of the maximal unramified extension  $\mathbf{A}$ of $\mathbf{A}_L$ into $W(\C_p^{\flat})_L.$
\end{rem}
\begin{proof} See \cite[Proposition 2.1.16, Remark 2.1.17 and Remark 3.1.4]{Schneider2017}.
\end{proof}
By abuse of notation we write $\varphi_L$ for the Frobenius operator on $W(\C_p^{\flat})_L,$ which is justified due to Remark \ref{rem:omegaembedding}.

\begin{thm} \label{thm:fontaineequiv}
	The functors $$V \mapsto \mathbb{D}(V):=(\mathbf{A}\otimes_{o_L} V)^{H_L}$$
	and $$D \mapsto \mathbb{V}(D):=(\mathbf{A}\otimes_{\mathbf{A}_L} D)^{\varphi_L=1}$$
	are quasi-inverse and give an equivalence of categories between $\operatorname{Rep}_{o_L}(G_L)$ and the category of étale $(\varphi_L,\Gamma_L)$-modules over $\mathbf{A}_L.$
	This equivalence is exact and respects duals, torsion-sub-objects and tensor products and inverting $p$ induces an equivalence of categories between $\operatorname{Rep}_L(G_L)$ and the category of étale $(\varphi_L,\Gamma_L)$-modules over $\mathbf{B}_L.$
\end{thm}
\begin{proof} See \cite[Theorem 3.3.10]{Schneider2017}.
\end{proof}
In more generality one has for $F/L$ finite a similar equivalence for $\operatorname{Rep}_{o_L}(G_F)$ but throughout this article we only concern ourselves with the case $F=L.$  
The rings $\mathbf{B}_L$ and $\cR_L$ are not directly comparable because the series in $\mathbf{A}_L$ do not necessarily converge on some annulus and on the other hand the coefficients of a series in $\cR_L$ are not necessarily bounded. 
However both rings contain the rings $$\mathbf{B}^{\dagger,r}_L:= \left\{ f = \sum_{k \in \Z} a_k T^k \mid  \lim_{k \to -\infty} \abs{a_k} r^k=0 \text{ and } \sup_{k} \abs{a_k} < \infty \right\}$$ and their union $\mathbf{B}_L^{\dagger}.$ 
\begin{defn}\label{def:dagger}
	A $(\varphi_L,\Gamma_L)$-module $D$ over $\mathbf{B}_L$ is called \textbf{overconvergent} if it admits a basis such that the matrices of $\varphi_L$ and all $\gamma \in \Gamma_L$ have entries in $\mathbf{B}^{\dagger}_L.$ The $\mathbf{B}^{\dagger}_L$-span of this basis is denoted by $D^{\dagger}.$ A representation  $V \in \operatorname{Rep}_L(G_L)$ is called \textbf{overconvergent} if $\mathbb{D}(V)$ is overconvergent. $V$ is called \textbf{$L$-analytic} if $\C_p \otimes_{L,\sigma}V$ is isomorphic to the trivial $\C_p$-semilinear representation $\C_p^{\operatorname{dim}_L(V)}$ for every embedding $\sigma:L \to \C_p$ with $\sigma \neq \id.$ If $V$ is overconvergent, we set $\mathbb{D}^{\dagger}_{\text{rig}}(V):= \cR_L \otimes_{\mathbf{B}^{\dagger}_L}\mathbb{D}^{\dagger}(V).$ 
\end{defn}
By \cite[Proposition 1.6]{FX12} étale $(\varphi_L,\Gamma_L)$-modules over $\cR_L$ always descend to $\mathbf{B}_L^{\dagger}.$  
But contrary to the classical cyclotomic situation there exist representations that are not overconvergent (cf. \cite[Theorem 0.6]{FX12}) and hence the category of $(\varphi_L,\Gamma_L)$-modules over $\cR_L$ is in some sense insufficient to study representations. Remarkably, restricting to $L$-analytic objects on both sides alleviates this problem. 
\begin{defn}
	Let $\rho: L^\times \to L^\times$ be a continuous character. We define $\cR_L(\rho)= \cR_Le_{\rho}$ as the free rank $1$  $(\varphi_L,\Gamma_L)$-module with basis $e_{\rho}$ and action given as $\varphi_L(e_{\rho}) = \rho(\pi_L)e_{\rho}$ and $\gamma(e_{\rho}) = \rho(\chi_{\mathrm{LT}}(\gamma))e_{\rho}.$ For $i \in \Z$ we write $\cR_L(x^i)$ as shorthand for the $(\varphi_L,\Gamma_L)$-module associated to the character $x \mapsto x^i.$
\end{defn}
\begin{rem}
	Let $M$ be a $(\varphi_L,\Gamma_L)$-module over $\cR_L$ of rank $1.$ Then there exists a character $\rho: L^{\times}\to L^{\times}$ such that $M$ is isomorphic to $\cR_L(\rho).$ The module is $L$-analytic if and only if $\rho_{\mid{o_L^{\times}}}$ is locally $L$-analytic.
\end{rem}
\begin{proof}
	See \cite[Proposition 1.9]{FX12}.
\end{proof}
\begin{defn}
	Let $M$ be an $L$-analytic $(\varphi_L,\Gamma_L)$-module over $\cR_L.$ If $M\cong \cR_L(\rho)$ has rank $1$ we define $$\deg(M):=\operatorname{val}_{\pi_L}(\rho(\pi_L)).$$ In general we define $\deg(M):= \deg(\Lambda^{\operatorname{rank}(M)}M)$ and finally the slope $\mu(M):= \deg(M)/\operatorname{rank}(M).$\footnote{With our conventions $\cR_L(x^i)$ has slope $i.$}
	$M$ is called \textbf{isoclinic} if $\mu(N)\geq \mu(M)$ for every subobject $N$ of $M.$
	$M$ is called \textbf{étale} if it is isomorphic to the $(\varphi_L,\Gamma_L)$-module  $\mathbb{D}^{\dagger}_{rig}(V)$ attached to some $L$-analytic $G_L$-representation $V.$
\end{defn}

\begin{thm}[Kedlaya/ Berger]
	Every non-zero ($L$-analytic) $(\varphi_L,\Gamma_L)$-module $M$ possesses a unique functorial filtration $$0 = M_0 \subsetneq M_1 \dots \subsetneq M_d = M$$ such that the succesive quotients $M_i/M_{i-1}$ are isoclinic ($L$-analytic) $(\varphi_L,\Gamma_L)$-modules and $\mu(M_1/M_{0})<\dots < \mu(M_d/M_{d-1}).$
	Furthermore the functor $V \mapsto D^{\dagger}_{rig}(V)$ defines an equivalence of categories between the category of $L$-analytic $L$-linear $G_L$-representations and the full subcategory of $L$-analytic $(\varphi_L,\Gamma_L)$-modules that are isoclinic of slope $0.$
\end{thm}
\begin{proof}
	From \cite[Theorem 6.4.1]{kedlaya2005slope} we obtain a unique filtration by (so-called) saturated $\varphi_L$-modules $M_i$. By the uniqueness of the filtration the additional $\Gamma_L$-structure is inherited by the modules in the filtration. Note that if $M$ is assumed to be $L$-analytic, then so are the $M_i$ in its Harder--Narasimhan filtration. The requirement in \textit{loc.\ cit.\ }for the sub-modules $M_i \subseteq M$ to be saturated is equivalent to the requirement that $M_i \subset M_{i+1}$ is a $\cR_L$-direct summand (c.f.\cite[§8.1]{pottharst2020harder}), which is equivalent to requiring that $M_i/M_{i+1}$ is again projective\footnote{Note that $\cR_L$ is a Bézout domain hence being finitely generated projective is equivalent to being finite free.}, i.e.,\ a $(\varphi_L,\Gamma_L)$-module. 
	The statement regarding étale modules is \cite[Theorem 10.4]{Berger2016}. Berger defines the notion of étale differently. The fact that being étale is equivalent to being isoclinic of slope $0$ is implicit in the proof of $10.1.$ \textit{ibidem.}
\end{proof}

\section{Cohomology of $M/t_{\mathrm{LT}}M$}

In order to conduct an induction over the slopes of a $(\varphi_L,\Gamma_L)$-module $M$ a detailed study of the cohomology of $M/t_{\mathrm{LT}}M$ is required. The following section is based on Sections 3.2 and 3.3 in \cite{KPX}. We adapt their methods to our situation. The main results are Lemmas \ref{lem:phimodt} and \ref{lem:psimodt}. The proofs are adaptations of their counterparts Propositions 3.2.4 and 3.2.5 in \cite{KPX}. They turn out to be more complicated due to the implicit nature of ``the'' variable $Z$ and the fact that by extending scalars to $K$ some care is required when studying the quotients $\cR_K^+/\varphi_L^n(T)$ since $K$ could contain non-trivial $\pi_L^n$-torsion points while having the trivial $\Gamma_L$-action. We will elaborate on crucial parts of the argument whose details are left out in \textit{loc.\@ cit.}, since we require them in greater generality.

\subsection{Some technical preparation}
Recall the product formula $$t_{\mathrm{LT}}= \log_{\mathrm{LT}}(X) = X\prod_{n\geq 1}\frac{Q_n(X)}{\pi_L},$$ where $Q_1(X)=\varphi(X)/X$ and $Q_n:=\varphi(Q_{n-1}).$
The zeroes of $Q_n$ are precisely the $\pi_L^n$-torsion points of the LT-group that are not already $\pi_L^{n-1}$ torsion points. Contrary to the cyclotomic situation these $Q_n$ are not necessarily polynomials. We denote by $G_n$ the polynomial $$G_n:= \prod_{a \in LT[\pi_L^n] \backslash LT[\pi_L^{n-1}]}(X-a).$$ By construction $G_n \in o_L[[X]]$ and $G_n \mid Q_n$ in $o_L[[X]]$ and therefore in every $\cR_L^{[r,s]}.$ We also observe that $Q_n/G_n$ is a unit in $\cR_L^{[r,s]}$ since it does not have any zeroes in $\C_p.$ 
\begin{rem}\label{cor:gammaspur} 
	Let $[r,s]$ be an interval with $r_L<r<s<1$ such that the corresponding annulus contains $LT[\pi_L^n].$ Fix a Basis $z = (z_n)_n$ of $\varprojlim LT[\pi_L^n].$
	By mapping $X$ to $z_n \in LT[\pi_L^n]$ we obtain a $\Gamma_L$-equivariant isomorphism $$\cR^{[r,s]}_L/Q_n \cong L_n$$ and the following diagrams commute:
	$$\begin{tikzcd}
		\cR^{[r,s]}_L/Q_n \arrow[r] \arrow[d, "\varphi_L"] & L_n \arrow[d, hook] & \cR^{[r^q,s^q]}_L/Q_{n+1} \arrow[d, "\psi_{\mathrm{LT}}"] \arrow[r] & L_{n+1} \arrow[d, "\pi_L^{-1}\operatorname{tr}_{L_{n+1}/L_n}"] \\
		\cR^{[r^q,s^q]}_L/Q_{n+1} \arrow[r]                    & L_{n+1}             & \cR^{[r,s]}_L/Q_n \arrow[r]                         & L_n                                                           
	\end{tikzcd}$$
\end{rem}
\begin{proof} By assumption the map in question is well-defined. The commutativity of the left-hand diagram is clear, the commutativity of the right-hand diagram follows from the definition of $\psi_{\mathrm{LT}}$ by first considering the map $(L \otimes_{o_L}o_L\llbracket T\rrbracket)/Q_n \to L_n,$ which factors over $\cR_L^{[r,s]}$ since $L_n =L(z_n)$ is a field and $z_n$ lies in the annulus $\mathbb{B}_{[r,s]}(\C_p).$ Note that a priori $\cR_L^{[r,s]}/Q_n \to L_n$ is merely surjective. By writing $Q_n = U_nG_n$ with a unit $U_n$ and a polynomial $G_n$ we can see, that the left hand-side is at most $[L_n:L]$-dimensional over $L,$ which implies the injectivity.
\end{proof}

\begin{lem}
	\label{lem:commutetensorlimit}
	
	\footnote{This is an adaption of an answer to \cite{181004}. Commonly this result is stated requiring $N_n$ to be of finite length, which does not suffice to treat our desired application to $R=A$ affinoid over $K$ unless $A$ is Artinian.}
	Let $M$ be a finitely presented $R$-module and $(N_n)_{n \in \N}$ either a countable projective Mittag-Leffler system of flat $R$-modules or a countable projective system of Artinian modules. Then the natural map
	$$M \otimes_R \varprojlim_n N_n \to \varprojlim_n M \otimes_R N_n$$ is an isomorphism.
\end{lem}
\begin{proof}We first treat the flat case.	
	Take a finite presentation $R^s \to R^r \to M \to 0$ of $M.$ Tensoring with $N:=\varprojlim_n N_n$ (resp. $N_n$)) allows us to express $M \otimes_R N$ (resp. $M \otimes_R N_n)$ as the cokernel of the induced map $N^s \to N^r$ (resp. $N_n^s \to N_n^r$). But then the statement follows if we can show $\operatorname{cok}(N^s \to N^r) = \varprojlim_{n} \operatorname{cok}(N_n^s \to N_n^r).$ 
	Consider the extended exact sequence 
	$$0 \to C \to R^s \to R^r \to M \to 0,$$ where $C:= \ker(R^s \to R^r).$ Since $N_n$ is assumed to be flat we have an exact sequence $$0 \to C \otimes_R N_n \to N_n^s\to N_n^r \to M \otimes N_n \to 0$$ and one checks that $C \otimes N_n$ is again Mittag-Leffler. Splitting the above sequence into two short sequences shows the vanishing of $\lim^1 N_n^s/(C \otimes N_n)$ since $\lim^1N_n^s=0$ surjects onto it, which via the isomorphism between coimage and image implies the vanishing of $\lim^1(\operatorname{im}(N_n^s\to N_n^r))$ hence the desired $\operatorname{cok}(N^s \to N^r) = \varprojlim_{n} \operatorname{cok}(N_n^s \to N_n^r).$
	In the Artinian case it is well-known that $(N_n)_n$ is Mittag-Leffler and since submodules of Artinian modules are Artinian, we see that $(\ker(N_n^s\to N_n^r))_n$ is also Mittag-Leffler. Hence one may proceed with the same arguments as in the first case.  
\end{proof}
\begin{lem} 
	\label{lem:ZlowerBound}
	Let $ n\geq  n_0$ and let $\rho: \Gamma_n \to K^{\times}$ be a non-trivial character of finite order. We denote by $K(\rho)$ the corresponding one-dimensional $K$-linear representation. Then 
	$$\norm{Z}_{K(\rho)}\geq\abs{p}^{\frac{1}{p-1}}.$$
\end{lem}
\begin{proof}
	First of all we remark  that any such character is automatically continuous by \cite[Théorème 0.1]{nikolov2003finite} because $\Gamma_n$ is topologically of finite type. $K(\rho)$ is a locally $L$-analytic representation since the orbit maps are even locally constant.
	Since $\Gamma_n \cong o_L$ by assumption, we have $\rho(\Gamma_n^{p^m})=1$ for some $m\gg0.$ By assumption there exists some $\gamma \in \Gamma_n$ such that $\rho(\gamma)\neq 1.$ In this case $\rho(\gamma)$ is some non-trivial $p$-power root of unity. In particular $\abs{\rho(\gamma)-1} \geq \abs{p}^{\frac{1}{p-1}}.$ Since $K(\rho)$ is one-dimensional the operator norm is multiplicative and because $\sum_{k\geq0} Z^k \in D(\Gamma_n,K)$ converges to a well-defined operator on $K(\rho)$ we necessarily have $\norm{Z}_{K(\rho)}<1.$ By expressing $\delta_{\gamma}-1$ as a power series $\delta_{\gamma}-1 = ZF(Z),$ we conclude $\norm{Z}_{K(\rho)} \geq \norm{ZF(Z)}_{K(\rho)}= \norm{\gamma-1}_{K(\rho)} \geq \abs{p}^{1/(p-1)},$ where for the first estimate we use $ZF(Z) \in Zo_K\llbracket Z\rrbracket.$
\end{proof}

\begin{lem}
	\label{lem:twistinvertible}
	
	Let $r \in (\abs{p}^{\frac{1}{p-1}},1)$ and fix a lift $X_n$ of $Z_n$ to $D_{\Q_p}(\Gamma_n,K)$ of norm $\abs{X}_{r}=C(r).$
	Let $W$ be a finite $A$-module with an $L$-analytic $\Gamma_n$-action such that $$\norm{\gamma-1}_{W}< \varepsilon:=C(r)^{-1}\abs{p}^{\frac{1}{p-1}}$$ for any $\gamma \in \Gamma_n.$ Then the action of $Z_n \in D(\Gamma_n,K)$ is invertible on $W(\rho) = W\otimes_K K(\rho)$ for any non-trivial character $\rho:\Gamma_n\to K^{\times}$ of finite order.
\end{lem}
\begin{proof} 
	It suffices to show that the action of $Z_m$ is invertible for some $m\geq n$ since $Z_{n+1}=\varphi(Z_n) = Q(Z_n)Z_n.$ If we fix a $\Z_p$-basis $\gamma_1,\dots,\gamma_d$ of $\Gamma_n$ the images $\rho(\gamma_i)$ have to be $p$-power roots of unity and by replacing $\Gamma_n$ with a small enough subgroup $\Gamma_m \subset \Gamma_n$ we may and do assume that $\rho(\gamma_i) \neq 1$ for at least one $i$ and $\rho(\gamma_i)$ is a $p$-th root of unity for every $i,$ i.e. $\rho$ is a non-trivial finite-order character whose values lie in the group of $p$-th roots of unity. By  \cite[Lemma 2.3.3]{RustamFiniteness} we can replace $Z_n$ by $Z_m$ and replace the lift $X_n$ by a lift $X_m$ of $Z_m$ whose $r$-norm is the same.
	Since $K$ is one-dimensional the $K$-linear action of $Z$ on $K(\rho)$ is either zero or invertible. Since $\rho$ is assumed to be non-trivial it has to be invertible. Let $X = \sum_{\mathbf{k} \in \N_0^d} a_\mathbf{k} \mathbf{b}^\mathbf{k}$ be a preimage of $Z$ in $D_{\Q_p}(\Gamma_m,K)$ of norm $\norm{X}_{r} \leq C(r).$ Our assumptions guarantee that $\gamma_i-1$ acts on $W \otimes K(\rho)$ with operator norm bounded above by $\abs{p}^{\frac{1}{p-1}} =\abs{\zeta_p-1}$ and Lemma \ref{lem:ZlowerBound} asserts that the invertible operator $\id \otimes Z$ has operator norm bounded below by $\abs{p}^{\frac{1}{p-1}}.$
	We use the notation $Z_{\text{diag}}$ to emphasize that $Z$ acts diagonally and compute
	\begin{align}
		&\notag\norm{Z_{\text{diag}} (a\otimes b) - (\id \otimes Z)(a \otimes b)}_{W(\rho)}\\
		&\notag\leq \sup_\mathbf{k}\abs{a_\mathbf{k}}\norm{\mathbf{b}^{\mathbf{k}}(a \otimes b) - a \otimes \mathbf{b}^{\mathbf{k}}(b)}_{W(\rho)}\\
		\label{eq:Zzdiagbasch}
		&<\sup_\mathbf{k}\abs{a_\mathbf{k}}\varepsilon  {\abs{p}^{\frac{1}{p-1}}}^{\abs{\mathbf{k}}}\norm{a \otimes b}_{W(\rho)}\\
		&\notag< \sup_\mathbf{k}\abs{a_\mathbf{k}}\varepsilon  r^{\abs{\mathbf{k}}}\norm{a \otimes b}_{W(\rho)}\\
		&\notag ={C(r)}^{-1} C(r){\abs{p}}^{\frac{1}{p-1}}\norm{a \otimes b}_{W(\rho)},
	\end{align}

where we use in \eqref{eq:Zzdiagbasch} the estimate $$\norm{(\gamma-1)(a \otimes b)-a \otimes (\gamma-1)b} = \norm{(\gamma-1)a \otimes \gamma b}< \varepsilon \norm{a \otimes b}$$ and the same inductive argument that we used in the proof of \cite[Lemma 2.3.4]{RustamFiniteness} to treat general multi-indices. 
We conclude $$\norm{Z_{\text{diag}}-(id \otimes Z)}< \norm{id \otimes Z}_{W(\rho)} = (\norm{\id \otimes Z}^{-1})_{W(\rho)}^{-1}$$ and by a standard geometric series argument we conclude that the diagonal action of $Z$ is invertible on $W(\rho).$ Note that $K(\rho)$ is one-dimensional hence the action of $1 \otimes Z$ is given by multiplication by a constant in $K^{\times}$ and hence satisfies the last equality $\norm{id \otimes Z}_{W(\rho)} = (\norm{\id \otimes Z}^{-1})_{W(\rho)}^{-1}.$
\end{proof}
\begin{lem} \label{lem:FlatAlgebraProjectiveResolution}
	Let $G$ be a compact $L$-analytic group, let $H\subset G$ be a (finite index) normal open subgroup, let $A$ be $K$-affinoid,
	let $S:=D(G,A)$\footnote{We take $D(G,A):=A\hat{\otimes}_KD(G,K)$ as a convention.}, $R:= D(H,A)$ and let $M$ be an $S$-module that admits (as an $R$-module) a $d+1$-term ($d \geq 0$) projective resolution by finitely generated projective $R$-modules. Then $M$ admits (as an $S$-module) a $d+1$-term projective resolution by finitely generated $S$-modules.
\end{lem}
\begin{proof}
	Let $T$ be an $S$-module. Using that the Dirac distributions $\delta_g,$ where $g$ runs through a system of representatives of $G/H,$ form a basis of $S$ as an $R$-module, we have $\operatorname{Hom}_{S}(M,T) = \operatorname{Hom}_R(M,T)^{G/H},$ where $G/H$ acts on a homomorphism $f$ via $(gf)(-) = gf(g^{-1}-).$ Because $A$ contains a field of characteristic $0$ the functor $(-)^{G/H}$ is exact on $A[G/H]$-modules and we obtain corresponding isomorphisms \begin{align*} \label{eq:extcomp} \operatorname{Ext}^i_{S}(M,T) \cong \operatorname{Ext}^i_R(M,T)^{G/H}.\end{align*}
	In particular the assumption on the length of the resolution asserts \begin{equation}\label{eq:extvanish}\operatorname{Ext}^i_{S}(M,T) =0 \tag{$*$}\end{equation}
	for every $i>d.$ This implies that the projective dimension of $M$ is bounded by $d.$
	By assumption we have a  resolution of $M$ with finitely generated projective $R$-modules.
	From  \cite[\href{https://stacks.math.columbia.edu/tag/064U}{Tag 064U}]{stacks-project} we obtain that $M$ is in particular pseudo-coherent as an $R$-module.
	The ring $S$ is finite free as an $R$ module and hence pseudo-coherent.
	Applying \cite[\href{https://stacks.math.columbia.edu/tag/064Z}{Tag 064Z}]{stacks-project} we can conclude that $M$ is pseudo-coherent as an $S$-module.	By  \cite[\href{https://stacks.math.columbia.edu/tag/064T}{Tag 064T}]{stacks-project}  admits a (potentially infinite) resolution
	$$\dots F_1 \to F_0 \to M\to0$$ by finite free $S$-modules.
	
	Using \eqref{eq:extvanish} the kernel of $F_{d} \to F_{d-1}$ is projective and we may truncate the sequence by replacing $F_d$ with $\operatorname{ker}(F_d \to F_{d-1}).$ The latter remains finitely generated because it is equal to the image of $F_{d+1} \to F_d$ by construction.
\end{proof}
\subsection{$\mod t_{\mathrm{LT}}$-cohomology}
We have seen in Remark \ref{cor:gammaspur} that $\cR_L^{[r,s]}/Q_n$ can be identified with the field $L_n$ whenever the zeroes of $Q_n$ (i.e. the $\pi_L^n$-torsion points of the LT group that are not already $\pi_L^{n-1}$-torsion) lie in the annulus $[r,s].$ This is $\Gamma_L$-equivariant for the Galois action on $L_n$ and $\varphi: \cR_L^{[r,s]}/Q_n\to \cR_L^{[r^q,s^q]}/Q_{n+1}$ corresponds to the inclusion $L_n \hookrightarrow L_{n+1}.$ When we extend coefficients to $K$ we let $\Gamma_L$ act trivially on the coefficients on $\cR_K.$ One has to be careful since $\cR_K^{[r,s]}/Q_n$ is not necessarily a field extension of $K.$ It is in general only some finite étale $K$-algebra, which we denote by $E_n.$ It carries a $\Gamma_L$-action induced from the action on $\cR_K.$ Using the decomposition $Q_n = G_nU_n$ the algebra $E_n$ can be explicitly described as $E_n = \cR_K^{[r,s]}/Q_n \cong K[T]/G_n \cong K \otimes_L L[T]/G_n \cong K \otimes_L L_n.$
Here $\cR_K^{[r,s]}$ carries its usual $\Gamma_L$-action while $\Gamma_L$ acts on the right-hand side via the right factor. For $n \geq n_0$ we define $$\mathfrak{L}_n:= \log_{\mathrm{LT}}(Z_n) \in D(\Gamma_n,K).$$ 

\begin{lem}
	\label{lem:modtprojlim}
	The natural map $\cR_A^{[r,1)}/(t_{\mathrm{LT}}) \to \varprojlim_{n \geq 0} \cR_A^{[r,1)}/(\varphi_L^n(T))$ is an isomorphism. 
\end{lem}
\begin{proof} 
	Recall the product decomposition $t_{\mathrm{LT}} = T\prod_{n\geq1}\frac{\varphi_L^n(T)}{\pi_L^{-1}\varphi_L^{n-1}(T)}.$ The partial products are equal to $P_n=\pi_L^{-(n-1)}\varphi^n(T).$ In particular $t_{\mathrm{LT}}$ is divisible by every $\varphi^n(T).$ We can describe $\cR_A^{[r,1)}/(t_{\mathrm{LT}})$ as $$\cR_A^{[r,1)}/(t_{\mathrm{LT}})=\operatorname{cok}(\cR_A^{[r,1)} \xrightarrow{t_{\mathrm{LT}}} \cR_A^{[r,1)})$$ and by coadmissibility we have $$\operatorname{cok}(\cR_A^{[r,1)} \xrightarrow{t_{\mathrm{LT}}} \cR_A^{[r,1)}) = \varprojlim_s \operatorname{cok}(\cR_A^{[r,s]} \xrightarrow{t_{\mathrm{LT}}} \cR_A^{[r,s]})$$ Now take a sequence of radii $s_n$ such that $[0,s_n]$ contains the $\pi_L^n$-torsion points of the $LT$-group (i.e. the zeroes of $\varphi^n(T))$ but no $\pi_L^{n+1}$-torsion points, that are not already $\pi_L^{n}$-torsion, i.e., none of the zeroes of $\varphi^{n+1}(T)/\varphi^n(T).$  In $\cR_K^{[0,s_n]}$ and hence in $\cR_K^{[r,s_n]}$ $\varphi^n(T)$ and $t_{\mathrm{LT}}$ differ by $t_{\mathrm{LT}}/\varphi^{n}(T)$, which has no zeroes in the annulus $[r,s_n]$ and is therefore a unit (since it is not contained in any maximal ideal of $\cR_K^{[r,s_n]}$ by \cite[3.3 Lemma 10]{bosch2014lectures} and the Weierstraß Preparation Theorem), hence they differ by a unit in $\cR_A^{[r,s_n]},$ in particular, $\cR_A^{[r,s_n]}/\varphi^n(T) = \cR_A^{[r,s_n]}/t_{\mathrm{LT}}.$ The statement now follows from $\cR_K^{[r,1)}/\varphi^n(T) = \cR_K^{[r,s_n]}/\varphi^n(T).$
\end{proof}
\begin{lem} 
	\label{lem:torspsi}Let $M^{r_0}$ be a $(\varphi,\Gamma)$-module over $\cR_A^{r_0}.$ Let $n_1=n_1(r_0) \in \N_0$ be minimal among $n \in \N$  such that ${\abs{\pi_L}}^{\frac{1}{q^{n-1}(q-1)}} \geq r_0.$ then: 
	\begin{itemize}
		\item $M^{r_0}/t_{\mathrm{LT}}M^{r_0} \cong \prod_{n\geq n_1}M^{r_0}/Q_nM^{r_0}.$
		\item  $\varphi^{m}$ induces an $A[\Gamma]$-linear isomoprhism $\varphi^m \otimes 1: M^{r_0}/Q_n \otimes_{E_n} E_{n+m} \to M^{r_0}/Q_{n+m}$ 
	\end{itemize}
\end{lem}
\begin{proof}
	The zeros of $Q_n$ are precisely the $\pi_L^n$-torsion points of the LT-group, which are not already $\pi_L^{n-1}$-torsion. Hence $Q_n$ is a unit in $\cR_K^{[r,s]}$ if and only if $$ v(n):={\abs{\pi_L}}^{\frac{1}{q^{n-1}(q-1)}} \notin [r,s].$$ Let $w(s)$ be the largest integer $n$ satisfying $v(n)\leq s.$ For a closed interval we obtain from the Chinese remainder theorem $$\cR_K^{[r_0,s]}/t_{\mathrm{LT}}\cR_K^{[r_0,s]} = \bigoplus_{n_1 \leq n \leq w(s)} \cR_K^{[r_0,s]}/Q_n\cR_K^{[r_0,s]}$$ and  $$\cR_A^{[r_0,s]}/t_{\mathrm{LT}}\cR_A^{[r_0,s]} = \bigoplus_{n_1 \leq n \leq w(s)} \cR_A^{[r_0,s]}/Q_n\cR_A^{[r_0,s]}.$$ 
	
	The first statement follows by passing to the limit $s \to 1.$
	The second statement follows inductively from the case $m=1.$
	Because the linearised map is an isomorphism we have 
	\begin{align}M^{r_0}/Q_{n+1} &= M^{r_0^{1/q}}/Q_{n+1}\\
		&\cong \varphi_L^*(M^{r_0})/\varphi_L(Q_{n}) \\
		&\cong M^{r_0}/Q_{n}\otimes_{E_n}E_{n+1},
	\end{align}
	where for the last isomorphism we use $\cR_K^{r_0^{1/q}}/Q_{n+1} = E_{n+1}$ since the zeros of $Q_{n+1}$ are the preimages of the zeros of $Q_{n}$ under $\varphi$ and we assumed that the latter are contained in $[r_0,1).$ We further used that $\varphi$ induces the canonical inclusion $E_n \to E_{n+1}$ and that the identification $\cR_K^{[r,s]}/Q_{n}\cong E_n$ is $\Gamma$-equivariant as described in the beginning of the chapter.
\end{proof}
To keep notation light we define $M_n := M_n^r:= M^{r}/Q_nM^r$ and suppress the dependence on $r.$ This poses no problem as long as $M$ admits a model over $[r,1)$ and $n \geq n_0(r)$ satisfying the conditions of Lemma \ref{lem:torspsi}. 
\begin{cor}
	\label{kor:psimodt}
	With respect to the decomposition $$M^{r_0}/t_{\mathrm{LT}}M^{r_0} \cong \prod_{n\geq n_1}M^{r_0}/Q_nM^{r_0} = \prod_{n \geq n_1} M_n$$ the map $\varphi_L: M^{r_0}/t_{\mathrm{LT}}M^{r_0} \to M^{r_0^{1/q}}/t_{\mathrm{LT}}M^{r_0^{1/q}}$ takes $(x_n)_n$ to $(x_{n-1})_n.$ The map $\psi_{\mathrm{LT}}:  M^{r_0^{1/q}}/t_{\mathrm{LT}}M^{r_0^{1/q}} \to M^{r_0}/t_{\mathrm{LT}}M^{r_0}$ takes $(x_n)_n$ to \linebreak $(\pi^{-1}\operatorname{tr}_{E_{n}/E_{n-1}}(x_n))_n,$ where 
	$$\operatorname{tr}_{E_{n}/E_{n-1}}: M^{r_0}_n = M^{r_0}_{n-1} \otimes_{E_{n-1}} E_n  \to M_{n-1}^{r_0}$$ 
	is given by the trace induced from the trace $L_n \to L_{n-1}$ on the second factor of $E_{n} \cong K \otimes_L L_n$. 
\end{cor}
\begin{proof}
	The statement for $\varphi_L$ follows by combining both points in Lemma \ref{lem:torspsi}. The statement for $\psi_{\mathrm{LT}}$ follows from Remark \ref{cor:gammaspur}.
\end{proof}
\begin{lem}
	\label{lem:phimodt} Let $M$ be an $L$-analytic $(\varphi_L,\Gamma_L)$-module over $\cR_A.$ Then
	the cohomology groups $H^i_{\varphi,Z}(M/t_{\mathrm{LT}}M)$ vanish outisde of degrees $0$ and $1$ and are finitely generated $A$-modules.
\end{lem}
\begin{proof}
	Assume $M$ is defined over $[r_0,1)$ and let $n_1\in \N$ such that $M^{r_0}/t_{\mathrm{LT}}M^{r_0} \cong \prod_{n\geq n_1} M_n.$ We claim that $\varphi_L-1$ is surjective which implies the vanishing of $H^2.$ Let $(x_n) \in \prod_{n\geq n_1+1} M_n.$ Set $y_n:= -x_n +x_{n-1} -x_{n-2}+ \dots + (-1)^{n-n_1}x_{n_1+1}.$ Then $(\varphi_L-1)(y_{n-1}) = x_n.$ On the other hand $$\operatorname{ker}((\varphi_L-1)_{[r,1)}) =(M^{r}/t_{\mathrm{LT}}M^r)^{\varphi_L=1} \cong M_{n_1(r)}$$ and thus $$\operatorname{ker}(\varphi_L-1)=\varinjlim_{r>0}(M^{r}/t_{\mathrm{LT}}M^r)^{\varphi=1}  \cong \varinjlim_n M_n.$$
	Next consider the complex
	$$(\varinjlim_n M_n) \xrightarrow{Z} (\varinjlim_n M_n).$$ Using Lemma \ref{lem:torspsi} we can explicitly describe each module appearing in the direct limit as $M_{n_1} \otimes_{L_{n_1}} L_m,$ where $W:=M_{n_1}$ is an $A$-module of finite type with a continuous $L$-analytic $A$-linear $\Gamma_L$-action and $\Gamma_L$ acts on the right factor via its natural action. We claim that for $m\gg0$ the natural map $$[M_m \xrightarrow{Z} M_m] \to [(\varinjlim_n M_n) \xrightarrow{Z} (\varinjlim_n M_n)]$$ is a quasi-isomorphism. By the normal basis theorem and Maschke's theorem there is for any pair $m \geq m'$ an isomorphism of representations $L_m \cong L_{m'}[\Gamma_{m'}/\Gamma_m]\cong \prod_\eta L_{m'}(\eta),$ where the product runs over all characters of $\operatorname{Gal}(L_m/L_{m'}).$ 
	Hence for $n\geq m$ as representations $$M_{n_1} \otimes_{L_{n_1}} L_n \cong \bigoplus_{\rho} M_{n_1}(\rho) = \bigoplus_{\rho(\Gamma_m)=1} M_{n_1}(\rho) \oplus \bigoplus_{\rho(\Gamma_m)\neq 1} M_{n_1}(\rho),$$ where $\rho$ runs through the characters of $\Gamma_{n_1}/\Gamma_n.$
	
	It suffices to show that there exists a $m$ such that for any $\rho$ with $\rho(\Gamma_m)\neq1$ the action of $Z$ is invertible meaning that the only contribution to the cohomology comes from the components corresponding to characters vanishing on $\Gamma_m$ hence the claim. 
	It suffices to show that the action of $Z_m$ is invertible for some $m \gg 0$ hence we may assume that $M_{n_1}$ satisfies the estimates of Lemma \ref{lem:twistinvertible} with respect to the action of $\Gamma_m.$ But then Lemma \ref{lem:twistinvertible} asserts that the action of $Z_m$ on $M_{n_1}(\rho)$ is invertible for any $\rho$ that is not trivial on $\Gamma_m.$ Since $M_{m}$ is finitely generated over $A$ we conclude that the complex computing $H^i_{\varphi,Z}(M/t_{\mathrm{LT}}M)$ is quasi-isomorphic to a complex of finitely generated $A$-modules and because $A$ is Noetherian we conclude that the cohomology groups are finitely generated.
	
\end{proof}
\begin{lem}
	\label{lem:psimodt}
	Let $M$ be an $L$-analytic $(\varphi_L,\Gamma_L)$-module over $\cR_A.$ Then $$\frac{\pi_L}{q}\psi_{\mathrm{LT}}-1:M/t_{\mathrm{LT}}M \to M/t_{\mathrm{LT}}M$$ is surjective and its kernel viewed as a $D(\Gamma_n,A)$-module admits a $2$-term finite projective resolution for any $n\geq n_0$.
\end{lem}
\begin{proof}
	Let $M^{r_0}$ be a model of $M$ with $r_0>\left\lvert \pi_L^{\frac{1}{q^{n_0-1}(q-1)}}\right\rvert$ such that $n_1(r_0)\geq n_0.$ Consider the decomposition $M^{r_0}/t_{\mathrm{LT}}M^{r_0}\cong \prod_{n\geq n_1}M_n$ from Lemma \ref{lem:torspsi} and let $x=(x_n)_{n\geq n_1} \in \prod_{n\geq n_1}M_n.$ If $z \in M_{n+1}$ belongs to the image of $M_n$ then its trace is $qz$ since due to our assumptions on $r_0$ we have $[L_{n+1}:L_{n}]=q$ for every $n\geq n_1.$
	In particular $\frac{\pi_L}{q}\pi_L^{-1}\operatorname{Tr}(z)=z$ which via the explicit description of the $\psi_{\mathrm{LT}}$-action in Corollary \ref{kor:psimodt} should be read as ``$\frac{\pi_L}{q}\psi_{\mathrm{LT}}(z)=z$''. Using the explicit description of $M^{r_0}/t_{\mathrm{LT}}M^{r_0}$ from Corollary \ref{kor:psimodt} we shall in the following define a tuple $y=(y_n)_{n\geq n_0(r_0^{1/q})} \in M^{r_0^{1/q}}/t_{\mathrm{LT}}M^{r_0^{1/q}}$ with $y_{n_0(r^{1/q})}=0$ such that $(\frac{\pi_L}{q}\psi_{\mathrm{LT}}-1)y=x.$ 
	This notation is abusive since in order to make sense of $\frac{\pi_L}{q}\psi_{\mathrm{LT}}-1$ we need to view $\psi_{\mathrm{LT}}$ as a map $$M^{r_0^{1/q}}/t_{\mathrm{LT}}M^{r_0^{1/q}} \xrightarrow{\frac{\pi_L}{q}\psi_{\mathrm{LT}}} M^{r_0}/t_{\mathrm{LT}}M^{r_0} \xrightarrow{\res} M^{r_0^{1/q}}/t_{\mathrm{LT}}M^{r_0^{1/q}}$$ i.e. via the description from Corollary \ref{kor:psimodt} a map $$ \prod_{n \geq n_0(r_0^{1/q})} M_n\xrightarrow{(\frac{\pi_L}{q}\pi_L^{-1}\operatorname{Tr})_n}  \prod_{n \geq n_1(r_0)=n_0(r_0^{1/q})-1} M_n \xrightarrow{\res}0 \times \prod_{n \geq n_1(r_0^{1/q})} M_n.$$ The map $\res$ is given by mapping $(x_n \mod Q_n)_n$ to the restriction in every component, i.e., $(\res(x_n) \mod Q_n)_n$ and one can see that by the choice of $n_1(r_0)$ the element $Q_{n_1}$ becomes invertible when restricted to $[r_0^{1/q},1).$ In particular the $n_1(r_0)$-th component is mapped to zero and the last map is given by ommiting this component.  We define $y_{n_1(r_0)}=0$ and $y_{n} =\sum_{j=n_1}^{n-1} x_{j}.$ One can see inductively that $(\frac{\pi_L}{q}\psi_{\mathrm{LT}}-1)(y)=x.$ Indeed the map $\frac{\pi}{q}\psi_{\mathrm{LT}}$ corresponds to shifting indices and applying $\frac{1}{q}$-times the trace map by Corollary \ref{kor:psimodt}. Hence we obtain $(\frac{\pi_L}{q}\psi_{\mathrm{LT}}-1)(y) = (\frac{1}{q}\operatorname{Tr}(y_{n+1})-y_{n})_n,$ which turns out to be $(x_n)_n$ since in every component $M_{n+1}$ we apply $1/q\operatorname{Tr}$ to elements in the image of $M_n$ such that all terms but $x_{n+1}$ cancel out. This proves the surjectivity of $\frac{\pi_L}{q}\psi_{\mathrm{LT}}-1.$\\
	 Using the description of $M^{r_0}/t_{\mathrm{LT}}M^{r_0}$ and the $\psi_{\mathrm{LT}}$-action we obtain $$\ker\left(\frac{\pi_L}{q}\psi_{\mathrm{LT}}-1\right) = \left\lbrace(m_n)_n \mid \frac{1}{q}\operatorname{Tr}(x_{n+1})=x_n\right\rbrace = \varprojlim_{\frac{1}{q}\operatorname{Tr}(L_{n+1}/L_n)} M_n \otimes_{L_n}L_{n+1}.$$
	By \cite[Corollary 1.2.9]{RustamFiniteness} (together with base change from $K$ to $A$) we have an isomorphism $D(\Gamma_{n_1},A)/Z_{n} \cong A[\Gamma_{n_1}/\Gamma_n]$ and we obtain \begin{align}\varprojlim_{\frac{1}{q}\operatorname{Tr}(L_{n+1}/L_n)} M_n \otimes_{L_n}L_{n+1}&\cong \varprojlim_n M_{n_1}\otimes_{A}D(\Gamma_{n_1},A)/Z_{n} \\&=M_{n_1}\otimes_A \varprojlim_n D(\Gamma_{n_1},A)/Z_n \\ &= M_{n_1} \otimes_A D(\Gamma_{n_1},A)/\mathfrak{L}_{n_1}.
	\end{align}
	Using that $M_{n_1}$ is finitely presented over $A$ due to being finitely generated over a Noetherian ring and that $(A[\Gamma_{n_1}/\Gamma_n])_{n\geq n_1}$ is Mittag-Leffler due to having surjective transition maps consisting of free and hence flat $A$-modules to apply Lemma \ref{lem:commutetensorlimit}. In the last equality we use the relationship \eqref{eq:Znbeziehung} and use Lemma \ref{lem:modtprojlim} via transport of structure along $\cR_A^+ \cong D(\Gamma_{n_1},A)$ since under the map $T \mapsto Z_{n_1}$ the element $t_{\mathrm{LT}}$ is mapped precisely to $\mathfrak{L}_{n_1}.$ This isomorphism is $\Gamma_{n_1}$-equivariant with respect to the diagonal action. We have a naive resolution $$M_{n_1} \otimes_A D(\Gamma_{n_1},A)/\mathfrak{L}_{n_1} = \operatorname{cok}(M_{n_1} \otimes_A D(\Gamma_{n_1},A)\xrightarrow{\operatorname{id}\otimes \mathfrak{L}_{n_1}} M_{n_1} \otimes_A D(\Gamma_{n_1},A)).$$ We shall prove that each factor of this resolution is finite projective over $D(\Gamma_{n_1},A)$ with respect to the diagonal action,  which will complete the proof.  Using Lemma \ref{lem:FlatAlgebraProjectiveResolution} applied to the algebras $D(\Gamma_m,A)\subset D(\Gamma_{n_1},A)  \subset D(\Gamma_{n_0},A)$ we are reduced to proving that each factor is finite projective over $D(\Gamma_m,A)$ for some $m\geq n_0.$ We remark at this point that $M_{n_1}$ is projective over $\cR_A^{r_0}/Q_{n}$ due to the projectivity of $M^{r_0}$ and hence projective over $A$ since $\cR_A^{r_0}/Q_{n}$ is free over $A.$ Because $M_{n_1}$ is finitely generated projective over $A,$ we can choose some finitely generated projective complement $N_{n_1}$ which we view with the trivial $\Gamma_L$-action, such that $M_{n_1}\oplus N_{n_1} \cong A^d$ and we endow the left-hand side with the norm corresponding to the $\sup$-norm of the Banach norm on the right side with respect to some basis $e_1,\dots,e_d$. 
As $M_{n_1}\otimes_A D(\Gamma_{m},A)$ is a direct summand of $(M_{n_1}\oplus N_{n_1}) \otimes_A D(\Gamma_{m},A)$ we are reduced to showing that the latter is projective as a $D(\Gamma_{m},A)$-module.
	Note that the isomorphism $M_{n_1}\oplus N_{n_1} \cong A^d$ is not $\Gamma_L$-equivariant for the trivial action on $A.$ Hence we cannot conclude that $(M_{n_1}\oplus N_{n_1}) \otimes_A D(\Gamma_{m},A) \cong D(\Gamma_{m},A)^d.$  
    After some technical reductions, we will show that the elements $e_i \otimes 1, $ which are obviously a basis of $(M_{n_1}\oplus N_{n_1}) \otimes_A D(\Gamma_{m},A)$ for the trivial action on $M_{n_1},$ also form a basis with respect to the diagonal action, following the strategy of the proof of \cite[Proposition 3.2.5]{KPX}.\footnote{In \cite{KPX} the fact that in the cyclotomic case $\Gamma_m$ admits a topological generator $\gamma_m$ makes estimating the operator norm of $Z_m = \gamma_m-1$ easier. Additionally, in the case $L=\QQ_p,$ there is no need to differentiate between $\mathbb{Q}_p$-analytic and $L$-analytic distributions.}
	 By an analogue of Lemma \ref{lem:operatornormestimates} for finitely generated $A$-modules with $A$-linear $L$-analytic $\Gamma_L$-action we may assume that $\lvert\lvert Z_m\rvert\rvert_{M_{n_1}\oplus N_{n_1}}<\varepsilon<1$ after eventually enlarging $m$ and the $D(\Gamma_m,K)$-action extends to an action of $D_{\mathfrak{r}_l}(\Gamma_m,K)$ for any $l\geq l_0$ with a suitable $l_0\in \N.$
	By the same reasoning as in the proof of \cite[Lemma 2.1.2]{RustamFiniteness} the action extends to a continuous action of $D_{\mathfrak{r}_l}(\Gamma_m,A) := A \hat{\otimes}_K D_{\mathfrak{r}_l}(\Gamma_m,K).$\footnote{The precise value $\mathfrak{r}_l = p^{-\frac{1}{p^l}}$ is not relevant in the following. One could replace $\mathfrak{r}_l$ by any sequence converging to $1,$ that is bounded below by $\mathfrak{r}_{l_0}.$}   
	Consider the maps 
	
	\begin{align}
		\bigoplus_{i=1}^d D(\Gamma_m,A)e_i &\to (M_{n_1}\oplus N_{n_1})\otimes_A D(\Gamma_m,A) \\
		\Phi : f(Z_m)e_i &\mapsto f(Z_m)\cdot(e_i\otimes 1) \\
		\Psi: f(Z_m)e_i &\mapsto e_i \otimes f(Z_m)
	\end{align}
	
	By construction $\Phi$ is equivariant for the diagonal action while $\Psi$ is a topological isomorphism. It remains to conclude that $\Phi$ is an isomorphism and for that purpose it suffices to show that $$\Phi \circ \Psi^{-1}: D(\Gamma_m,A)^d \to D(\Gamma_m,A)^d$$ is an isomorphism. By passing to the limit it suffices to show that for any $\mathfrak{r}_l>\varepsilon$ the induced map $$\Phi \circ \Psi^{-1}: D_{\mathfrak{r}_l}(\Gamma_m,A)^d \to D_{\mathfrak{r}_l}(\Gamma_m,A)^d$$ is an isomorphism. We henceforth assume $\mathfrak{r}_l>\varepsilon$ in particular we may find $\delta <1$ such that  $\varepsilon= \mathfrak{r}_l\delta.$\footnote{This choice of $\delta$ is a technicality in order to obtain a strict bound with respect to the quotient topology in \eqref{eq:quotientnormstrict}.} 
	Let $\overline{\lambda}= (\overline{\lambda_1},\dots,\overline{\lambda_d}) \in  D_{\mathfrak{r}_l}(\Gamma_m,A)^d$ and let $\lambda_i \in D_{\Q_p,\mathfrak{r}_l}(\Gamma_m,A)$ be lifts. We wish to show $\norm{\Phi\circ \Psi^{-1}-\id}_{D_{\mathfrak{r}_l}(\Gamma_m,A)^d}<1.$ For that purpose we view the $D(\Gamma_m,A)$-action as a $D_{\Q_p}(\Gamma_m,A)$ action that factors over the natural projection $D_{\Q_p}(\Gamma_m,A) \to D(\Gamma_m,A).$ This allows us to define analogously $$\Phi_{\Q_p},\Psi_{\Q_p}: \bigoplus_{i=1}^d D_{\Q_p}(\Gamma_m,A)e_i \to A^d \otimes_A D_{\Q_p}(\Gamma_m,A),$$ given on Dirac distributions by $\Phi_{\Q_p}(\gamma e_i) = \gamma( e_i \otimes 1)$ and $\Psi_{\Q_p}(\gamma e_i) = e_i \otimes \gamma.$ Evidently $\Psi_{\Q_p}$ is an isomorphism. We denote by $\norm{\cdot}_{\mathfrak{r}_l}$ the norm introduced in \cite[Definition 1.3.6]{RustamFiniteness}. 
	Let $\gamma_1,\dots,\gamma_h$ be a $\Z_p$-Basis of $o_L$ and let $\mathbf{b} = (\gamma_j-1)_j.$ Recall that we have $\norm{\gamma_j-1}_{\mathfrak{r}_l} = \mathfrak{r}_l<1$ by definition and hence also $\norm{\gamma_j}_{\mathfrak{r}_l} =1.$ 
	We first show $$\norm{\Phi_{\Q_p}\circ\Psi_{\Q_p}^{-1}(e_i\otimes\mathbf{b}^\mathbf{k})-e_i \otimes \mathbf{b}^\mathbf{k}}\leq \varepsilon \mathfrak{r}_l^{\abs{\mathbf{k}}-1}\leq \delta \mathfrak{r}_l^{\abs{\mathbf{k}}}<\mathfrak{r}_l^{\abs{\mathbf{k}}}$$ for any $\mathbf{k}\in \N_0^h.$
	The assumption on $\varepsilon$ guarantees $$(\gamma_i-1)(x\otimes y) = (\gamma_i-1)x \otimes \gamma_i y + x \otimes (\gamma_i-1)y$$ has operator norm $\leq \mathfrak{r}_l,$ which allows us to reduce the computation by induction on $\abs{\mathbf{k}}$ (the case $k=0$ being trivial) .
	Assume $\mathbf{b}^{\mathbf{k}}= (\gamma_j-1)\mathbf{b}^{\mathbf{k}'}$ with $\abs{\mathbf{k}'}+1=\abs{\mathbf{k}}.$ We compute 
	\begin{align}
		&\Phi_{\Q_p}\circ\Psi_{\Q_p}^{-1}(e_i \otimes\mathbf{b}^\mathbf{k})-e_i\otimes\mathbf{b}^\mathbf{k}\\
		&= \mathbf{b}^\mathbf{k}(e_i\otimes 1)-e_i\otimes\mathbf{b}^\mathbf{k} \\
		&= (\gamma_j-1)\mathbf{b}^{\mathbf{k}'}(e_i\otimes 1) - e_i\otimes(\gamma_j-1)\mathbf{b}^{\mathbf{k}'}\\
		&= (\gamma_j-1)(\Phi_{\Q_p}\circ\Psi_{\Q_p})(e_i\otimes\mathbf{b}^{\mathbf{k}'})- (\gamma_j-1)(e_i\otimes(\mathbf{b}^{\mathbf{k}'}))- (\gamma_j-1)e_i \otimes \gamma_j\mathbf{b}^{\mathbf{k}'}\\
		&= (\gamma_j-1)([\Phi_{\Q_p}\circ\Psi_{\Q_p}^{-1}-\id](e_i \otimes\mathbf{b}^{\mathbf{k}'})) - ((\gamma_j-1)e_i)\otimes\gamma_j\mathbf{b}^{\mathbf{k}'}. 
	\end{align} Assuming that the corresponding estimate holds for $\mathbf{b}^{\mathbf{k}'}$ we obtain
	$$\norm{(\Phi_{\Q_p}\circ\Psi_{\Q_p}^{-1}-\id)(e_i \otimes\mathbf{b}^\mathbf{k})} \leq \sup(\mathfrak{r}_l \norm{(\Phi_{\Q_p}\circ\Psi_{\Q_p}^{-1}-\id)(e_i \otimes\mathbf{b}^{\mathbf{k}'})}, \varepsilon \mathfrak{r}_l^{\abs{\mathbf{k}'}})\leq \delta \mathfrak{r}_l^{\abs{\mathbf{k}}}.$$
	In conclusion for each $\overline{\lambda_i}$ and any lift $\lambda_i$ thereof we have $$\norm{(\Phi_{\Q_p}\circ{\Psi_{\Q_p}}^{-1}-\id)(e_i \otimes \lambda_i)}<\norm{\lambda_i}_{\mathfrak{r}_l}.$$ More precisely our proof shows $$\norm{(\Phi_{\Q_p}\circ{\Psi_{\Q_p}}^{-1}-\id)(e_i \otimes \lambda_i)}\leq \delta \norm{\lambda_i}_{\mathfrak{r}_l}.$$ Hence the corresponding estimate with respect to the quotient norm $\norm{\cdot}_{\overline{\mathfrak{r}_l}}$ of $\norm{\cdot}_{\mathfrak{r}_l}$ on $D(\Gamma_m,A)$ namely 
	\begin{equation}
		\label{eq:quotientnormstrict}\norm{(\Phi\circ{\Psi}^{-1}-\id)(e_i \otimes \lambda_i)}\leq\delta\norm{\lambda_i}_{\overline{\mathfrak{r}_l}} <\norm{\lambda_i}_{\overline{\mathfrak{r}_l}}
	\end{equation} remains valid. By a geometric series argument $\Phi\circ{\Psi}^{-1}$ is an isomorphism forcing $\Psi$ to also be an isomorphism.
	
\end{proof}

The two following results are analogues of \cite[Proposition 3.3.2 (1) and (2)]{KPX}. Since their proofs are direct adaptations, we do not repeat them here. For details we refer to \cite{rusti}. The case $A=K$ was treated earlier in  \cite{SchneiderVenjakobRegulator}.
\begin{thm}
	\label{thm:finitepsi}
	Let $M$ be a $(\varphi,\Gamma)$-module over $\cR_A.$ Then $M/(c\psi-1)$ is a finitely generated $A$-module for any $c\in A^\times.$
\end{thm}

\begin{lem}
	\label{lem:psisurjektiv}
	There exists $n\gg 0$ such that $\psi-1: t_{\mathrm{LT}}^{-n}M \to t_{\mathrm{LT}}^{-n}M$ is surjective.
\end{lem}

\section{Lubin--Tate deformations and Iwasawa cohomology}
In this section we study deformations of $(\varphi_L,\Gamma_L)$-modules with the distribution algebra and their relationship to Iwasawa cohomology. 
We shall establish a comparison between the $(\Psi,Z)$-cohomology of $\Dfm(M)$ and the Iwasawa cohomology of $M$ which is defined as the cohomology of the complex $$C_{\Psi}(M): M \xrightarrow{\Psi-1}M$$ concentrated in degrees $[1,2].$ This is motivated by the fact that for an étale module (coming from a representation $V$) this cohomology is closely related to the Iwasawa cohomology of $V.$
Roughly speaking the deformation is a family of $(\varphi_L,\Gamma_L)$-modules parametrised by the Fréchet-Stein algebra $D(\Gamma_L,K)$ and specialising to a point $\mathfrak{m}_x\in \operatorname{Sp}(D(\Gamma_L,K))$ corresponds to twisting the module at $x=0$ by an analytic character. As before we run into the problem that, contrary to the cyclotomic case, the inclusion $\Gamma_{n_0} \subset \Gamma_L$ does not split and hence we restrict a priori to this subgroup. 

For now fix $U:=\Gamma_m$ for some $m \geq{n_0}.$ We shorten our notation and write $D:=D(U,K)$ and pick an affinoid cover $D= \varprojlim_{n} D_n$ that arises as a base change of an affinoid cover of $D(U,L)$ with $U$-stable terms (e.g. $D_n:=D_{r_n}(U,K)$ for a non trivial sequence $r_n$ converging to $1$ from below). We further abbreviate $D(\Gamma_L) := D(\Gamma_L,K) = \Z[\Gamma_L]\otimes_{\Z[U]}D(U,K)$ and $D_n(\Gamma_L):=\Z[\Gamma_L]\otimes_{\Z[U]}D_n.$ In this chapter let $M$ be an $L$-analytic $(\varphi_L,\Gamma_L)$-module over $\cR_A$ with a model over $[r_0,1).$ As before we denote by $Z$ the preimage of a coordinate $T$ under the Fourier isomorphism.
We denote by $\Psi$ the left-inverse operator to $\varphi_L$ i.e. $\Psi= \frac{\pi_L}{q} \psi_{\mathrm{LT}}.$ We henceforth omit the subscript $\psi_{\mathrm{LT}}$ and write $\psi:= \psi_{\mathrm{LT}}.$ 
\begin{defn} 
	\label{def:dfm}We define for $r \geq r_0$ $$\Dfm_n(M^r):= D_n \hat{\otimes}_{K}M^r$$ and $$\Dfm(M^r):= \varprojlim_n \Dfm_n = D \hat{\otimes}_{K} M^r.$$
	We endow $\mathbf{Dfm}_n(M^r)$ (resp. $\mathbf{Dfm}(M^r)$) with an action of $\varphi_L$ and $U$ via $\gamma(a \otimes m) = \delta_{\gamma^{-1}}a \otimes \gamma m$ and  $\varphi_L(a \otimes m):=a \otimes \varphi_L(m).$
	We further define $\Dfm_n^{\Gamma_L}(M^r)$ (resp.\ $\Dfm^{\Gamma_L}(M^r)$) as $D_n(\Gamma_L) \hat{\otimes}_KM^r$ (resp. $D(\Gamma_L) \hat{\otimes}_KM^r$) with analogously defined actions. As before, all completions are taken with respect to the projective tensor product topology. 
\end{defn}

In particular $\Dfm(M^r)$ can be viewed as a sheaf on (the rigid analytic space associated to) $D.$ We could also define $$\Dfm_n(M):= D_n \hat{\otimes}_{K,i}M,$$ where $\hat{\otimes}_{K,i}$ denotes the completion with respect to the inductive tensor product topology, but passing to $\Dfm$ poses a problem, since inductive tensor product topologies do not necessarily commute with projective limits. However the inductive tensor product topology is the most reasonable choice for a tensor product of an LF-space and a Fréchet space. We shall avoid this problem by working on the level of models. 
\begin{prop} $\Dfm_n^{\Gamma_L}(M^r)$ is an $L$-analytic family of $(\varphi_L,\Gamma_L)$-modules over $\operatorname{Sp}(D_n({\Gamma_L})) \times_K \operatorname{Sp}(A)$ i.e. an $L$-analytic $(\varphi_L,\Gamma_L)$-module over the relative Robba ring $\cR_{D_n({\Gamma_L})\hat{\otimes}A}^{[r,1)}.$ 
\end{prop}
\begin{proof}
	By construction $\Dfm_n^{\Gamma_L}(M^r)$ is finite projective over $\cR_{D_n({\Gamma_L})\hat{\otimes}A}^{[r,1)}$  of the same rank as $M.$ The actions are semi-linear (in particular $D_n \hat{\otimes}A$-linear) because $$\gamma(\lambda\mu \otimes am) = (\lambda \otimes a) (\delta_{\gamma}^{-1}\mu \otimes m)$$ for any $\lambda \in D, \mu \in D_n, a\in A$ and $m \in M^r.$ Note that semi-linearity refers to the action of $\lambda \otimes a \otimes f \in \cR_{D_n\hat{\otimes}A}$ via multiplication
	$(\lambda \otimes a \otimes f)(\mu \otimes m) = \lambda \mu \otimes afm,$ where $\lambda \in D, \mu \in D_n, a\in A, f\in \cR_K,m \in M^r.$ $L$-analyticity (and hence continuity) follows from \cite[Lemma 1.3.5]{RustamFiniteness} once we establish that the $U$-action on $D_n$ (via inverted multiplication) is locally $L$-analytic. By our assumptions the Dirac distribution corresponding to $u \in U$ admits an expansion $\delta_u = \eta(a(u),Z)$ with a suitable $a(u) \in o_L.$ Fixing $\lambda \in D_n$ its orbit map is given by $u \mapsto \eta(-a(u),Z)\lambda.$ Expanding out the terms shows that the orbit map is locally $L$-analytic. The $\varphi_L$ and $\Gamma_L$-actions clearly commute and the linearised $\varphi_L$-map is invertible, since it is the linear extension of $\varphi_{M^r}.$ 
\end{proof}
\begin{rem}
	$\Dfm(M^r), \Dfm_n(M^r)$ carry two different $D$-actions. One induced by the scalar action in the left tensor component, that we shall call scalar action. The other one induced by the $L$-analytic action defined in Definition \ref{def:dfm}, which we shall call diagonal action.
	Note that $\Dfm(M^r)$ is only a $(\varphi_L,U)$-module but \textbf{not} a $(\varphi_L,\Gamma_L)$-module. (Because there is no action of the full group $\Gamma_L.$) Instead we can restrict the action to the subgroup $U\subset \Gamma_L$ which is still enough to make sense of the operator $Z \in D(U,K)$ and the complex $C_{\varphi,Z}(\Dfm(M^r))$ (resp. $C_{\varphi,Z}(\Dfm_n(M^r))$). We use a subscript $(-)_{\text{diag}}$ to emphasize that the action is given diagonally when ambiguity can arise.
\end{rem}
Another subtlety is the fact that the induced diagonal action of $D$ on $\Dfm(M^r),$ in particular, the diagonal action of $Z$ is harder to understand. It is neither given by $Z \otimes Z$ nor by $Z^i \otimes Z,$ where $Z^i$ denotes the action induced by the inversion on $U$. This problem already occurs in the cyclotomic case where $Z=\gamma-1$ and $$(\gamma-1)_{\text{diag}}(a \otimes m ) = \delta_{\gamma^{-1}}a \otimes \gamma m - a \otimes m$$
while 
$$(\gamma-1)a \otimes (\gamma-1)m = \delta_{\gamma^{-1}}a \otimes \gamma m - a \otimes \gamma m - \delta_{\gamma^{-1}}a \otimes m + a \otimes m.$$
\subsection{Coadmissibility of Iwasawa cohomology $C_{\Psi}(M)$ over a field.}	
For technical reasons that will become clear in the proof of Theorem \ref{thm:iwasawadfm} we require that the complex of $D$-modules $C_{\Psi}(M)$ has coadmissible cohomology groups to obtain a comparison between the Iwasawa cohomology and the cohomology of $\mathbf{Dfm}(M).$ Our Theorem \ref{thm:perfectTrianguline} asserts that this perfectness holds for so-called trianguline modules. We further explore conjecturally how the étale case can be incorporated into the picture. More precisely we show that it suffices to proof the statement in the étale case to conclude that it holds for every $(\varphi_L,\Gamma_L)$-module coming from $\cR_L.$ 
\begin{rem}\label{rem:Prüfer}
	The rings $\cR_K^+,\cR_K^{[r,1)},\cR_K,D(o_L,K),D(o_L,L)$ are Prüfer domains (i.e.\ every finitely generated ideal is invertible). In particular a module over the above rings is flat if and only if it is torsion-free and any finitely generated torsion-free module is projective.
\end{rem}
\begin{proof}
	Using the Fourier isomorphism \cite[Theorem 2.3]{schneider2001p} this is  \cite[Corollary 1.1.8]{Berger}.
\end{proof}
\begin{defn}
	We define the \textbf{heart} of $M$ as $$\mathcal{C}(M):=(\varphi_M-1)M^{\Psi=1}.$$ If there is no possibility of confusion we omit $M$ and simply write $\mathcal{C}:= \mathcal{C}(M).$
	For each $c \in K^{\times}$ we define a variant of the heart as $$C_c(M):=(\varphi_M-c)M^{c\Psi=1}.$$
\end{defn}
\begin{rem}
	\label{rem:exactheartsequence}
	$\mathcal{C}$ is a $D(\Gamma_L,K)$-submodule of $M^{\psi_{\mathrm{LT}}=0},$ in particular $\mathcal{C}$ is $D(\Gamma_L,K)$-torsion-free. Furthermore we have for every $c \in K^{\times}$ an exact sequence $$0 \to M^{\varphi=c} \xrightarrow{\iota} M^{c\Psi=1}\xrightarrow{\varphi-c} \mathcal{C}_c(M)\to 0,$$ where $\iota$ is the inclusion. 
\end{rem}
\begin{proof}
	$\mathcal{C}_c(M)$ is a $K[\Gamma_L]$-submodule since the actions of $\Gamma_L$ and $\varphi$ (resp. $\Psi$) commute and by continuity considerations it is also a $D(\Gamma_L,K)$-submodule of $M.$ Using $\Psi\circ \varphi = \id$ one concludes that $\mathcal{C}_c(M)$ is contained in $M^{\Psi=0}$ which is projective over $\cR_K(\Gamma_L)$ by \cite[Theorem 2.4.5]{RustamFiniteness} and, in particular, $D(\Gamma_L,K)$-torsion-free. The exactness of the sequence on the right is given by definition. It remains to see $M^{\varphi=c}\cap M^{c\Psi=1} = M^{\varphi=c}$ i.e. $M^{\varphi=c}\subset M^{c\Psi=1}$ for this purpose let $m \in M^{\varphi=c}$ then $cm = \varphi(m)$ implies $\Psi(\varphi(m)) = c\Psi (m) $ but then $c\Psi(m)=m.$
\end{proof}
The following lemma is a strengthening of \cite[Lemma 4.1.6]{KPX}. We elaborate on the proof for the convenience of the reader.
\begin{lem}
	\label{lem:finiteperfect} Let $r \in [0,1).$
	Any $\cR_K^{[r,1)}$-module $V$ of finite $K$-dimension admits a resolution of the form $$0 \to F_1 \to F_2 \to V \to 0,$$ where $F_i$ are finite free $\cR_K^{[r,1)}$-modules.
\end{lem}
\begin{proof} Let $v_1,\dots,v_d$ be a $K$-basis of $V.$ Consider the surjection $(\cR_K^{[r,1)})^d \to V$ mapping $e_i$ to $v_i.$ Since for each $i$ the elements $v_i,Tv_i,T^2v_i,\dots$ have to be linearly dependent we observe that there exists a polynomial $f \in K[T]$ such that the structure map $\cR_K^{[r,1)} \to \operatorname{End}(V)$ factors over $\cR_K^{[r,1)}/(f).$ Now take a factorisation of $f$ viewed as an element in $K\langle T\rangle$ of the form $ug,$ where $u$ is a unit and $g$ is a Weierstraß polynomial and rewrite this decomposition as $u g_1g_2,$ such that the zeroes of $g_1$ lie outside the annulus $[r,1)$ (i.e. $g_1$ becomes a unit in $\cR_K^{[r,1)}$) and the zeroes of $g_2$ are contained inside the annulus. Without loss of generality we assume $g_2=g.$ By the coadmissibility of $\cR_K^{[r,1)}$ and $g\cR_K^{[r,1)}\cong \cR_K^{[r,1)}$ we have $\cR_K^{[r,1)}/(g) = \operatorname{cok}(\cR_K^{[r,1)} \xrightarrow{g} \cR_K^{[r,1)}) = \varprojlim_s \operatorname{cok}(\cR_K^{[r,s]} \xrightarrow{g} \cR_K^{[r,s]}).$ For $s$ large enough (such that the zeroes of $g$ are contained in the annulus $[r,s]$) we have by \cite[3.3 Lemma 10]{bosch2014lectures} and the chinese remainder theorem $\cR_K^{[r,s]}/(g) \cong K \langle T\rangle/(g).$ In particular the limit stabilises for $s$ large enough and we obtain $\cR_K^{[r,1)}/(g) \cong K\langle T\rangle/(g).$ Recall that $K\langle T\rangle$ is a principal ideal domain by \cite[2.2 Corollary 10]{bosch2014lectures}.
	By the elementary divisor theorem we may find a free resolution of $V$ as a $K\langle T\rangle$-module of the form $$0 \to K\langle T\rangle^{d_1} \to K\langle T\rangle^{d_2} \to V\to 0$$ with some $d_i \in \N.$ Since $K\langle T\rangle$ is a principal ideal domain and $\cR_K^{[r,1)}$ is torsion free we get via base change along the flat map $K\langle T\rangle \to \cR_K^{[r,1)}$ a resolution
	$$0 \to (\cR_K^{[r,1)})^{d_1}  \to(\cR_K^{[r,1)})^{d_2} \to  \cR_K^{[r,1)}\otimes_{K\langle T\rangle} V\to 0.$$
	Because $V$ is a $\cR_K^{[r,1)}/(g)$-module we obtain $ \cR_K^{[r,1)}\otimes_{K\langle T\rangle} V = \cR_K^{[r,1)}/(g) \otimes_{K\langle T\rangle}V = K\langle T\rangle/(g)\otimes_{K\langle T\rangle}V  = V.$  We have therefore constructed the desired resolution of $V.$
\end{proof}
\begin{rem}
	\label{rem:finiteperfectapplied}
	Let $N$ be a not necessarily $L$-analytic $(\varphi_L,\Gamma_L)$-module over $\cR_K$ and let $n \geq n_0$ such that $\Gamma_n \cong o_L.$ Then \begin{enumerate}
		\item $N^{c\varphi=1}$ has finite $K$-dimension for any $c \in K^\times$.
		\item If $N$ is $L$-analytic and $c \in K^\times$ then $N^{c\varphi=1}[0]\in \mathbf{D}^b_{\text{perf}}(D(\Gamma_n,K)).$ 
		\item If $N$ is $L$-analytic and $c \in K^\times$ then $N/(c\psi_{\mathrm{LT}}-1)[0]$ belongs to $\mathbf{D}^b_{\text{perf}}(D(\Gamma_n,K)).$ 
	\end{enumerate} 
\end{rem}
\begin{proof}
	Restricting the residue pairing from Proposition \ref{prop:pairing}  to $N^{c\varphi=1}$ we see that the pairing factors over $\check{N}/(c\psi_{\mathrm{LT}}-1)$ which is finite-dimensional by Theorem \ref{thm:finitepsi}. If $N$ is $L$-analytic, then $N^{\varphi=1}$ and $N/(c\psi_{\mathrm{LT}}-1)$ carry natural $D(\Gamma_n,K)$-module structures and are finite dimensional over $K$ by the above (resp. Theorem \ref{thm:finitepsi}). Hence 2. and 3. follow from Lemma \ref{lem:finiteperfect} by transport of structure along $\cR_K^+ \cong D(\Gamma_n,K).$
\end{proof}

\begin{prop} Let $c \in K^{\times}.$
	\label{prop:psiperfchar} The following are equivalent. 
	\begin{enumerate}[(i.)]
		\item $C_{c\Psi}(M) \in \mathbf{D}^b_{\text{perf}}(D(U,K)).$
		\item $M^{c\Psi=1}$ is coadmissible and finitely generated as a $D(U,K)$-module.
		\item $M^{c\Psi=1}$ is finitely generated as a $D(U,K)$-module.
		\item $\mathcal{C}_c(M)$ is finitely generated as a $D(U,K)$-module.
	\end{enumerate}
\end{prop}
\begin{proof}
	(i.) $\implies$ (ii.) follows from the fact that finite projective modules are automatically coadmissible and the latter form an abelian category. Hence $M^{\Psi=1}$ is coadmissible as a cohomology group of a complex of coadmissible modules. Finite generation follows from \cite[Lemma 1.1.9]{Berger}. The implication (ii.) $\implies $(iii.) is trivial.  (iv.) follows immediately from (iii.) via the exact sequence from Remark \ref{rem:exactheartsequence}. Lastly assume (iv.) then $\mathcal{C}_c(M)$ is a torsion-free module which is finitely generated. Since $D(U,K)$ is a Prüfer-domain we conclude that $\mathcal{C}_c(M)$ has to be finitely generated projective by Remark \ref{rem:Prüfer}. From the exact sequence in Remark \ref{rem:exactheartsequence} and by Remark \ref{rem:finiteperfectapplied} 2.) we conclude that the bounded complex $C_{c\Psi}(M)$ has cohomology groups belonging to $\mathbf{D}^b_{\text{perf}}(D)$. Then \cite[\href{https://stacks.math.columbia.edu/tag/066U}{Tag 066U}]{stacks-project} implies that $C_{c\Psi}(M)$ itself belongs to $\mathbf{D}^b_{\text{perf}}(D)$.
	
\end{proof}

The following Lemma provides some flexibility concerning the constant $c.$

\begin{lem}
	\label{lem:constantirrelevantpsi}
	Let $c \in K^\times$ and define a character $\rho: L^\times \to K^\times$  by $\rho(\pi_L) = c$ and $\rho_{\mid o_L^{\times}}=1.$
	Then the identity induces a $\Gamma_L$-equivariant isomorphism. $$C_{\Psi}(M) \cong C_{c\Psi}(M(\rho)).$$
	Furthermore if $M$ is $L$-analytic then so is $M(\rho)$ and the isomorphism above is $D(\Gamma_L,K)$-equivariant.
\end{lem}
\begin{proof}
	Since the character $\rho$ is trivial on $o_L^{\times}$ the identity is $\Gamma_L$-equivariant. Furthermore we have $\varphi_{M(\rho)}(m) = c\varphi_M(m)$ and hence $\Psi_{M(\rho)} = c^{-1}\Psi_{M}$ which shows that the identity induces a morphism of complexes. The second part of the statement follows from the fact that $\rho$ does not change the $\Gamma_L$-action and hence it remains $L$-analytic on $M(\rho).$ The $D(\Gamma_L,K)$-equivariance follows from continuity.
\end{proof}
\begin{cor}
	\label{kor:modtperfect}
	Let $M$ be an $L$-analytic $(\varphi_L,\Gamma_L)$-module over $\cR_K$. Then the complex $C_{c\Psi}(M/t_{\mathrm{LT}}) = [M/t_{\mathrm{LT}}\xrightarrow{c\Psi-1} M/t_{\mathrm{LT}}]$ is perfect.
\end{cor}
\begin{proof} Since twisting by a character that is trivial on $o_L^\times$ preserves the property that $M$ is $L$-analytic
	by Lemma \ref{lem:constantirrelevantpsi} we may without loss of generality assume $c=1.$ Because the complex is bounded it suffices to show that the cohomology groups are perfect. The cohomology groups are precisely kernel and cokernel of $\Psi-1$ for which the statement has been shown in Lemma \ref{lem:psimodt}.
\end{proof}

	\subsection{Consistent complexes}
	When studying the Iwasawa cohomology complex $C_{\Psi}(M)$, the conceptual approach is to view the cohomology groups $H^i_{\mathrm{Iw}}(M)$ as coherent sheaves on $D(U,K).$ We do however not know if $M$ itself can be viewed as a sheaf on $D(U,K)$ in a suitable sense and we instead study the ``sheaf'' $D_n \mapsto C_{\Psi}(D_n \otimes_D M).$ We describe a framework for studying complexes of $D$ modules whose cohomology groups are coadmissible $D$ modules. 
	A similar situation was studied by Berthelot and Ogus (cf. Appendix B in  \cite{berthelot1978notes}) in the case where $A$ is a noetherian ring that is $I$-adically complete and $A_n := A/I^n.$ We adapt their setup for our purpose following \cite{pottharst2013analytic}. 
	One can view a projective system of $D_n$ modules as a sheaf on the ringed site $\mathbb{N}$ (where $\mathcal{O}(n) = D_n$) with the indiscrete topology (such that only isomorphism are coverings and thus every presheaf is a sheaf). In overblown terms given a projective system $(A_n)_n$ of rings with $A:= \varprojlim_n A_n$ we have a canonical morphism of topoi $f:\operatorname{Sh}(\mathbb{N},A_\bullet) \to \operatorname{Sh}(\operatorname{pt},A).$
	Where $f_*$ is just $\varprojlim_n$ and $f^*$ will be described below. 
	In order to describe the cohomology groups of complexes of sheaves on both sides we need to understand the respective derived functors. Fix a countable projective system $(A_n)_{n \in \N}$ of rings and denote by $A$ their limit.
	We denote by $\operatorname{mod}(\N,A)$ the abelian category of inverse systems $(M_n)_n$ of abelian groups indexed by $\N$ such that each $M_n$ is a $A_n$-module and the transition maps $M_{n+1} \to M_n$ are $A_{n+1}$-linear. Denote by $\mathbf{Rlim}$ the right-derived functor of the functor $(M_n) \mapsto \varprojlim_n M_n$ taking values in $\mathbf{D}(A).$ Observe that a morphism of complexes in $\operatorname{mod}(\N,A)$ is a quasi-isomorphism if and only if this is the case on every level and hence the projection to the $n$-th degree of the projective system induces a functor $\mathbf{D}(\operatorname{mod}(\N,A)) \to \mathbf{D}(A_n).$ We denote the image of $C \in \mathbf{D}(\operatorname{mod}(\N,A))$ by $C_n.$ 
	\begin{lem}\label{lem:rlimabstract}
		Given $C = (C_n^\bullet)_n \in \mathbf{D}(\operatorname{mod}(\N,A_n))$ we have a canonical distinguished triangle 
		$$\mathbf{Rlim} C \to \prod_n C_n^\bullet \to \prod_n C_n^\bullet \to \mathbf{Rlim}C[1]$$ in $\mathbf{D}(A).$ The map in the middle is given by $(c_n)_n \mapsto (c_n-f_{n+1}(c_{n+1}))_n,$ where $f_{n+1}\colon C_{n+1}\to C_n$. Its long exact sequence splits into short exact sequences $$0 \to \mathbf{R}^1\varprojlim_n H^{i-1}(C^\bullet_n) \to H^i(\mathbf{Rlim}C) \to \varprojlim_nH^{i}(C^\bullet_n) \to 0.$$
	\end{lem}
	\begin{proof}
		See \cite[\href{https://stacks.math.columbia.edu/tag/0CQD}{Tag 0CQD}]{stacks-project} together with \cite[\href{https://stacks.math.columbia.edu/tag/0CQE}{Tag 0CQE}]{stacks-project}.
	\end{proof}
	Now let $M$ be an $A$ module and consider $S(M):= (A_n \otimes_A M)_n.$ By taking the component wise base change of a morphism along $A \to A_n$ this defines a right-exact functor to $\operatorname{mod}(\mathbb{N},A)$ and we denote by $\mathbf{L}S:\mathbf{D}^-(A) \to \mathbf{D}^-(\operatorname{mod}(\mathbb{N},A))$ its left-derived functor. Here we remark that a complex in $\operatorname{mod}(\mathbb{N},A)$ is bounded above if for some $i_0$ the cohomology groups vanish in degree $i \leq i_0$ in every step of the projective system. By construction $(\mathbf{L}S C)_n\simeq A_n\otimes_A^{\mathbb{L}}C_n .$
	We have constructed functors $$\mathbf{D}^-(A)\substack{\xrightarrow{\mathbf{L}S}\\ \xleftarrow[]{\mathbf{Rlim}}}\mathbf{D}^-(\operatorname{mod}(\mathbb{N},A)).$$
	One can check that $\varprojlim$ and $S(-)$ are adjoint and hence by \cite[\href{https://stacks.math.columbia.edu/tag/0DVC}{Tag 0DVC}]{stacks-project} so are their derived functors (restricted to the respective $\mathbf{D}^-$). 
	Assume henceforth that each $A_n$ (but not necessarily $A$) is Noetherian. In this case it makes sense to speak of the full triangulated subcategory $\mathbf{D}_{ft}(A_n)$ of objects in  $\mathbf{D}(A_n)$ whose cohomology groups are $A_n$-finitely generated (cf. \cite[\href{https://stacks.math.columbia.edu/tag/06UQ}{Tag 06UQ}]{stacks-project}).
	\begin{defn}
		Let $C \in \mathbf{D}^-(\operatorname{mod}(\mathbb{N},A)).$ \begin{enumerate}
			\item We call $C$ \textbf{quasi-consistent} if $A_{n}  \otimes^\mathbb{L}_{A_{n+1}}C_{n+1}\to C_{n}$ is an isomorphism (in $\mathbf{D}({A_{n}})$) for every $n \in \N.$
			\item 	$C$ is called \textbf{consistent} if it is quasi-consistent and $C_n \in \mathbf{D}_{ft}(A_n).$
		\end{enumerate} 
		We denote by $\mathbf{D}^{-}_{con}(\operatorname{mod}(\mathbb{N},A))$ the full subcategory of $\mathbf{D}^{-}(\operatorname{mod}(\mathbb{N},A))$ of consistent objects.
		
	\end{defn}
	The following result is \cite[Corollary B.9]{berthelot1978notes}.
	\begin{rem}
		If $A$ is Noetherian and $I$-adically complete for some ideal $I,$ then $\mathbf{L}S$ and $\mathbf{Rlim}$ induce an equivalence of categories $$\mathbf{D}^{-}_{ft}(A) \cong \mathbf{D}^-_{con}(\operatorname{mod}(\mathbb{N},A)).$$
	\end{rem}
	We now specialise to the situation where $A$ is a Fréchet-Stein algebra.
	We denote by $\mathbf{D}^{-}_{\mathcal{C}}(A)$ the full triangulated subcategory of objects in the bounded above derived category whose cohomology groups are coadmissible $A$-modules. This makes sense by  \cite[\href{https://stacks.math.columbia.edu/tag/06UQ}{Tag 06UQ}]{stacks-project} because an extension of two coadmissible modules is again coadmissible and hence the category of coadmissible modules is a weak Serre-subcategory of the category of all $A$-modules.
	\begin{prop} \label{prop:equivalenceRlim}
		Let $A = \varprojlim A_n$ be a Fréchet-Stein algebra then the functor $S$ is exact and the adjoint pair $S \dashv \mathbf{Rlim}$ restricts to an equivalence of categories

		$$\mathbf{D}^-_{\mathcal{C}}(A) \cong \mathbf{D}^{-}_{con}(\operatorname{mod}(\mathbb{N},A))$$
	\end{prop}
	\begin{proof}
		By \cite[Remark 3.2]{schneider2003algebras} $A \to A_n$ is flat and hence $S$ is exact. Flatness also implies that the functors are well-defined since for a complex $C$ of modules with coadmissible cohomology we have $A_n \otimes_A H^i(C) \cong H^i (A_n \otimes_A C)$ and the left-hand side is finitely generated by assumption. On the other hand if we have $(C_n)_n$ representing an object in $\mathbf{D}^-_{con}(\operatorname{mod}(\mathbb{N},A))$ then due to flatness of $A_n$ over $A_{n+1}$ and quasi-consistency the natural morphism $A_n \otimes_{A_{n+1}}H^i(C_n) \to H^i(C_{n+1})$ is an isomorphism for every $n.$ The assumption that $(C_n)_n$ is consistent asserts that $(H^i(C_n))_n$ is a coherent sheaf in the sense of Schneider and Teitelbaum and thus $\varprojlim_{n}(H^i(C_n))$ is a coadmissible $A$-module. 
		The key observation of the proof is that $\lim^1(H^i(C_n))=0$ for any quasi-consistent $C$ and any $i$ by \cite[Theorem B]{schneider2003algebras}. 
		This applied to the exact sequence in Lemma \ref{lem:rlimabstract} shows that $H^i\mathbf{Rlim}(C_n)$ is coadmissible and hence the functor is well-defined. The same observation allows us to conclude that the natural maps (obtained from the adjunction) $S(\mathbf{Rlim}C_n) \to (C_n)_n$ and $M \to \mathbf{Rlim}(S(M))$ are quasi-isomorphism. 
		We have $$H^i(S(\mathbf{Rlim}C_m)) = (A_n \otimes_A H^i\mathbf{Rlim}C_m)_{n} =(A_n \otimes_A \varprojlim_{m}H^i(C_m))_n=(H^i(C_n))_n$$ using flatness in the first, Lemma \ref{lem:rlimabstract} in the second and \cite[Corollary 3.1]{schneider2003algebras} in the last equation.
		For the second quasi-isomorphism we have $$H^i(\mathbf{Rlim}(S(C))) = \varprojlim_n H^i(A_n \otimes_AC) = \varprojlim_n A_n \otimes_A H^i(C) = H^i(C)$$ using similar arguments and coadmissibility in the last equation.
		
	\end{proof}
	\subsection{Comparison to Iwasawa cohomology.}
	Proposition \ref{prop:equivalenceRlim} gives us the correct framework to describe a comparison between Iwasawa cohomology and analytic cohomology of the Lubin--Tate deformation. 
	\begin{rem} \label{rem:coadmissiblerlim}
		The projective system $(C_{\Psi,Z}(\mathbf{Dfm}_n(M)))_n$ defines a consistent object in $\mathbf{D}(\operatorname{mod}(\N,D_n)).$ In particular the cohomology groups \linebreak
		 $H^i(\mathbf{Rlim}(C_{\Psi,Z}(\Dfm_n(M))))$ are coadmissible $D$-modules for every $i$ and $$H^i(\mathbf{Rlim}(C_{\Psi,Z}(\Dfm_n(M)))) \cong \varprojlim_n H^i_{\Psi,Z}(\Dfm_n(M)).$$
	\end{rem}
	\begin{proof}
		Consistency follows from Theorem \ref{thm:perfect} together with the fact that $D$ is a Fréchet-Stein algebra. The latter cohomology groups are coadmissible by Proposition \ref{prop:equivalenceRlim}. The isomorphism follows from Lemma \ref{lem:rlimabstract} using again Theorem \ref{thm:perfect}.
	\end{proof}

	\begin{lem}\label{lem:finitedimtrivialU}
		Let $V$ be a finite dimensional $K$-linear $U$-representation. Then for $W = D_n \otimes_K V$ we have $H^0(U,W)=0$ with respect to the $U$-action via $\gamma(1 \otimes m) = \delta_{\gamma^{-1}} \otimes \gamma m.$ 
	\end{lem}
	\begin{proof} Fix a basis $w_1,\dots w_d$ of $W$ let $w \in W^U$ and write $w = \sum_{i = 1}^d \lambda_i \otimes w_i.$ Let $g \in U$ and define $G \in M_{\dim_KV}(K)$ via $gw_j = \sum_{i}G_{ij}w_i.$
		We compute $$gw = \sum_{j=1}^d (g^{-1}\lambda_j \otimes gw_j) = \sum_{j=1}^d \sum_{i=1}^d G_{ij}(g^{-1}\lambda_j \otimes w_i).$$ Since we assumed that $w$ is fixed by $g$ and the decomposition with respect to the basis $w_i$ is unique we conclude $$\lambda_j  = \sum_{i=1}^dG_{ji}g^{-1}\lambda_i.$$ Multiplying both sides by $g$ we see $g \lambda_j \in \operatorname{span}_K(\lambda_i).$ Since this works for any choice of $g$ we conclude that the $\lambda_i$ span a finite-dimensional $U$-stable subspace of $D_n.$ This is only possible if $\lambda_1  = \dots =\lambda_d =0.$ 
	\end{proof}
	\begin{lem}\label{lem:ZinjectiveDfm}
		Let $m \in \N$ then for the diagonal action of $U$ we have $H^0(U,D_m\hat{\otimes}_K M^r)=0.$ In particular the kernel of $Z$ acting diagonally is trivial.
	\end{lem}
	\begin{proof} It suffices to show that there are no non-trivial $U$-invariant elements.
		By Remark \ref{rem:sentheoryembedding} there exists $n \gg 0 $ such that $M^r$ can be embedded into $D_{\text{dif},n}^+(M^r)$ which is a projective finitely generated $(K \otimes_L L_n)\llbracket t_{\mathrm{LT}} \rrbracket$-module.
		We claim that we have an injection $$D_m\hat{\otimes}_KM^r \to \mathbb{D}:=D_m\hat{\otimes}_KD_{\text{dif},n}^+(M^r).$$ Since we do not know whether $\iota_n$ is strict we instead make use of \cite[1.1.26]{emerton2017locally} by rewriting $D_m \hat{\otimes}_K - = D_m^{(L)} \hat{\otimes}_L K\hat{\otimes}_K-$ with a suitable Banach algebra  $D_m^{(L)}$ over $L$ and using the associativity of projective tensor products from \cite[2.1.7 Proposition 7]{BGR} after reducing to the Banach case via \cite[Lemma 2.1.4]{Berger}.
		By applying again \cite[Lemma 2.1.4]{Berger} it suffices to show that $\mathbb{D}/t_{\mathrm{LT}}^k\mathbb{D}$ has no non-trivial $U$-invariants for each $k\geq0.$ Dévissage using the exact sequence $$0 \to \mathbb{D}/t_{\mathrm{LT}}\mathbb{D}\to\mathbb{D}/t^k_{\mathrm{LT}}\mathbb{D}\to\mathbb{D}/t^{k-1}_{\mathrm{LT}}\mathbb{D}\to0 ,$$ induction on $k$ and passing to the limit $\mathbb{D} = \varprojlim_k \mathbb{D}/t_{\mathrm{LT}}^k\mathbb{D}$ shows that it suffices to prove the statement for $\mathbb{D}/t_{\mathrm{LT}}\mathbb{D} = D_m \otimes_K (D_{\text{dif},n}^+(M^r)/t_{\mathrm{LT}}D_{\mathrm{dif},n}^+(M^r)),$  where we use Lemma \ref{lem:onablebasechange} and can omit the completion since the right-hand side is finite-dimensional Hausdorff. The statement now follows from Lemma \ref{lem:finitedimtrivialU}.
		
	\end{proof}
	\begin{lem}
		\label{lem:nakamuraseq}
		The natural map 
		\begin{align*}
			\mathfrak{I}:\Dfm(M^r) & \to M^r\\
			\lambda \otimes m &\mapsto \lambda m
		\end{align*} is surjective and its kernel is the image of $Z \in D_n$ (acting diagonally).
	\end{lem}
	\begin{proof}
		Surjectivity is clear by definition. Observe that for any $y \in \Dfm(M^r)$ and any $\gamma \in U$ the element $(\gamma-1)y$ lies in the kernel of $\mathfrak{I}.$ Since $Z$ lies in the closure of the augmentation ideal in $D$ we conclude that $\operatorname{Im}(Z)\subseteq \ker(\mathfrak{I}).$ In order to show $\ker(\mathfrak{I}) \subseteq \operatorname{Im}(Z)$ we will reduce to the case of elementary tensors via a series of technical arguments. We will show that an element of the form $\lambda \otimes m -1\otimes \lambda m$ belongs to the image of $Z$ and its preimage can be bounded for each norm defining the Fréchet topology of $\Dfm(M^r)$ which, in particular, implies strictness with respect to the Fréchet topology.   Note that the map $\mathfrak{I}$ admits a section $\mathfrak{S}:m \mapsto 1 \otimes m$ and any element $y$ of the kernel can be written as $z-\mathfrak{S}\mathfrak{I}(z)$ for some $z.$ As an intermediate step we consider an element of the form $\lambda \otimes m \in \ker(\mathfrak{I})$ with $\lambda \in D.$ Fix a $\cR^{[r,s]}$-module norm on $M^{[r,s]}$ and consider the tensor product norm induced by the norm on $D_m$ and said norm on $D_m \hat{\otimes}_KM^{[r,s]}$ for $m \in \N_0$ We will show that there exists a constant $C$ depending only on $m$ and $\norm{-}_{M^{[r,s]}}$ such that $\lambda \otimes m -1 \otimes \lambda m= Z_{diag}x$ and $\norm{x}\leq C\norm{\tilde{\lambda}\otimes m}_{D_{m,\Q_p}\hat{\otimes}_KM^{[r,s]}},$ where $\tilde{\lambda}$ is any lift of $\lambda$ in $D_{\Q_p,m}(U,K).$ Since $Z$ is injective by Lemma \ref{lem:ZinjectiveDfm} we obtain that $x$ is uniquely determined and hence satisfies this bound with respect to the quotient norm on $D_m.$ Let $\varepsilon = \sup_{\gamma \in U}\norm{\gamma-1}_{D_m}.$  Choose $n\gg0$ such that $\norm{\gamma-1}_{M^{[r,s]}}<\varepsilon$ for $\gamma \in \Gamma_n.$ 
		Before treating the general case assume that $\lambda$ belongs to $D(\Gamma_n,K).$ Fix a $\Z_p$-Basis $\gamma_i$ of $\Gamma_n$ and set $\mathbf{b}:= (\delta_{\gamma_i}-1)_i.$ By taking a preimage in $D_{\Q_p}(\Gamma_n,K)$ we can express $\lambda$ as a convergent series $$\lambda = \sum_{\mathbf{k} \in \N_0^d}a_k\mathbf{b}^\mathbf{k}$$
		We compute
		\begin{align*}\lambda \otimes m -1 \otimes \lambda m &=\sum_{\mathbf{k} \in \N_0^d}a_k(\mathbf{b}^\mathbf{k} \otimes m -1 \otimes \mathbf{b}^\mathbf{k}m)
		\end{align*}
		The terms in degree $\mathbf{k}=0$ cancel out and we shall estimate $\mathbf{b}^\mathbf{k} \otimes m -1 \otimes \mathbf{b}^\mathbf{k}m.$ Without loss of generality we can assume $\mathbf{k} \neq 0$ i.e.  at least one term $\delta_{\gamma_i}-1$ appears in $\mathbf{b}^{\mathbf{k}}.$ We will show that each summand is in the image of the diagonal $Z$-map and estimate the norm of its $Z$-preimage. We first explain a dévissage procedure to arrive at a situation where we estimate terms of the form 
		\begin{align}\label{eq:simpleform}
			(\gamma-1)a \otimes b- a \otimes (\gamma-1)b &= \gamma a\otimes b - a \otimes \gamma b\\
			& = (\gamma^{-1}-1)(a \otimes \gamma b)\\
			& = ZG(Z)(a \otimes \gamma b)
		\end{align}
		Where $\gamma \in \{\gamma_1,\dots\gamma_d\}$ and the assumption on the operator norm of $\gamma-1$ acting on $M^{[r,s]}$ asserts that $\norm{\gamma b}_{[r,s]} = \norm{b}_{[r,s]}$ and hence $$\norm{G(Z)(a \otimes \gamma b)}\leq C(\gamma) \norm{(\gamma-1)a \otimes b}$$ with $C(\gamma)$ depending on $\gamma$ and $[r,s].$
		Without loss of generality assume $\mathbf{k} = (k_1,\dots,k_d)$ with $k_1 \neq 0$ and let $\mathbf{k'} = (k_1-1,\dots,k_d).$ We rewrite
		\begin{align}\mathbf{b}^{\mathbf{k} } \otimes m - 1 \otimes \mathbf{b}^{\mathbf{k} }m
			&= (\gamma_1-1)\mathbf{b}^{\mathbf{k'} }\otimes m - 1 \otimes (\gamma_1-1)\mathbf{b}^{\mathbf{k'} }m\nonumber \\
			&= (\gamma_1-1)\mathbf{b}^{\mathbf{k'} }\otimes m- \mathbf{b}^{\mathbf{k'} }\otimes (\gamma_1-1)m \nonumber \\
			&+\mathbf{b}^{\mathbf{k'} }\otimes (\gamma_1-1)m - 1 \otimes(\gamma_1-1)\mathbf{b}^{\mathbf{k'} }m
		\end{align} We see that $ (\gamma_1-1)\mathbf{b}^{\mathbf{k'} }\otimes m- \mathbf{b}^{\mathbf{k'} }\otimes (\gamma_1-1)m$ is an expression of the form \eqref{eq:simpleform}. While the remainder i.e. $\mathbf{b}^{\mathbf{k'} }\otimes (\gamma_1-1)m - 1 \otimes(\gamma_1-1)\mathbf{b}^{\mathbf{k'} }m$ can be rewritten as $$\mathbf{b}^{\mathbf{k'} }\otimes m' - 1 \otimes\mathbf{b}^{\mathbf{k'} }m',$$ where $m' = (\gamma_i-1)m$ and by the assumption on the operator norm we have $$\norm{\mathbf{b}^{\mathbf{k'} }\otimes m'}\leq \norm{\mathbf{b}^{\mathbf{k} }\otimes m}.$$ The remainder vanishes as soon as $\mathbf{k'}=0$ and if $\mathbf{k'} \neq 0$ we may again find an index, that is not zero and apply the same procedure to, in the end, express $\mathbf{b}^\mathbf{k}\otimes m-1 \otimes \mathbf{b}^\mathbf{k} m$ as a finite sum of elements of the form from \eqref{eq:simpleform} more explicitly we can group them as $$\mathbf{b}^\mathbf{k}\otimes m-1 \otimes \mathbf{b}^\mathbf{k} m = \sum_{i=1}^d \sum_{j=1}^{e_i}(\gamma_i-1)a_{ij} \otimes b_{ij} -a_{ij} \otimes(\gamma_i-1)b_{ij} ,$$
		Where the elements $a_{ij},b_{ij}$ are not canonical and depend on the order in which we reduce the components of $\mathbf{k}$ in the inductive procedure. Nonetheless our construction asserts that each $(\gamma_i-1)a_{ij} \otimes b_{ij}$ is bounded above by $\mathbf{b}^k\otimes m.$ Using \eqref{eq:simpleform} we can write $$\mathbf{b}^\mathbf{k}\otimes m-1 \otimes \mathbf{b}^\mathbf{k} m = Zx_{\mathbf{k}},$$ where $$ \norm{x_{\mathbf{k}}}_{D_n \hat{\otimes} M^{[r,s]}} \leq C\norm{\mathbf{b}^\mathbf{k} \otimes m}_{D_n(\Q_p) \hat{\otimes} M^{[r,s]}}$$ with a suitable constant $C$ depending only on $[r,s]$. By Lemma \ref{lem:ZinjectiveDfm} $Z$ is injective and hence the element $x_k$ is uniquely determined. Note that a priori \eqref{eq:simpleform} produces constants for each $\gamma_i$ but we can take the supremum over all those constants. In particular $$x:= \sum_{\mathbf{k} \in \N_0^d} a_k x_{\mathbf{k}}$$ converges to an element in $D_n \hat{\otimes}_KM^r$ satisfying $Zx =  \lambda \otimes m - 1 \otimes \lambda m$ and $\norm{x}_{D_n \hat{\otimes} M^{[r,s]}} \leq C\norm{\lambda\otimes m}_{D_n(\Q_p) \hat{\otimes} M^{[r,s]}}$ and because this estimate holds for any preimage of $\lambda$ and $x$ is uniquely determined by the injectivity of $Z$ we also obtain $\norm{x}_{D_n \hat{\otimes} M^{[r,s]}} \leq C\norm{\lambda\otimes m}_{D_n \hat{\otimes} M^{[r,s]}}.$ Now assume $\lambda \in D(U,K)$ and
		decompose $\lambda= \sum_{g \in U/\Gamma_n} g\lambda_g$ with $\lambda_g \in D(\Gamma_n,K)$ we obtain 
		\begin{align}
			\label{eq:reductiontoknowncase}
			g\lambda_g \otimes m-1 \otimes g\lambda_gm 
			&= g\lambda_g\otimes m - \lambda_g \otimes gm + \lambda_g \otimes gm - 1 \otimes \lambda_g gm \nonumber\\
			& = (g^{-1}-1)_{diag}(\lambda_g \otimes gm ) + \lambda_g \otimes gm - 1 \otimes \lambda_g gm 
		\end{align}
		The case $\lambda_g \otimes gm - 1 \otimes \lambda_g gm$ has been treated above and 
		recall that $g^{-1}-1$ is divisible by $Z$ in $D$ (and hence also in every $D_m$). Combining the case treated above with an argument similar to the one after $\eqref{eq:simpleform}$ we conclude that there exists a unique element $x$ such that $\lambda \otimes m-1 \otimes \lambda m = Zx$ and a constant $C$ depending on $[r,s]$ such that $\norm{x}_{D_n \hat{\otimes}M^{[r,s]}}\leq C \norm{\lambda \otimes m}_{D_n \hat{\otimes}M^{[r,s]}}.$
		If we pass from $M^{[r,s]}$ to $M^{[r,s']}$ the $n$ that we chose before might no longer satisfy the desired bound on the operator norm and we might be required to pass to a subgroup satisfying the corresponding bound.
		Similarly if we pass from $D_m$ to $D_{m+1}$ we have $\norm{\gamma-1}_{D_{m+1}}\leq\norm{\gamma-1}_{D_{m}}$ and given $[r,s]$ we might need to pass to a smaller subgroup $\Gamma_{{\tilde{n}}}\subset U$ in order to achieve the estimate $\norm{\gamma-1}_{M^{[r,s]}}<\norm{\gamma-1}_{D_{m+1}}$ that we used in the preceding computations. 
		In both those cases a consideration analogous to \eqref{eq:reductiontoknowncase} leads to the existence of a constant $C(m,[r,s'])$ such that $$ \norm{x_{\mathbf{k}}}_{D_m \hat{\otimes} M^{[r,s']}} \leq C(m,[r,s'])\norm{\mathbf{b}^\mathbf{k} \otimes m}_{D_m(\Q_p) \hat{\otimes} M^{[r,s']}}$$ implying the convergence of $x$ with respect to the Fréchet topology and the estimate
		\begin{equation}\label{eq:frechetbound}\norm{x}_{D_m \hat{\otimes}M^{[r,s]}} \leq C(m,[r,s'])\norm{\lambda\otimes m}_{D_m \hat{\otimes}M^{[r,s]}}.
		\end{equation}
		Now consider a general element $y$ of the kernel of $\mathfrak{I}$ and write it as a convergent series $$y=\sum_{i=0}^\infty \lambda_i \otimes m_i.$$ Since it belongs to the kernel we have $$y = y-\mathfrak{S}(\mathfrak{I}(y)) = \sum_{i=0}^\infty (\lambda_i \otimes m_i-1 \otimes \lambda_im_i)$$ and the preceding discussion shows that each $\lambda_i\otimes m_i-1\otimes \lambda_im_i$ belongs to the image of $Z_{diag}$ and can be written as $$\lambda_i\otimes m_i-1\otimes \lambda_im_i = Zx_i.$$ Clearly if $x=\sum_{i=0}^\infty x_i$ converges then it satisfies $Zx =y.$  The convergence of $x$ with respect to the Fréchet topology defined by the tensor product norms on $D_m \hat{\otimes}_KM^{[r,s]}$ follows from the convergence of the series defining $y$ and the estimates \eqref{eq:frechetbound}.
	\end{proof}
	\begin{defn}
		Let $F$ be a topological $D$-module whose underlying $K$-vector space is Fréchet.
		We define $$D_n \hat{\otimes}_D F$$ as the completion of $D_n \otimes_DF$ with respect to the quotient topology of $D_n \otimes_{K}F.$
		For a $D$-module whose underlying $K$-vector space is an LF-space $E =\varinjlim E_n$ we define $$ D_n \tilde{\otimes}_D E := \varinjlim_r D_n \hat{\otimes}_D E^r.$$
	\end{defn}
	We do not know whether $D_n \tilde{\otimes}_D M$ is complete for a $(\varphi_L,\Gamma_L)$-module $M.$ Even if we knew that it was complete, we would run into subtleties concerning commutation of completion and cohomology since these spaces are in general not metrizable. 
	\begin{lem}
		We have for each $r \in [r_0,1)$ a strict exact sequence of $D$-modules 
		\begin{equation}
			\label{eq:seq1}
			0 \to D\hat{\otimes}_KM^r \xrightarrow{Z}D\hat{\otimes}_KM^r \xrightarrow{\mu} M^r\to 0,
		\end{equation}
		where the $D \hat{\otimes}_KM^r$ is viewed as $D$-module via the left tensor component and $\mu$ is given by $\mu(\lambda \otimes m)=\lambda m.$ Which induces for every $m$ compatible exact sequences 
		
		\begin{equation}
			\label{eq:seq3}
			0 \to D_m\otimes_D( D\hat{\otimes}_KM^r) \xrightarrow{\id \otimes Z}D_m\otimes_D( D\hat{\otimes}_KM^r) \xrightarrow{\id \otimes \mu} D_m \otimes_D M^r\to 0,
		\end{equation}
		\begin{equation}
			\label{eq:seq4}
			0 \to D_m\hat{\otimes}_KM^r \xrightarrow{Z}D_m\hat{\otimes}_KM^r \xrightarrow{ \mu} D_m \hat{\otimes}_D M^r\to 0
		\end{equation}
		and \begin{equation} \label{eq:seq5} 0 \to D_m\hat{\otimes}_KM \xrightarrow{ Z}D_m\hat{\otimes}_KM \xrightarrow{ \mu} D_m \tilde{\otimes}_D M\to 0.
		\end{equation}
		The sequence \ref{eq:seq4} is strict for every $m.$ 
	\end{lem}
	\begin{proof} One first checks that all maps are $D$-linear.
		The exactness of the first sequence was proved in Lemma \ref{lem:nakamuraseq}. On the one hand the operator $\mu$ is a continuous surjection of Fréchet spaces and hence strict. On the other hand, since the image of $Z$ is the kernel of a continuous map of Fréchet spaces, it is closed and hence itself a Fréchet space. By the same argument $Z$ is a continuous surjection onto $\operatorname{Im}(Z)$ and hence a homeomorphism on its image. The exactness of the second sequence is clear because $D_n$ is flat over $D.$
		The modules appearing in \eqref{eq:seq3} endowed with the quotient topology from a surjection from $D_m \otimes_{K}(D\hat{\otimes}_KM^r)$ are not necessarily Hausdorff (hence in particular not necessarily metrizable). Nonetheless the Hausdorff completion of $D_m \otimes_D(D\hat{\otimes}_KM^r)$ can naturally be identified with $D_m\hat{\otimes}_KM^r$ and we can argue using the maximal Hausdorff quotients \footnote{For a topological group $G$ the maximal Hausdorff quotient is defined as $G/\overline{\{1_G\}}.$} as follows.
		The Hausdorff quotient $\mathcal{X}$ of $D_m\otimes_D (D \otimes M^r)$ can be embedded into $D_m \hat{\otimes}_KM^r$ and the map induced by $\id \otimes Z$ is continuous and strict on the dense subset $D \hat{\otimes}_K M^r$ by the preceding \eqref{eq:seq1} and hence strict by Lemma \ref{lem:strictdensesubset}. As a quotient of a metrizable space by a closed space $\mathcal{X}$ is again metrizable and we conclude using Lemma \ref{lem:strictcompletion} that $$D_m\hat{\otimes}_KM^r \xrightarrow{ Z}D_m\hat{\otimes}_KM^r$$ is strict, injective and its cokernel is the Hausdorff completion of $D_m \otimes_DM^r$ since a strict map of Fréchet spaces has closed image. This gives the desired \eqref{eq:seq4}.
		Finally passing to direct limits produces the sequence \ref{eq:seq5} using \cite[Remark 1.1.13]{RustamFiniteness} to see that $\varinjlim_r D_m \hat{\otimes}_KM^r$ is complete. 
		
		\begin{defn} We define the complex $$C_\Psi(M): M \xrightarrow{\Psi-1} M$$ concentrated in degrees $1,2$ and we call its cohomology groups the \textbf{Iwasawa \linebreak cohomology of $M.$ } Analogously we define $C_{c\Psi}(M)$ and $C_{c\Psi }(M^r)$ for $r \in [0,1),$ $c \in K^{\times}.$ 
		\end{defn}

	\end{proof}
	\begin{lem} 	\label{lem:DDnquasiiso}
		Assume $C_{\Psi}(M)$ has coadmissible cohomology groups. Then
		
		the natural map $$D_n \otimes_D M \to D_n \tilde{\otimes}_DM$$ induces a quasi-isomorphism $$D_n \otimes_D C_{\Psi}(M)   = C_{\Psi}(D_n \otimes_D M)  \to C_{\Psi} (D_n \tilde{\otimes}_DM)$$
	\end{lem}
	\begin{proof}
		Taking cohomology commutes with colimits and thus it suffices to show that $$D_n \otimes_D M^r \to D_n \hat{\otimes}_DM^r$$ is a quasi-isomorphism for sufficiently large $r.$ The groups $D_n\hat{\otimes}_DM^r$ are metrizable and by Lemma \ref{lem:strictcompletion} and strictness of $\Psi-1$ taking kernels and cokernels commutes with completion i.e. $H^i_{\Psi}(D_n \otimes_DM^r)\cong H^i_{\Psi}(D_n \hat{\otimes}_DM^r)$ . By assumption the cohomology groups are coadmissible and thus $D_n\otimes_DH^i_{\Psi}(M)$ is $D_n$-finite, complete and any submodule is itself $D_n$ finite and complete because $D_n$ is Noetherian. This implies $D_n \otimes_DH^0_{\Psi}(M^r)\cong D_n \hat{\otimes}_DH^0_{\Psi}(M^r).$ Regarding $H^1$ the proof of Theorem \ref{thm:finitepsi} shows that for all sufficiently large $r$ we have that $M^r/(\Psi-1)$ is $K$-finite therefore $D$-finitely generated and is in addition Hausdorff by the strictness of $\Psi-1.$ We conclude $D_n \otimes_DH^1_{\Psi}(M^r)\cong D_n \hat{\otimes}_DH^1_{\Psi}(M^r)$ (for $r\gg 0$ as in the proof of Theorem \ref{thm:finitepsi}). 
	\end{proof}
	\begin{thm}[Comparison between Herr- and Iwasawa-cohomology]
		\label{thm:iwasawadfm}
		Consider the complexes
		$$(C_{\Psi,Z}(\Dfm_n(M)))_{n \in \N}$$
		and $$(C_{\Psi}(D_n \tilde{\otimes}_DM))_{n \in \N}$$
		in $\operatorname{mod}(\N,D).$ There is a canonical compatible family of morphisms
		$$\operatorname{Comp}_{\mathrm{Iw}}(C_{\Psi,Z}(\Dfm_n(M)))_{n \in \N} \to (C_{\Psi}(D_n \tilde{\otimes}_DM))_{n \in \N}$$ induced by the exact sequences \eqref{eq:seq3}. If the cohomology groups of $C_{\Psi}(M)$ are coadmissible as  $D$-modules we further obtain canonical compatible quasi-isomorphisms $$C_{\Psi,Z}(\Dfm_n(M))\simeq C_{\Psi}(D_n \tilde{\otimes}_DM),$$ which together with the maps induced by the natural maps $$C_{\Psi}(M)\to \mathbf{Rlim}(C_{\Psi}(D_n\otimes_D M))$$ and $$(C_{\Psi}(D_n\otimes_D M))_n \to (C_{\Psi}(D_n\tilde{\otimes}_D M))_n$$ induce an isomorphism in $\mathbf{D}(D)$  $$\mathbf{Rlim}C_{\Psi,Z}(\Dfm_n(M)) \simeq C_{\Psi}(M)$$
		and, in particular, isomorphisms $$\varprojlim_{n} H_{\Psi,Z}^i(\Dfm_n(M))\cong H^i_{\mathrm{Iw}}(M).$$
		
	\end{thm}
	\begin{proof} We have seen in \refeq{eq:seq4} that there exists a compatible family of surjections $\mu_n:\Dfm_n(M) \to D_n \tilde{\otimes}_DM$ whose kernel is the image of the diagonal $Z$-map. Rewrite $C_{\Psi,Z}(\Dfm_n(M))$ as a total complex of the double complex 
		\begin{equation}\label{eq:diagram}\begin{tikzcd}
				\Dfm_n(M) \arrow[r, "\Psi-1"] \arrow[d, "Z"] & \Dfm_n(M) \arrow[d, "-Z"] \\
				\Dfm_n(M) \arrow[r, "\Psi-1"]                & \Dfm_n(M)               
			\end{tikzcd}
		\end{equation}	
		and consider $C_{\Psi}(D_n \hat{\otimes}_DM)$ as the total complex of the ``double complex''
		$$\begin{tikzcd}
			0 \arrow[r, "\Psi-1"] \arrow[d, "Z"] & 0 \arrow[d, "-Z"] \\
			D_n \tilde{\otimes}_DM \arrow[r, "\Psi-1"]                & D_n \tilde{\otimes}_DM               
		\end{tikzcd}$$
		Applying $\mu_n$ in the lower row and zero in the upper row of \eqref{eq:diagram} induces a surjective morphism of double complexes	with kernel 
		$$\begin{tikzcd}
			\Dfm_n(M) \arrow[r, "\Psi-1"] \arrow[d, "Z"] & \Dfm_n(M) \arrow[d, "-Z"] \\
			\operatorname{Im}(Z) \arrow[r, "\Psi-1"]                & \operatorname{Im}(Z)               
		\end{tikzcd}.$$ The kernel double-complex has exact columns and hence acyclic total complex by \cite[Lemma 2.7.3]{Weibel}. Since passing to total complexes is exact we obtain the desired compatible family of quasi-isomorphisms.
		By Lemma \ref{lem:DDnquasiiso}
		the natural map $$D_n \otimes_D C_{\Psi}(M)\to C_{\Psi}(D_n \tilde{\otimes}_DM)$$ is a quasi-isomorphism.
		
		Composing the first quasi-isomorphism with the inverse of the latter gives isomorphism in $\mathbf{D}(D)$ $$\mathbf{Rlim}C_{\Psi,Z}(\Dfm_n(M)) \simeq \mathbf{Rlim}C_{\Psi}(D_n\otimes_DM).$$
		Using the coadmissibility assumption on $C_{\Psi}(M)$ by Proposition \ref{prop:equivalenceRlim} the natural map $$C_{\Psi}(M) \to \mathbf{Rlim}((C_{\Psi}(D_n\otimes_DM))_n)$$ is an isomorphism in $\mathbf{D}(D).$ 
		
		By Remark  \ref{rem:coadmissiblerlim} we have  $H^i(\mathbf{Rlim}C_{\Psi,Z}(\Dfm_n(M))) \cong \varprojlim_n H^i_{\Psi,Z}(\Dfm_n(M))$ and putting everything together we obtain $$H^i(\mathbf{Rlim}C_{\Psi,Z}(\Dfm_n(M)))\cong \varprojlim_n H^i_{\Psi,Z}(\Dfm_n(M)) \cong \varprojlim_n D_n\otimes_D H^i_{\Psi}(M)$$ and by the coadmissibility assumption $$ \varprojlim_n D_n\otimes_D H^i_{\Psi}(M)\cong  H^i_{\Psi}(M).$$ 
	\end{proof}
		\section{Explicit results in the rank one and trianguline case}
		In this section we study $M^{\Psi=1}$ for rank one modules of the form $\cR_K(\delta).$ The ideas are based on \cite{colmez2008representations}, \cite{colmez2016representations}, \cite{FX12} and \cite{chenevier2013densite}. 
		Some small adjustments are required in order to incorporate \cite[Theorem 2.4.5]{RustamFiniteness}. We are mostly interested in showing that $M^{\Psi=1}$ is finitely generated and coadmissible.

		\begin{defn}
			An $L$-analytic $(\varphi_L,\Gamma_L)$-module over $\cR_K$ is called trianguline if it is a successive extension of modules of the form $\cR_K(\delta)$ with locally $L$-analytic characters $\delta\colon L^\times \to K^\times.$ 
		\end{defn}
		Fix as usual a subgroup $U\subset \Gamma_L$ isomorphic to $o_L.$ Recall that the space  $C^{an}(U,K)$ of locally $L$-analytic functions $U \to K$ is reflexive and its strong dual is by definition $D(U,K).$ On the other hand $\cR_K^+$ is reflexive and its strong dual is $\cR_K/\cR_K^+$ via the residue pairing from Proposition \ref{prop:pairing}. 
		\begin{lem} \label{lem:Colmezsequence}
			By transport of structure along $\cR_K^+ \cong D(U,K)$ we obtain a $(\varphi_L,U)$-semi-linear strict exact sequence $$0 \to D(U,K) \to \cR_K \to C^{an}(U,K) \otimes \chi^{-1} \to 0,$$
			where $\chi (\pi) = \frac{\pi}{q}$ and $\chi(a)=a$ for $a \in o_L^{\times}.$
			
		\end{lem}
		\begin{proof}
			See \cite[Théorème 2.3]{colmez2016representations}.
		\end{proof} 
		\begin{rem} Let $\delta: L^{\times} \to K^{\times}$ be a locally $L$-analytic character and $\cR_K(\delta) = \cR_Ke_{\delta}$ the corresponding free rank $1$ module. Then $\cR_K^+(\delta):= \cR_K^+e_{\delta}$ is a $(\varphi_L,\Gamma_L)$-stable submodule and fits into a short exact sequence  
			$$0 \to \cR_K^+(\delta) \to \cR_K(\delta) \to C^{an}(U,K) \otimes \chi^{-1}\delta \to 0,$$
			where $\chi (\pi) = \frac{\pi}{q}$ and $\chi(a)=a$ for $a \in o_L^{\times}.$
		\end{rem}
		\begin{proof}
			The stability follows because the image of $\delta$ is contained in $(\cR_K^+)^{\times}.$ The sequence is obtained by twisting the sequence from Lemma \ref{lem:Colmezsequence}.
		\end{proof}
		\begin{lem}
			\label{lem:mahler}
			Any element $\zeta \in C^{an}(o_L,K)$ admits a unique expansion of the form $$\zeta = \sum_{k \geq 0} a_k \begin{bmatrix}x\\k \end{bmatrix},$$ where $\begin{bmatrix}x\\k \end{bmatrix}\colon o_L \to K$ is the polynomial function in $x$ defined via $$\eta(x,T) = \sum_{k\geq0} \begin{bmatrix}x\\k \end{bmatrix}T^k$$ and the coefficients satisfy $\lim_{k \to \infty} \abs{a_k}r^k $ for some $r>1.$
		\end{lem}
		\begin{proof}
			See \cite[Theorem 4.7]{schneider2001p}.
		\end{proof}
		The following lemmas are essentially an $L$-analytic version of \cite[Lemme 2.9]{chenevier2013densite}. We let $\Psi$ act on $C^{an}(o_L,K)$ as $\Psi(f)(x) = f(\pi_L x).$
		\begin{lem} 
			\label{lem:CherbonnierLemma}
			Let $\alpha \in K^\times.$ Denote by $x \in C^{an}(o_L,K)$ the function $x \mapsto x.$ 
			\begin{enumerate}
				\item If $N$ is such that $\abs{\alpha\pi_L^N}<1$ then $1-\alpha\Psi$ is bijective on $x^{N}C^{an}(o_L,K).$
				\item The cokernel of $1-\alpha \Psi$ acting on $\cR_K^+$ is at most one-dimensional over $K.$
				\item The cokernel of the inclusion $(\cR_K^+)^{\alpha\Psi=1} \to \cR_K^{\alpha\Psi=1}$ is finite dimensional over $K.$
			\end{enumerate}
		\end{lem}
		\begin{proof} For $h \in \N_0$ denote by $C^{an}_h$ the subspace of functions that are globally analytic on $a+\pi_L^ho_L$ for each $a \in o_L/\pi_L^ho_L.$ Every element $f \in C^{an}(o_L,K)$ belongs to some $C^{an}_h$ and from the definitions one has $\Psi(x^NC^{an}_h) \subset x^NC^{an}_{h-1}$ for $h \geq 1.$ If $f \in x^NC^{an}_0$ then expanding $f$ as a power series yields
			\begin{equation} \label{eq:estimateCan}\abs{\Psi(f)}_{o_L} \leq \abs{\pi_L^N}\abs{f}_{o_L}\end{equation} 
			with respect to the $\sup$-norm on $o_L.$
			Let $f \in x^NC^{an}_h$ then $\Psi^h(f) \in C^{an}_0$ and the estimate \eqref{eq:estimateCan} together with the assumption on $N$ shows that the series $$\sum_{k=0}^\infty \alpha^m\Psi^m$$ converges to an inverse of $1-\alpha\Psi$ on $x^NC^{an}_h.$ The claim follows because $C^{an}(o_L,K) = \varinjlim_h C^{an}_h.$\\	
			For the second statement choose $N\gg 0 $ such that $$F:=\sum_{i=0}^{\infty} \alpha^{-i}\varphi^i$$ converges to a continuous operator on $T^N\cR_K^+.$ The existence of such $N$ can be seen using that $T^N$ tends to zero as $N \to \infty$ and  that $\varphi_L$ is contractive with respect to the norms $\abs{-}_{[0,r]}$ for any $0<r<1$ as a consequence of \cite[Remark 1.5.2]{RustamFiniteness}.  Given $h \in T^n\cR_K^+$ we observe that $(1-\alpha\Psi)(F(h)) = \alpha \Psi(h).$ Using that $\Psi$ is surjective on $\cR_K^+,$ we can deduce that $h$ belongs to the image of $1-\alpha\Psi.$ By writing $\cR_K^+ = \bigoplus_{k=0}^{N-1}Kt_{\mathrm{LT}}^k \oplus T^K\cR_K^+$ we conclude that it remains to show that at most one $t_{\mathrm{LT}}^k$ is not contained in the image of $1-\alpha\Psi.$ Because $\Psi(t_{\mathrm{LT}}^i) = \pi_L^{-i}t_{\mathrm{LT}}^i$ there can be at most one $i_0$ such that $(1-\alpha\Psi)(t_{\mathrm{LT}}^{i_0}) =0.$ For every $i\neq i_0$ we have $(1-\alpha\Psi)(t^i_{\mathrm{LT}}) = (1-\alpha \pi_L^{-i})t^i_{\mathrm{LT}}$ with $(1-\alpha \pi_L^{-i}) \in K^\times.$ \\
			For the third statement observe that Lemma \ref{lem:Colmezsequence} and the Snake Lemma gives a short exact sequence 
			$$0 \to (\cR_K^+)^{\alpha\Psi=1} \to \cR_K^{\alpha\Psi=1} \to C^{an}(U,K)^{\alpha\frac{q}\pi\Psi=1}.$$
			The image of the last map is finite dimensional by 1.) because any element in $C^{an}$ can be written as a sum of a polynomial of degree $\leq N-1$ and an element of $x^NC^{an}(o_L,K)$ by Lemma \ref{lem:mahler}. 
		\end{proof}
		\begin{lem} \label{lem:pluspart} Let $\delta: L^{\times} \to K^{\times}$ be a locally $L$-analytic character and $M:=\cR_K(\delta)$ the corresponding free rank $1$ module. Let $M^+:= \cR_K^+(\delta).$
			\begin{enumerate}
				
				\item We have $(M^+)^{\psi=0} = D(\Gamma_L,K)\eta(1,T)\varphi(e_{\delta}).$ In particular $(M^{+})^{\Psi=0}$ is free of rank $1$ over $D(\Gamma_L,K).$
				\item The map $1-\varphi\colon (M^+)^{\Psi=1} \to (M^+)^{\Psi=0}$ has finite dimensional kernel and co-kernel.
			\end{enumerate}
		\end{lem}
		\begin{proof}1. follows from the explicit description in \cite[Corollary 2.4.7]{RustamFiniteness}.
			For 2. let $\eta(1,T)\varphi(m) \in (M^{+})^{\Psi=0}$ with some $m = fe_{\delta} \in M^+.$ Let $N \in \N$ and write $f = f_0 + T^Ng$ according to the decomposition $\cR_K^+  = \bigoplus_{k=0}^{N-1} t_{\mathrm{LT}}^k+T^N\cR_K^+.$ Choosing $N$ large enough (like in the proof of Lemma \ref{lem:CherbonnierLemma}) we can ensure that $$h:=\sum_{k=0}^\infty  \delta(\pi_L)^{k}\varphi^k(\eta(1,T)\varphi((T^Ng))$$ converges independently of $g \in \cR_K^+.$  The element $m':= he_{\delta}\in M^+$ satisfies $(\varphi-1)(m') = \eta(1,T)(\varphi(T^Ng)e_{\delta}).$ In particular $(\varphi_L-1)(m') \in M^{\Psi=0}$ and hence necessarily $\Psi(m') =m'.$ Finite dimensionality of the kernel is immediate from Remark \ref{rem:finiteperfectapplied}. For the codimension of the image $(M^{+})^{\Psi=1}$ in $(M^+)^{\Psi=0}$ our proof thus far shows that any element in $\eta(1,T)(\varphi(T^NM^+))$ lies in the image of $\varphi_L-1.$ We use the analogue of the decomposition from \cite[Proposition 2.2.6]{RustamFiniteness} for $M^+$ to conclude that the codimension of the image is at most $N[\Gamma_L:\Gamma_1] = N(q-1).$
		\end{proof}
		
		\begin{prop} \label{prop:rank1perf}
			Let $\delta \colon L^{\times} \to K^\times$ be a locally $L$-analytic character and let $M=\cR_K(\delta)$ then $M^{\Psi=1}$ is a finitely generated coadmissible $D(U,K)$-module of rank \mbox{$[\Gamma_L:U].$} In particular, $C_{\Psi}(M)$ is perfect.
		\end{prop}
		\begin{proof}
			Applying Lemma \ref{lem:CherbonnierLemma} and Lemma \ref{lem:pluspart} we conclude that $M^{\Psi=1}$ fits into an exact sequence $$0 \to (M^+)^{\Psi=1} \to M^{\Psi=1} \to V \to 0,$$ with a $D(U,K)$-module $V$ whose underlying $K$-vector space is finite dimensional. By Lemma \ref{lem:finiteperfect} $V$ is coadmissible and evidently torsion as a $D(U,K)$-module. Because the category of coadmissible module is abelian we conclude that it suffices to show that $(M^+)^{\Psi=1}$ is coadmissible of the desired rank. For $M^+$ we have by Lemma \ref{lem:pluspart} an exact sequence of the form $$0 \to V_1 \to (M^+)^{\Psi=1} \xrightarrow{\varphi-1}(M^+)^{\Psi=0} \to V_2 \to 0$$ with two $D(U,K)$-modules whose underlying $K$-vector-spaces are finite-dimensional. From this exact sequence and Lemma \ref{lem:pluspart} it is clear that the rank is precisely $[\Gamma_L:U].$  
			Again $V_2$ is coadmissible by Lemma \ref{lem:finiteperfect} and the image of $(M^+)^{\Psi=1}$ is the kernel of a map between coadmissible modules and hence coadmissible. By \cite[Lemma 3.6]{schneider2003algebras} the image is closed in the canonical topology. Because $(M^{+})^{\Psi=0}$ is finitely generated projective we obtain that $(1-\varphi)((M^+)^{\Psi=1}) $ is finitely generated by  \cite[Lemma 1.1.9]{Berger}.
			Now the short exact sequence $$0 \to V_1 \to (M^+)^{\Psi=1} \to (1-\varphi)((M^+)^{\Psi=1})  \to 0$$ proves that $(M^+)^{\Psi=1}$ is finitely generated and coadmissible. Perfectness follows from Proposition \ref{prop:psiperfchar}.
		\end{proof}
		
		\begin{thm}
			\label{thm:perfectTrianguline}
			Let $M$ be a trianguline $L$-analytic $(\varphi_L,\Gamma_L)$-module over $\cR_K.$ Then $C_{c\Psi}(M)$ is a perfect complex of $D(\Gamma_L,K)$-modules for any constant $c \in K^{\times}.$
		\end{thm}
		\begin{proof} By applying Lemma \ref{lem:FlatAlgebraProjectiveResolution} to the cohomology groups of $C_{\Psi}(M)$ we conclude that perfectness as a complex of $D(\Gamma_L,K)$-modules follows from perfectness as a complex of $D(U,K)$-modules.
			For $c=1$ this is a corollary of Proposition \ref{prop:rank1perf}. Because twisting by a character preserves the property of $M$ being trianguline the general statement follows from Lemma \ref{lem:constantirrelevantpsi}.
		\end{proof}

\section{Perfectness of $C_{\Psi}(M).$}
In the cyclotomic case the perfectness of $C_{\Psi}(M)$ as a complex of $D(U,K)$-modules is obtained by an inductive procedure from the étale case. In the cyclotomic étale case the heart $C=(1-\varphi)(M)$ is free over $D_{\Q_p}(\Gamma_{\Q_p},\Q_p)$ of the same rank as $M$ (cf. \cite[Proposition V.1.18]{colmez2010representations}) and an induction over the Harder--Narasimhan slopes implies the general result. 
In our situation we run into several problems. Most notably the corresponding result concerning the heart is not known. Furthermore the theory of slope filtrations makes heavy use of the Bézout property of the coefficient rings. By passing to the large extension $K/L$ we leave the situation of Kedlaya's slope theory and also do not have the equivalence of categories at our disposal for modules of slope zero. We interpret the passage to $K$ as a technical tool to understand $(\varphi_L,\Gamma_L)$-modules coming from $\cR_L$ employed for example in \cite{Berger} and \cite{FX12}. Hence we put a special emphasis on those coming from $\cR_L.$ We show that the induction over Harder--Narasimhan slopes as in \cite{KPX} works in essentially the same way. To do so we require some technical lemmas concerning the $(\Psi,\nabla)$-cohomology of $M$ that might be interesting in their own right.
\begin{con} \label{con:perfectfromL}
	
	Let $M_0$ be an $L$-analytic $(\varphi_L,\Gamma_L)$-module over $\cR_L.$ Then the complex $C_{\Psi}(K \hat{\otimes}_LM_0)$ is a perfect complex of $D(\Gamma_L,K)$-modules.
\end{con}
So far we have not concerned ourselves with the étale picture. For $V \in \operatorname{Rep}_{o_L}(G_L)$ let $H^i_{\mathrm{Iw}}(L_{\infty}/L,V):=\varprojlim_{L\subset F\subset L_{\infty}}H^i(F,V),$ where the limit is taken along the corestriction maps and $F$ runs through the finite intermediate extensions.   Schneider and Venjakob showed the following 
\begin{thm}
	\label{thm:etalecase}
	Let $V \in \operatorname{Rep}_{o_L}(G_L)$  and let $M_{\mathrm{LT}}$ be the étale Lubin--Tate-$(\varphi_L,\Gamma_L)$-module over $\mathbf{A}_L$ attached to $V(\chi_{\mathrm{LT}}\chi_{cyc}^{-1})$ endowed with its integral $\psi$-operator (i.e. $\psi\circ \varphi =\frac{q}{\pi}$).
	\begin{enumerate}
		\item $H^i(L_{\infty}/L,V)$ vanishes for $i \neq 1,2.$
		\item $H^2(L_{\infty}/L,V)$ is $o_L$-finitely generated. 
		\item $H^1(L_{\infty}/L,V)$ is $\Lambda:=o_L\llbracket \Gamma_L \rrbracket$-finitely generated.
		\item $H^i_{\mathrm{Iw}}(L_{\infty}/L,V)$ is computed by the complex $$M_{\mathrm{LT}} \xrightarrow{\psi-1}M_{\mathrm{LT}}$$ concentrated in degrees $1$ and $2$. 
	\end{enumerate}
\end{thm}
\begin{proof}
	See \cite[Lemma 5.12 and Theorem 5.13]{SV15}.
\end{proof}
In order to connect this result to the present situation assume first that $M_{\mathrm{LT}}$ is $L$-analytic and thus a fortiori overconvergent. An analogue of \cite[Proposition III.3.2]{cherbonnier1999theorie} (cf. \cite[Appendix]{SchneiderVenjakobRegulator}) allows us to view $M_{\mathrm{LT}}^{\psi=1}[1/p]$ as a $\Lambda$-submodule of $\mathbb{D}^{\dagger}(V)$\footnote{$\mathbb{D}^{\dagger}(V)$ was defined in \ref{def:dagger}.} and thus as a $\Lambda$-submodule of $\mathbb{D}^{\dagger}_{\text{rig}}(V)^{\psi=1}.$ Let $M_{\text{rig}}:=\mathbb{D}^{\dagger}_{\text{rig}}(V).$ Since $M_{\text{rig}}$ is $L$-analytic we obtain that $M_{\text{rig}}^{\psi=1}$ is even a $D(\Gamma_L,L)$-module and thus a natural map $$D(\Gamma_L,L) \otimes_\Lambda M_{\mathrm{LT}}^{\psi=1}\xrightarrow{\text{comp}} {M_{\text{rig}}}^{\psi=1}.$$
This leads us to the following natural conjecture.
\begin{con} \label{con:comparison}
	Let $M_{\mathrm{LT}}=\mathbb{D}(V)$ be an étale $L$-analytic $(\varphi_L,\Gamma_L)$-module over $\mathbf{A}_L$ and let $M_{\text{rig}} := \mathbb{D}^{\dagger}_{\text{rig}}(V)$ then the natural map $$D(\Gamma_L,L) \otimes_\Lambda M_{\mathrm{LT}}^{\psi=1}\xrightarrow{\text{comp}} {M_{\text{rig}}}^{\psi=1}$$ is surjective.
\end{con}
In the classical case the map above is bijective by \cite[Proposition V.1.18]{colmez2010representations}. We will show that the surjectivity is sufficient for Conjecture \ref{con:perfectfromL}. We expect the map to not be injective as soon as $L\neq \QQ_p$ (cf. section \ref{sec:etalecomp}). 
Observe that Theorem \ref{thm:finitepsi} works over any base field and hence $M_{\text{rig}}/(\psi-1)$ is finite $L$-dimensional. This implies that $\psi-1$ is strict by Lemma \ref{lem:strictcheck} and because $K$ is an $L$-Banach space of countable type over $L$ we conclude using Lemma \ref{lem:onablebasechange} $$(M_{\text{rig}}\hat{\otimes}_{L}K)^{\psi=1} =(M_{\text{rig}}^{\psi=1})\hat{\otimes}_{L}K.$$
The results of Schneider and Venjakob in the étale case suggest that $\psi-1$ is the ``correct'' operator to study Iwasawa cohomology. 
This leads us to believe that the complex defined using $\psi$ is well-behaved while in order to obtain a quasi-isomorphism to the $(\varphi_L,Z)$-complex we need to work with the $\Psi$-complex with the left-inverse operator. Our philosophy is that the $c\Psi$-complex for some constant $c$ can be reinterpreted as the $\Psi$-complex of a module twisted by a (non-étale) character. In particular, if the analogue of Conjecture \ref{con:perfectfromL} concerning the $\psi$-complex is true for all analytic (not necessarily étale) $(\varphi_L,\Gamma_L)$-modules coming from $\cR_L$, then Lemma \ref{lem:constantirrelevantpsi} asserts that Conjecture \ref{con:perfectfromL} itself is true.

Let $\cR_L(x)$ be the rank one module attached to the character $\id\colon L^\times \to L^\times.$
We require an analogue of the fact that there exists an isoclinic module $E$ of slope $1/d$ that is a successive extension of $\cR_L(x)\cong\cR_Lt_{\mathrm{LT}}$ by $d-1$ copies of $\cR_L$ requiring in addition that the latter module is $L$-analytic. This allows us to argue inductively over the slopes of a module. While the $(\varphi_L,Z)$-cohomology enjoys a number of nice properties, its biggest downfall is the fact that it can only be defined over the large field $K$ that is not discretely valued and hence $\cR_K$ does not fit into the framework of Kedlayas slope theory.  In order to translate known results from the étale case to more general modules we use the Lie-algebra cohomology of $M$ that can be defined over $L$ using either the operators $(\nabla,\varphi-1)$ or $(\nabla,\Psi-1).$ We do not know whether they give the same cohomology groups in degree $2$ and since the cokernel of $\Psi-1$ is better behaved than that of $\varphi_L-1,$ we shall work with the $\Psi$-version.

\begin{defn} Let $M$ be an $L$-analytic $(\varphi_L,\Gamma_L)$-module over $\cR_L.$ Let $\nabla$ be the operator corresponding to $1 \in L \cong \operatorname{Lie}(\Gamma_L)$. We define the complex 
	$$C_{\text{Lie}}(M):= [M \xrightarrow{(\nabla,\Psi-1)}  M\oplus M \xrightarrow{(\Psi-1)(\operatorname{pr}_1)-\nabla(\operatorname{pr}_2)}M]$$
	concentrated in degrees $[0,2].$
	Denote by $H^i_{\text{Lie}}(-)$ the cohomology groups of the complex $C_\text{Lie}(-)$ and define $H^i_\text{an}(M):=H^i_{\text{Lie}}(M)^{\Gamma_L}.$
	Analogously we define $H^i_{\text{an}}(M/t_{\mathrm{LT}}M).$
\end{defn}
\begin{rem}
	One can show that the residual $\Gamma_L$-action on $H^i_{\text{Lie}}(M)$ is discrete (i.e.\ every element has open stabiliser). Since the cohomology groups are $L$-vector spaces we can deduce that $H^i_{\text{an}}(M)$ takes short exact sequences to long exact sequences in cohomology. 
\end{rem}
\begin{rem}
	Analogously one can define a version with $\varphi-1$ instead of $\Psi-1.$ Copying the proof of \cite[Proposition 4.1]{FX12} together with the comparison isomorphism in \cite[Theorem 2.2.2]{BFFanalytic} shows that the $H^0$ and $H^1$ agree for the $\varphi_L$ and $\Psi$-versions and agree with the corresponding cohomology groups for $L$-analytic cocycles studied by \cite{BFFanalytic}.
\end{rem}
The main reason to use the $\Psi$-version is the following Lemma.
\begin{lem} Let $M$ be an $L$-analytic $(\varphi_L,\Gamma_L)$-module then
	$H^2_{Lie}(M)$ is finite dimensional and Hausdorff.
\end{lem}
\begin{proof}
	The continuous map $\Psi-1\colon M \to M$ has finite dimensional cokernel by Theorem \ref{thm:finitepsi}. The statement follows from Lemma \ref{lem:strictcheck}.
\end{proof}
\begin{lem}
	\label{lem:h2vanish}
	For the $L$-analytic rank one $(\varphi_L,\Gamma_L)$-modules $\cR_L(x^i)$ we have
	$$H^2_{\text{an}}(\cR_L(x^i)) = 0$$ for every $i \in \Z.$
\end{lem}
\begin{proof}
	Since $H^2_{\text{Lie}}(\cR_L(x^i))$ is finite-dimensional Hausdorff, every linear form is continuous. 
	Suppose $$H^2_{\text{Lie}}(\cR_L(x^i)) =(\cR_L(x^i)/(\nabla,(\Psi-1)))$$ admits a (continuous) functional $H^2_{\text{Lie}}(\cR_L(x^i)) \to L,$ then pre-composing it with the canonical map $\cR_L(x^i) \to H^2_{\text{Lie}}(\cR_L(x^i))$ gives a continuous functional $\cR(x^i) \to L$  which by the duality described in  Proposition \ref{prop:pairing} corresponds to a unique element in $m \in \cR_L(x^{-i})(\chi_{\mathrm{LT}})$ killed by $\varphi_L-1$ and $\nabla_{\iota},$ where $\nabla_{\iota}$ denotes the adjoint of $\nabla.$ A computation using that the adjoint of $\gamma \in \Gamma_L$ is $\gamma^{-1}$ shows $\nabla_{\iota} = -\nabla$ and that $m$ is also killed by $\nabla$. Hence $m$ is an element in $\cR_L(x^{-i})(\chi_{\mathrm{LT}})^{\nabla=0,\varphi=1}$, but $\cR_L(x^{-i})(\chi_{\mathrm{LT}})^{\nabla=0,\varphi=1}=0$ by  \cite[Proposition 5.6]{FX12} and thus $m=0$. We conclude that $H^2_{\text{Lie}}(\cR_L(x^i))$ admits no non-zero functionals and therefore has to be zero.
\end{proof}
\begin{lem}
	\label{lem:explicitcomputation}
	Let $M$ be an $L$-analytic $(\varphi_L,\Gamma_L)$-module over $\cR_L.$ Then:
	\begin{enumerate}
		\item There is a canonical bijection between $H^1_{\text{an}}(M)$ and isomorphism classes of extensions $$0 \to M \to E\to \cR_L \to 0$$ of $L$-analytic $(\varphi_L,\Gamma_L)$-modules.
		\item For $i \in \N_0$ the $L$-dimension of $H^j_{\text{an}}(\cR_L(x^{-i}))$ in degrees $j=0,1,2$ is $1,2,0$ respectively.
		\item For $i \in \N$ the $L$-dimension of $H^j_{\text{an}}(\cR_L(x^{i}))$ in degrees $j=0,1,2$ is $0,1,0$ respectively.
	\end{enumerate}
\end{lem}
\begin{proof} For the corresponding results using the $\varphi_L$-version see \cite[Theorem 0.1 (resp. Theorem 4.2)]{FX12} for the first statement and
	\cite[Théorème 5.6]{colmez2016representations} for 2 and 3 under the assumption that the field contains $\Omega_L$. The dimensions were computed without this assumption by Fourquaux and Xie in degrees $0,1$ and agree with the results of Colmez. By \cite[Theorem 2.2.2]{BFFanalytic} they can be translated to the $\Psi$-version. The computation of $H^2$ over $L$ was done in Lemma \ref{lem:h2vanish}.
\end{proof}
The proof of the following Lemma is based on Liu's proof in \cite[Lemma 4.2]{liu2007cohomology}, but we need to make some adjustments since we do not know in general whether the Euler--Poincaré characteristic formula holds for analytic cohomology.
\begin{lem}
	\label{lem:liuanalytic}
	There exist an $L$-analytic $(\varphi_L,\Gamma_L)$-module $E$ of rank $d$ that is isoclinic of slope $1/d$ and a successive extension of $\cR(x^{i}),$ where $i=0,1.$ It can be chosen such that $H^1_{\text{an}}(E) \neq 0$ and $H^2_{\text{an}}(E)=0.$
\end{lem} 
\begin{proof}
	We shall construct a sequence $(E_d)_{d}$ of $L$-analytic modules of rank $d \in \N$ isoclinic of slope $1/d$ such that $H^1_{\text{an}}(E_d) \neq 0$ and $H^2_{\text{an}}(E_d)=0.$
	Clearly $E_1:=\cR(x)$ is $L$-analytic of rank $1$ and isoclinic of slope $1.$ By Lemma \ref{lem:explicitcomputation} $E_1$ satisfies $H^1_{\text{an}}(E_1)\neq 0 $ and $H^2_{\text{an}}(E_1)=0.$
	Suppose $E_d$ has been constructed. Take a non-trivial extension $E_{d+1}$ corresponding to a non-zero element in $H^1_{\text{an}}(E_d)$ and consider the exact sequence $$0 \to E_{d} \to E_{d+1}\to \cR_L \to 0.$$ Passing to the long exact cohomology sequence we obtain by the second point of Lemma \ref{lem:explicitcomputation} an exact sequence $$\dots \to H^2_{\text{an}}(E_d) \to H^2_{\text{an}}(E_{d+1})\to 0$$ and the vanishing of $H^2_{\text{an}}(E_d)$ implies the vanishing of $H^2_{\text{an}}(E_{d+1}).$
	Due to the exactness of $$\dots \to  H^1_{\text{an}}(E_{d+1})\to H^1_{\text{an}}(\cR_L) \to H^2_{\text{an}}(E_d)=0$$ we conclude that $H^1_{\text{an}}(E_{d+1})$ surjects onto a non-zero space and hence has to be non-zero. The slope of $E_{d+1}$ is $1/(d+1)$ by construction. It remains to see that $E_{d+1}$ is isoclinic. For the convenience of the reader we reproduce Liu's argument. Suppose $P \subset E_{d+1}$ is a non-zero proper subobject of slope $\mu(P)=\frac{\deg(P)}{\operatorname{rank}(P)}<1/(d+1).$ Then its rank is bounded above by ${d+1}$ and hence  $\deg(P) \leq 0$ is necessary which implies $\mu(P)\leq 0.$ Denote by $X$ the image of $P$ in $\cR_L.$ The exact sequence $$0 \to P\cap E_{d} \to P \to X \to 0$$ implies by general Harder--Narasimhan theory (cf. \cite[4.4]{pottharst2020harder}) that the corresponding slopes are given either in  ascending or descending order. Since $\mu(P) \leq 0$ and $\mu(X)\geq 0$ due to $\cR_L$ being isoclinic of slope $0$ we conclude that we have $\mu(P\cap E_{d}) \leq 0$ which together with the fact that $E_d$ is isoclinic of slope $1/d$ implies that $P\cap E_{d}=0$ holds. This in turn means $P$ is a subobject of $\cR_L$ with slope $\leq 0$ and hence isomorphic to $\cR_L.$ This contradicts the assumption that the extension $E_{d+1}$ is not split.
\end{proof}
\begin{thm}
	\label{thm:psiperfect} Assume Conjecture \ref{con:comparison} is true.
	Let $M_0$ be an $L$-analytic $(\varphi_L,\Gamma_L)$-module over $\cR_L$ and let $M:= K\hat{\otimes}_{L}M_0$. Then the complex $C_{\psi_{\mathrm{LT}}}(M)$ of $D(U,K)$-modules is perfect.
\end{thm}
\begin{proof} We abbreviate $\psi= \psi_{\mathrm{LT}}.$ By abuse of language we say $M$ is étale if $M_0$ is étale. Similarily we mean the slopes of $M_0$ when refering to the slopes of $M.$
	Since the complex is bounded it suffices to prove the statement for the cohomology groups by \cite[\href{https://stacks.math.columbia.edu/tag/066U}{Tag 066U}]{stacks-project}. More precisely Proposition \ref{prop:psiperfchar} shows that finite generation of $M^{\psi=1}$ is sufficient.
	Hence if $M$ is étale the statement follows by combining Conjecture \ref{con:comparison} with Proposition \ref{prop:psiperfchar}.
	Suppose $M$ is isoclinic and has integral slopes. Then $Mt^n_{\mathrm{LT}}$ is étale for some power $n \in \Z$ and we argue inductively via the exact sequence of complexes induced by the sequence $$0 \to Mt_{\mathrm{LT}} \to M \to M/t_{\mathrm{LT}} \to 0.$$ For negative integers we set $N = t^{-1}_{\mathrm{LT}}M$ and use the corresponding sequence for $N.$
	We obtain a short exact sequence of complexes
	$$0 \to C_{\psi}(Mt_{\mathrm{LT}}) \to C_{\psi}(M) \to C_{\psi}(M/t_{\mathrm{LT}}) \to 0.$$ Here we make implicit use of the exactness of $K \hat{\otimes}_L-$ for strict sequences of Fréchet spaces given by Lemma \ref{lem:onablebasechange} and the exactness of filtered colimits to reduce to the Fréchet case. 
	The perfectness of the rightmost term holds unconditionally by Corollary \ref{kor:modtperfect}. By induction hypothesis either the middle or the leftmost term are perfect. In both cases we conclude from \cite[\href{https://stacks.math.columbia.edu/tag/066R}{Tag 066R}]{stacks-project} that the third term is also perfect. This concludes the case of integral slopes. 
	We first show that the theorem holds for any isoclinic $M$ such that $M/(\psi-1)$ and $M(x)/(\psi-1)$ vanishes.
	Recall that in the proof of Lemma \ref{lem:liuanalytic} we produced a family of exact sequences $$0 \to E_{i}\to E_{i+1} \to \cR_L \to 0$$ such that $E_{i}$ is isoclinic of slope $1/i$ starting with $E_1 = \cR(x).$ Tensoring this sequence with $M_0$ and applying $\psi-1$ we get a commutative diagram with exact rows
	$$\begin{tikzcd}
		0 \arrow[r] & E_i\otimes_{\cR_L} M_0 \arrow[r] \arrow[d, "\psi-1"] & E_{i+1}\otimes_{\cR_L} M_0 \arrow[r] \arrow[d, "\psi-1"] & M_0 \arrow[r] \arrow[d, "\psi-1"] & 0 \\
		0 \arrow[r] & E_i\otimes_{\cR_L} M_0 \arrow[r]                     & E_{i+1}\otimes_{\cR_L} M_0 \arrow[r]                     & M_0 \arrow[r]                     & 0
	\end{tikzcd}$$
	and induction on $i$ together with the Snake Lemma shows that $(E_i \otimes_{\cR_L} M_0)/(\psi-1)$ vanishes for every $i.$
	For $i = d-1$ we obtain a surjection $(E_{d} \otimes_{\cR_L} M_0)^{\psi=1} \to (M_0)^{\psi=1}$ and $E_d \otimes_{\cR_L} M_0$ is pure of slope $c/d +1/d = (c+1)/d.$ By  Proposition \ref{prop:psiperfchar}  we are done if we can show that $M^{\psi=1}$ is finitely generated. Due to strictness of $\psi-1$ we have $M^{\psi=1} = K\hat{\otimes}_L M_0^{\psi=1}$ and it suffices to show that $M_0^{\psi=1}$ is finitely generated. Our argument thus far shows that given an isoclinic module $M_0$ of some slope $\frac{c}{d}$ with vanishing $(\psi-1)$-cokernel there exists an isoclinic module of slope $\frac{c+i}{d}$ say $X_i$ with vanishing $(\psi-1)$-cokernel such that $X_i^{\psi=1}$ surjects onto $M_0^{\psi=1}.$
	We fix $d$ and argue inductively ``in reverse'', i.e.\ start with the base case $\frac{d}{d}$ and from the above deduce the statement for $\frac{(d-i)}{d}$ for $i \in \{1,\dots,d-1\}.$ \\
	Now let $M_0$ be arbitrary isoclinic then $t_{\mathrm{LT}}^{-n}M_0$ satisfies $t_{\mathrm{LT}}^{-n}M_0/(\psi-1)=0$ by Lemma \ref{lem:psisurjektiv} and we can thus apply the preceding result and Corollary \ref{kor:modtperfect}.
	Finally assume $M_0$ is arbitrary then either $M_0$ is isoclinic or it fits into an exact sequence $0 \to N \to M_0 \to M_0/N\to 0$ with $M_0/N$ isoclinic and $\operatorname{rank}(N)< \operatorname{rank}(M).$ Since every rank $1$ module is automatically isoclinic we deduce the general statement by induction over the rank of $M.$
\end{proof}
\begin{cor} \label{cor:constanttheorem}
	In the situation of Theorem \ref{thm:psiperfect} the complex $C_{c\Psi}(M)$ is perfect as a complex of $D(\Gamma_L,K)$-modules for any $c \in L^{\times}.$ 
\end{cor}
\begin{proof}
	Apply Lemma \ref{lem:constantirrelevantpsi} and Theorem \ref{thm:psiperfect} to conclude perfectness as a complex of $D(U,K)$-modules. An application of Lemma \ref{lem:FlatAlgebraProjectiveResolution} to the cohomology groups implies perfectness as a complex of $D(\Gamma_L,K)$-modules. The cohomology groups are coadmissible because finite projective modules are coadmissible and coadmissible modules form an abelian category.
\end{proof}
\begin{rem}
	We do not know whether the corresponding statement is true for the complex of $D(\Gamma_L,L)$-modules $C_{\psi}(M_0)$ since Lemma \ref{lem:psimodt} and Lemma \ref{lem:finiteperfect} made use of the explicit description of $D(U,K)$ as a power-series ring.
\end{rem}
\section{Towards the Euler--Poincaré formula}
\label{sec:etalecomp}
In this section we discuss the Euler--Poincaré characteristic formula for the analytic Herr complex $C_{\Psi,Z}(M)$ for $L$-analytic $(\varphi_L,\Gamma_L)$-modules. From formal arguments one can deduce a variant of the formula involving the heart $\mathcal{C}(M)$  for all $L$-analytic $(\varphi_L,\Gamma_L)$-modules over $\cR_K.$ Our computations in the trianguline case are sufficient to prove the expected formula in this case. 
\begin{defn}
	Let $M$ be an $L$-analytic $(\varphi_L,\Gamma_L)$-module over $\cR_A.$ We define the \textbf{Euler--Poincaré characterstic} of $M$ as $$\chi(M):=\sum_{i \in \mathbb{N}_0} (-1)^i \operatorname{rank}_A(H^i_{\psi,Z}(M)).$$
\end{defn}
Theorem \ref{thm:perfect} asserts that $\chi(M)$ is well-defined.
Note that this formula depends on $Z$ or more precisely on the index of the group $U \subset \Gamma_L$ used to define $Z.$ We expect  $$\chi(M) = - [\Gamma_L:U] \operatorname{rank}_{\cR_A}(M).$$

When considering modules over relative Robba rings $\cR_A,$ the validity of such a formula can be checked on each fibre $z \in \operatorname{Sp}(A)$ and thus there is no harm in assuming $A = K.$ The classical methods of Herr (cf.\cite[Section 4.2]{herr1998cohomologie}) show that the heart $\mathcal{C}(M)$ of $M$ plays an integral role. The following proposition is taken from an older version of \cite{MSVW}.
\begin{prop}
	\label{prop:heartEPC}
	Let $M$  be an $L$-analytic $(\varphi_L,\Gamma_L)$-module over $\cR_K.$ Then $\mathcal{C}(M)/Z$ is finite $K$-dimensional and 
	$\chi(M) = -\operatorname{dim}_K(\mathcal{C}(M)/Z).$ 
\end{prop}
\begin{proof}
	By the exact sequence from Remark \ref{rem:exactheartsequence} it suffices to show that $M^{\Psi=1}/Z$ is finite $K$-dimensional. Let $m \in M^{\Psi=1}.$ One checks that $(0,m)$ is a $1$-cocycle for $C_{\Psi,Z}(M)$ and the coboundary condition $(0,m) = ((\Psi-1)(n),Zn)$ for some $n \in M$ implies $n \in M^{\Psi=1}$ and therefore $m \in ZM^{\Psi=1}.$ We thus have an injection $M^{\Psi=1}/Z  \hookrightarrow H^1_{\Psi,Z}(M)$ forcing the left-hand side to be finite $K$-dimensional by Theorem \ref{thm:perfect} as a subspace of a finite-dimensional space. Consider the finite filtration of $\mathcal{F}^0C_{\Psi,Z}(M):=C_{\Psi,Z}(M)$ by the complexes
	$$\mathcal{F}^1C_{\Psi,Z}(M) = [M^{\Psi=1} \xrightarrow{Z} M^{\Psi=1}],$$ 
	$$\mathcal{F}^2C_{\Psi,Z}(M) = [M^{\varphi_L=1} \xrightarrow{Z} M^{\varphi_L=1}]$$ concentrated in degrees $[0,1]$ and $\mathcal{F}^3C_{\Psi,Z}(M)=0.$
	Clearly $$\mathrm{gr}^2C_{\Psi,Z}(M)  = [M^{\varphi_L=1} \xrightarrow{Z} M^{\varphi_L=1}]$$ and the exact sequence from Remark \ref{rem:exactheartsequence} shows $$\mathrm{gr}^1C_{\Psi,Z}(M) =[ \mathcal{C}(M) \xrightarrow{Z} \mathcal{C}(M)].$$ An argument analogous to \cite[Lemme 4.2]{herr1998cohomologie} shows that $\mathrm{gr}^0C_{\Psi,Z}(M)$ is quasi-isomorphic to the complex $$M/(\Psi-1) \xrightarrow{Z}M/(\Psi-1)$$ concentrated in degrees $[1,2].$
	From the associated convergent spectral sequence (cf. \cite[\href{https://stacks.math.columbia.edu/tag/012W}{Tag 012W}]{stacks-project}) $$E_1^{p,q} = H^{p+q}(gr^p(C_{\Psi,Z}(M))) \implies H^{p+q}(C_{\Psi,Z}(M))$$ we conclude, using that all terms on the first page are finite-dimensional, $$\chi(M) = \sum_{p,q} (-1)^{p+q} \operatorname{dim}_KH^{p+q}(gr^p(C_{\Psi,Z}(M))).$$ 
	By remark \ref{rem:finiteperfectapplied} the terms  for $\operatorname{gr}^2C_{\Psi,Z}(M)$ and $\operatorname{gr}^0C_{\Psi,Z}(M)$ cancel out. Because $\mathcal{C}(M)$ is a subspace of  $\operatorname{ker}(\Psi),$ the bijectivity of $Z$ on the kernel of $\Psi$ implies that the only remaining term is $-\operatorname{dim}_K(\mathcal{C}(M)/Z).$
\end{proof}
The expected Euler--Poincaré characteristic formula holds in the trianguline case.
\begin{rem} \label{rem:epctriangulin}
	Let $M$ be a trianguline $L$-analytic $(\varphi_L,\Gamma_L)$-module over $\cR_K,$ then $$\chi(M)= -[\Gamma_L:U]\operatorname{rank}_{\cR_K}(M).$$
\end{rem}
\begin{proof}
	By induction it suffices to treat the case where $M=\cR_K(\delta)$ for some locally $L$-analytic character  $\delta\colon L^{\times} \to K^{\times}.$ In this case by Proposition \ref{prop:rank1perf} $M^{\Psi=1}$ is finitely generated and hence $\mathcal{C}(M)$ is projective by Remark \ref{rem:Prüfer}. The proof of Proposition \ref{prop:rank1perf} shows further that $\operatorname{rank}_{D(U,K)}(\mathcal{C}(M))$ is precisely $[\Gamma_L:U].$ Since $D(U,K)$ is a domain with maximal ideal $(Z),$ the rank of $\mathcal{C}(M)$ is equal to $\operatorname{dim}_K(\mathcal{C}(M)/Z).$ The result now follows from Proposition \ref{prop:heartEPC}.
\end{proof}
Proposition \ref{prop:heartEPC} shows that the Euler--Poincaré formula would follow from $\mathcal{C}(M)$ being projective as a $D(U,K)$-module of rank $[\Gamma_L:U]\operatorname{rank}_{\cR_L}M.$ In the cyclotomic case one uses slope theory to reduce to the étale case where it follows from corresponding results for $(\varphi,\Gamma)$-modules over $\mathbf{A}_{\Q_p}$ (cf. \cite[V.1.13, V.1.18]{colmez2010representations}). We run into various problems because our theory requires the passage to $K$ in order to be able to define the cohomology groups in the first place and some key structural results like \cite[Theorem 2.4.5]{RustamFiniteness} do not have an analogue over the base field $L.$ Furthermore the dimension of $o_L\llbracket U\rrbracket$ poses a problem when working with an étale module $M_{\mathrm{LT}}$ over $\mathbf{B}_L.$ The projectivity of $(1-\frac{\pi_L}{q}\varphi_L)M_{\mathrm{LT}}^{\psi_{\mathrm{LT}}=1}$ is unknown to us and does not follow from a reflexivity argument like in \cite[I.5.2]{colmez2010representations}, because $o_L\llbracket U\rrbracket$ can be of dimension greater than two. As a closing remark we would like to point out that a stronger form of Conjecture \ref{con:comparison} requiring the natural map to be an isomorphism is not expected if $L\neq \Q_p.$ This was pointed out to us by David Loeffler and Muhammad Manji. Consider a character $\rho \colon G_L \to o_L^\times.$ One can show, using techniques from \cite{nekovavr2006selmer}, that $H^i_{\mathrm{Iw}}(L_\infty/L,o_L(\rho))$ is of $o_L\llbracket U\rrbracket$-rank $[\Gamma_L:U][L:\QQ_p].$ In contrast to this, Proposition \ref{prop:rank1perf} tells us that for $M:= K\hat{\otimes}_L\mathbb{D}^{\dagger}_{rig}(o_L(\rho))$ the $D(\Gamma_L,K)$-rank of $M^{\psi_{\mathrm{LT}}=1}$ is $[\Gamma_L:U].$ In particular the natural map can not be an isomorphism unless $L=\QQ_p.$
\bibliographystyle{amsalpha}
\let\stdthebibliography\thebibliography
\let\stdendthebibliography\endthebibliography
\renewenvironment*{thebibliography}[1]{%
	\stdthebibliography{MSVW24}}
{\stdendthebibliography}
\bibliography{Literatur}
\end{document}